\documentclass[a4paper,12pt,leqno]{article}
\usepackage{amsmath}
\usepackage{amsthm}
\usepackage{amssymb}
\usepackage{amscd}
\usepackage{graphicx, color}
\usepackage[all]{xy}

\newtheorem{thm}{Theorem}[section]
\newtheorem{dfn}[thm]{Definition}
\newtheorem{lmm}[thm]{Lemma}%[section]
\newtheorem{crl}[thm]{Corollary}%[section]
\newtheorem{rmk}[thm]{Remark}%[section]
\numberwithin{equation}{section}

\usepackage[all]{xy}

\newcommand{\Hol}{\mbox{{\rm Hol}}}

\newcommand{\Q}{{\rm Q}}
\newcommand{\Z}{\Bbb Z}
\newcommand{\C}{\Bbb C}
\newcommand{\R}{\Bbb R}
\newcommand{\K}{\Bbb K}
\newcommand{\G}{\Bbb G}
\newcommand{\F}{\Bbb F}
\newcommand{\N}{\Bbb N}
\renewcommand{\P}{{\rm P}}
\newcommand{\SP}{\mbox{{\rm SP}}}

\newcommand{\RP}{\Bbb R\mbox{{\rm P}}}

\newcommand{\Map}{\mbox{{\rm Map}}}

\newcommand{\CP}{\mathbb{C}\mathbb{P}}

\newcommand{\dis}{\displaystyle}
\newcommand{\p}{\prime}

\newcommand{\I}{\mbox{{\rm (i)}}}
\newcommand{\II}{\mbox{{\rm (ii)}}}
\newcommand{\III}{\mbox{{\rm (iii)}}}
\newcommand{\IV}{\mbox{{\rm (iv)}}}
\newcommand{\V}{\mbox{{\rm (v)}}}

\newcommand{\XS}{X_{\Sigma}}

\newcommand{\GS}{G_{\Sigma, \C}}

\newcommand{\GSK}{G_{\Sigma, \K}}
\newcommand{\n}{\textbf{\textit{n}}}
\newcommand{\dmin}{d_{\rm min}}
\newcommand{\rmin}{r_{\rm min}}
\newcommand{\KS}{\mathcal{K}_{\Sigma}}
\newcommand{\T}{\Bbb T}

\newcommand{\SZ}{{\mathcal{X}}^{D}}

\newcommand{\po}{\mbox{{\rm Poly}}}

\newcommand{\e}{\textbf{\textit{e}}}
\newcommand{\nn}{\textbf{\textit{n}}}

\newcommand{\QQ}{\mathcal{Q}^{\Sigma}}
\title{\bf
Spaces of non-resultant systems of real bounded multiplicity 
determined by
a toric variety
}
\author{Andrzej Kozlowski\footnote{%
Institute of Applied Mathematics and Mechanics,
University of Warsaw, Banacha 2, 02-097 Warsaw, Poland
(E-mail: akoz@mimuw.edu.pl)
}
\  and \ 
Kohhei Yamaguchi\footnote{%
Department of Mathematics,
University of Electro-Communications,  Chofu, Tokyo 182-8585, Japan
(E-mail: kohhei@im.uec.ac.jp)
}}

\date{}
\begin{document}
\maketitle

%%%(Abstract)%%%%%%%%
\begin{abstract}
%%%%%%%%%%%%%%%%%%
%For a fan $\Sigma$ in $\R^m$, let $\XS$ denote the toric variety
%associated to the fan $\Sigma$.
%%
%In \cite{KY12}
For each field $\F$ and positive integers $m,n,d$ with $(m,n)\not= (1,1)$,
Farb and Wolfson \cite{FW} defined  the certain affine variety
 $\po^{d,m}_n(\F)$ 
 as generalizations of spaces first studied by Arnold, Vassiliev, Segal and others.
As a natural generalization,
for each fan $\Sigma$ and $r$-tuple $D=(d_1,\cdots ,d_r)$ of positive integers,
the authors \cite {KY12} also defined and considered a more general space  
$\po^{D,\Sigma}_n(\F)$,
 where
 %$\XS$ denotes the toric variety (over $\C$) associated to the fan $\Sigma$ and 
 $r$ is the number of one dimensional cones in $\Sigma$.
This space can also be regarded
as a generalization of 
the space $\Hol^*_D(S^2,\XS)$ of based rational curves from the Riemann sphere $S^2$ to
the toric variety
$\XS$  
of degree $D$, where $\XS$ denotes the toric variety (over $\C$) corresponding to the fan $\Sigma$.
\par
 In this paper,
we  define a space $\Q^{D,\Sigma}_n(\F)$
($\F=\R$ or $\C$) which its real analogue and which can be viewed as  a generalization of spaces
considered by Arnold, Vassiliev and others in the context  of real singularity theory.
We prove that homotopy stability holds for
this space and compute the stability dimension
explicitly.
%%%
\end{abstract}
%%%%%

%%
%% Start line numbering here if you want
%%
%\linenumbers

%% main text
%%
%%%%(SECTION 1: Introduction)%%%%
%%%%%
%%%%(SECTION 1: Introduction)%%%%
%%%%%
\section{Introduction}\label{section 1}
%%%%%
%%%%%%
\paragraph{1.1 Historical survey.}
For a complex manifold $X$, let $\Map^*(S^2,X)=\Omega^2X$ (resp. $\Hol^*(S^2,X)$)
denote the space of all based  continuous maps
(resp. based holomorphic maps)
from the Riemann sphere $S^2$ to $X$.
The relationship between the topology of the space $\Hol^*(S^2,X)$ and that of the space
$\Omega^2X$ has played a significant role in several different areas of geometry and 
mathematical physics  (e.g. \cite{AJ}, \cite{A}).
In particular there arose the question  whether the inclusion
$\Hol^*(S^2,X)\stackrel{\subset}{\longrightarrow} \Omega^2X$ is a homotopy equivalence (or homology equivalence)
up to a certain dimension,
which we will refer to as the stability dimension.
Since G. Segal \cite{Se} studied this problem for the case $X=\CP^m$,
a number of mathematicians have investigated various closely related ones
(e.g. \cite{AKY1}, \cite{Gu2},  \cite{GKY2},  \cite{KY8}, \cite{KY9},  \cite{Mo2}, 
\cite {Mo3}, \cite{MV}).
%
%In particular, M. Guest \cite{Gu2} obtained a generalization of Segal's result to the case of compact non-singular toric varieties $X$.
%More generally, J. Mostovoy and E. Munguia-Vilanueva \cite{MV}
%generalized the result of Guest to the case of spaces of holomorphic maps from $\CP^m$ to a compact
%non-singular toric variety $X$ for $m\geq 1$ and they also improved the 
%homology stability dimension of Guest for the case $m=1$.
%The authors \cite{KY9} also generalized the result of Mosotovy-Vilanueva for
% {\it non-compact} non-singular toric varieties $X$ in the case $m=1.$
%(see  Theorem \ref{thm: I:KY9} in detail).
\par
Similar stabilization results appeared in the work of Arnold (\cite{Ar1}, \cite{Ar2}), and Vassiliev 
(\cite{Va}, \cite{Va2}) in connection with singularity theory.  
They considered spaces of polynomials without  roots of multiplicity greater than a certain  natural number. These spaces are examples of \lq\lq complement of discriminants\rq\rq\  in Vassiliev's terminology \cite{Va} (cf. \cite{KY1}).
%In fact, a part of Segal's proof in \cite{Se} was  based on an argument due to Arnold.
\par
 Inspired by these results, Farb and Wolfson \cite{FW} 
introduced a new family of spaces
$\po^{d,m}_n(\mathbb{F})$,
 which is defined for every field  $\F$ and integers $m,n,d\geq 1$ with $(m,n)\not=(1,1)$. 
%Inspired by their work,
The present authors generalised this further in \cite{KY12}, by considering  a fan $\Sigma$ (or toric variety)  
and
a field $\F$, 
%and $r$-tuple $D=(d_1,\cdots ,d_r)$ of positive integers,
and define a space  $\po^{D,\Sigma}_n(\F)$ as follows.
%it is defined as follows. 
%%%%
%%
%%
%%
%%(Definition 1.1)%%%
\begin{dfn}[\cite{KY12}]\label{def: poly}
%%%%%%%%%%%%%
{\rm
Let $\F$ be a field with its algebraic closure $\overline{\F}$, and
let $\Sigma$ be a fan in $\R^m$ such that 
$\Sigma (1)=\{\rho_1,\cdots ,\rho_r\}$, where $\Sigma (1)$ denotes
the set of all one dimensional cones in $\Sigma$
%and $\n_k\in \Z^m$ is the primitive generator of $\rho_k$ 
as in
(\ref{eq: one dim cone}).\footnote{%
%%%(Footnote 1)%%%%
Precise definitions and a description of the notation related to toric varieties and their fans will be 
given in  \S \ref{section: main results}.}
%%%(End of Footnote 1)%%%%%%%%
Let $\XS$ denote the toric variety over $\C$ associated to the fan $\Sigma$,
and let $\N$ denote the set of all positive integers.
%%%%%%
%%%%%%
\par\vspace{1mm}\par
%%%%%%
For each $r$-tuple $D=(d_1,\cdots ,d_r)\in \N^r$,
let $\po^{D,\Sigma}_n(\mathbb{F})$ denote the space of all
$r$-tuples $(f_1(z),\cdots ,f_r(z))\in \mathbb{F}[z]^r$ of 
$\F$-coefficients monic polynomials satisfying the following
two conditions (\ref{eq: Hol(s,XS)}a) and (\ref{eq: Hol(s,XS)}b):
\begin{enumerate}
%%%(1.1a)%%
\item[(\ref{eq: Hol(s,XS)}a)]
%%%%%%%%%%%%
$f_i(z)\in \F [z]$ is an $\mathbb{F}$-coefficients monic polynomial of the degree $d_i$
for each $1\leq i\leq r$.
%%(1.1b)%%
\item[(\ref{eq: Hol(s,XS)}b)]
%%%%%%%%%%%
For each $\sigma =\{i_1,\cdots  ,i_s\}\in I(\KS)$,
polynomials $f_{i_1}(z),\cdots ,f_{i_s}(z)$ have no common root 
$\alpha\in \overline{\F}$ of multiplicity
$\geq n$.
\par
Here,
$\KS$ denotes the underlying simplicial complex of the fan $\Sigma$ on the index set
$[r]=\{1,2,\cdots ,r\}$ defined by (\ref{eq: KS}), and
$I(\KS)$ is the set
$I(\KS)=\{\sigma\subset [r]:\sigma \not\in \KS\}$
as in (\ref{eq: IK}).
\qed
%%%%%
\end{enumerate}
%%%%
}
%%%%%%
\end{dfn}
%%End of Definition 1.1)%%%%%
%%
%%%%%%(Remark 1.2)%%%
\begin{rmk}
%%%%%%%%%%%%%%%%%%%
{\rm
(i)
By using the classical theory of resultants, one can show that 
$\po^{D,\Sigma}_n(\F)$ is an affine variety over $\F$ and that it is the complement of the set of solutions of a system of 
 polynomial equations (called a generalized resultant) with integer coefficients. 
 For this reason,
we call it
{\it the space of non-resultant systems of bounded multiplicity determined by a toric variety}.
\par
(ii)
Note that
%Note that $\po^{d,m}_n(\mathbb{F})=\po^{D,\Sigma}_n(\mathbb{F})$ 
%if $\XS =\CP^{m-1}$ and $D=(d,\cdots ,d)$.
%Thus,
%the space $\po^{D,\Sigma}_n(\F)$ may be considered
%as a
%generalization of the  space $\po^{d,m}_n(\F)$.
%Moreover, it is known that
%%(1.1)%%
\begin{equation}\label{eq: Hol(s,XS)}
\po^{D,\Sigma}_n(\C)=\Hol^*_D(S^2,\XS)
\quad
\mbox{if $n=1$ and \ $\sum_{k=1}^rd_k\n_k={\bf 0}_m,$}
\end{equation}
%%%%%%
 where 
$\Hol^*_D(S^2,\XS)$ denotes the space of based rational curves of of degree $D$  on $X$ (i.e. rational maps  of degree $D$ from the Riemann surface $S^2$
to $\XS$) (see  \cite{KY9} for further details).
Thus, the space $\po^{D,\Sigma}_n(\C)$ can be also regarded as a  generalization of the space $\Hol^*_D(S^2,\XS)$.
\qed
%%%
}
%%%%%%%%%
\end{rmk}
%%%%%%%(End of Remark 1.2)%%%%

%%%%%%%%%%
%On the other hand, the authors also considered the space $\Q^{d,m}_n(\F)$
%(for $\F=\R$ or $\C$) as
%one of the real analogues of $\po^{d,m}_n(\F)$, and
%it is defined as follows.
%%%
%%
%%%%
%%%%
%%%%
%%%%%%%%%%
%\paragraph{1.2.  Related known results.}
Now recall
the following  homotopy stability result.
% concerning to the space $\po^{D,\Sigma}_n(\C)$.
%%%%%
%%
%%
%%
%%
%%
%%%%(Theorem 1.3)%%%
\begin{thm}[\cite{KY12}]
\label{thm: main result of KY12} 
%%%%%%%%%%%%%%%%%%%%%%
Let $D=(d_1,\cdots ,d_r)\in \N^r$,
$n\geq 2$,
and let $\XS$ be an $m$ dimensional simply connected non-singular toric variety
over $\C$ 
such that 
the condition $($\ref{equ: condition dk}$)^*$ holds.
%%%
\par
%%(i)%%
$\I$
If $\sum_{k=1}^rd_k\textbf{\textit{n}}_k= {\bf 0}_m$, then
the natural map 
%(given by (\ref{eq; def iD}))
$$
i_D:\po^{D,\Sigma}_n (\C)\to 
\Omega^2_D\XS (n)\simeq
\Omega^2_0\XS (n)\simeq 
\Omega^2\mathcal{Z}_{\KS}(D^{2n},S^{2n-1})
$$ 
is a homotopy equivalence
through dimension $d_{\rm poly}(D;\Sigma ,n)$. 
\par
%%(ii)%%
$\II$
If
$\sum_{k=1}^rd_k\textbf{\textit{n}}_k\not= {\bf 0}_m$,
there is a map
$$
j_D:\po^{D,\Sigma}_n(\C) \to 
%\Omega^2_0\XS (n)\simeq
%\Omega^2_D\XS (n)\simeq 
\Omega^2\mathcal{Z}_{\KS}(D^{2n},S^{2n-1})
$$
which
is a homotopy equivalence
through dimension $d_{\rm poly}(D;\Sigma ,n)$.
\par\vspace{2mm}\par
Here, we denote by
$\lfloor x\rfloor$ the integer part of a real number $x$.
Moreover, let
$d_{\rm min}=\min\{d_1,\cdots ,d_r\}$
and
$\rmin(\Sigma)$ denote positive integers given by (\ref{eq: rmin}),
and let
$d_{\rm poly}(D;\Sigma,n)$ denote the positive integer defined by\footnote{%
%%%(Footnote 2)%%
Note that the spaces
$\XS (n)$ and $\mathcal{Z}_K(X,A)$
are 
the orbit space and the polyhedral product of a pair
$(X,A)$ given by (\ref{eq: XSigma(n)}) and
Definition \ref{dfn: polyhedral product}, respectively.
}
%%%(End of Footnote 2)%%%
%%
%%
%%
%%(1.2)%%
\begin{equation}\label{eq: dpoly}
%%%%%%%
d_{\rm poly}(D;\Sigma,n)=(2n\rmin (\Sigma)-3)\lfloor d_{\rm min}/n\rfloor -2.
\qquad
\qed
\end{equation}
%%%%%%%
%%%
%%%%
\end{thm}
%%%(End of Theorem 1.4)%%%%%
%%%%%%
%%
%%
%%%
%%%%%
\paragraph{1.2  Basic definitions.}
In this paper, we replace the space $\po^{D,\Sigma}_n(\F)$
 by its {\it real} analogue $\Q^{D,\Sigma}_n(\F)$
 for $\F=\C$ or $\R$. 
Its formal definition is below. 
%%
%%
%%(Definition 1.4)%%
\begin{dfn}\label{dfn: Pol(K)}
%%%%%%%%%%
{\rm
Let $\Sigma$ be a fan in $\R^m$ such that 
$\Sigma (1)=\{\rho_1,\cdots ,\rho_r\}$, where $\Sigma (1)$ denotes
the set of all one dimensional cones in $\Sigma$ as in Definition \ref{def: poly}.
\par
For each $r$-tuple $D=(d_1,\cdots ,d_r)\in \N^r$ and $\K=\C$ or $\R$,
let $\Q^{D,\Sigma}_n(\K)$ denote the space of all
$r$-tuples $(f_1(z),\cdots ,f_r(z))\in \mathbb{K}[z]^r$ of 
$\K$-coefficients monic polynomials satisfying the following
two conditions (\ref{eq: dpoly}a) and (\ref{eq: dpoly}b):
\begin{enumerate}
%%%(1.2a)%%
\item[(\ref{eq: dpoly}a)]
%%%%%%%%%%%%
For each $1\leq i\leq r$,
$f_i(z)\in \K [z]$ is an $\mathbb{K}$-coefficients monic polynomial of the degree $d_i$.
%%(1.2b)%%
\item[(\ref{eq: dpoly}b)]
%%%%%%%%%%%
For each $\sigma =\{i_1,\cdots  ,i_s\}\in I(\KS)$,
polynomials $f_{i_1}(z),\cdots ,f_{i_s}(z)$ have no common {\it real} root 
$\alpha\in \R$ of multiplicity
$\geq n$ (but may have a common root $\alpha \in \C\setminus\R$ of any multiplicity).
\qed
%%%%%
\end{enumerate}
%%%%
}
%%%%%%
%%%%
\end{dfn}
%%%%%%(End of Definition 1.4)%%%
%%
 Note that 
% the above two conditions (\ref{eq: dpoly}.a) and (\ref{eq: dpoly}.b)
%may be regarded as a real analogue of the conditions 
%(\ref{eq: Hol(s,XS)}.a) and (\ref{eq: Hol(s,XS)}.b) and that 
the following inclusion holds:
%%(1.3)%%
\begin{equation}\label{1.3}
%%%%
%\begin{cases} 
%%%%
%\Q^{D,\Sigma}_n(\K)&=\Q^{d,m}_n(\K)
%\quad
%\mbox{ if }\XS=\CP^{m-1}\ \mbox{and} \
%D=D_m(d),
%\\
\po^{D,\Sigma}_n(\K)\subset \Q^{D,\Sigma}_n(\K)
\quad
\mbox{ for }\K=\R\mbox{ or }\C.
%\end{cases}
\end{equation}
%%
%Note also that
%the case $\XS=\CP^{m-1}$ was already studied well in \cite{KY10}.\footnote{%
%%%%(Footnote 3)%%%%
%When $\XS=\CP^{m-1}$ and $D=D_m(d)=(d,d,\cdots ,d)\in \N^m$,
%it is written as
%$\Q^{d,m}_n(\K)=\Q^{D,\Sigma}_n(\K)$ in \cite{KY10}.
%%%%
%}
%%%(End of Footnote 3)%%
%%%(1.4)%%
%\begin{equation}\label{1.4}
%D_m(d)=(d,d,\cdots ,d)\quad
%(\mbox{$m$-times}).
%\end{equation}
Recall that  the space $\Q^{D,\Sigma}_n(\C)$ was already investigated for the case 
$n=1$ in \cite{KY11},\footnote{%
%%%(Footnote 3)%%
It is written as ${\rm Pol}^*_D(S^1,\XS)=\Q^{D,\Sigma}_n(\C)$ if $n=1$
in \cite{KY11}.
}
%%(End of Footnote 3)%%%
%%%%%%%%%%%%%%%
 and that the space $\Q^{D,\Sigma}_n(\K)$ was already extensively studied 
in \cite{KY10} for the
the case $(\XS,D)=(\CP^{m-1},D_m(d))$,\footnote{%
%%%(Footnote 4)%%
It is written as $\Q^{d,m}_n(\K)=\Q^{D,\Sigma}_n(\K)$ in \cite{KY10} for
$(\XS,D)=(\CP^{m-1},D_m(d))$.
}
%%%(End of Footnote 4)%%%% 
where $D_m(d)\in \N^m$ denotes the $m$-tuple of positive
integers defined by
%%%(1.4)%%
\begin{equation}\label{1.4}
D_m(d)=(d,d,\cdots ,d)\quad
(\mbox{$m$-times).}
\end{equation}
%%%%%%%
%\par\vspace{1mm}
\paragraph{1.3 The main results.}
In this paper we will study the homotopy type of the space
$\Q^{D,\Sigma}_n(\K)$ for $\K=\C$  or $\R$.
In particular, we will show that 
Atiyah-Jones-Segal type  homotopy stability holds for the space
$\Q^{D,\Sigma}_n(\K)$.
\par\vspace{1mm}\par
In our  result we will need 
%to consider 
the following two conditions
$($\ref{1.4}$)^*$ and $($\ref{1.4}$)^{\dagger}$:\footnote{%
%%%%%%%%%%%%%%%%%%%%%%
%%%(FootNote 5)%%%%%%%%%%%%
%%%%%%%%%%%%%%%%%%%%%
%The numbers $d_{\rm min}$ and $\rmin (\Sigma)$ are defined 
%in Definition \ref{dfn: rnumber}.
If the condition $($\ref{1.4}$)^*$ (resp. $($\ref{1.4}$)^{\dagger}$) is satisfied, 
the space
$\Q^{D,\Sigma}_n(\C)$ 
(resp. $\Q^{D,\Sigma}_n(\R)$) is simply connected
(see Corollary \ref{crl: 1-connected}).
%Similarly, 
%if the condition $($\ref{1.4}$)^{\dagger}$ is satisfied, the space
%$\Q^{D,\Sigma}_n(\R)$ is simply connected (see
Moreover,
if the condition (\ref{1.4}$)^*$ or (\ref{1.4}$)^{\dagger}$
is satisfied, the condition  $\lfloor \dmin/n\rfloor \geq 1$ holds.
Thus,  
$d(D;\Sigma,n,\K)\geq 1$
and
the main results (Theorem \ref{thm: I S1},  Corollary \ref{crl: cor S1})
 are not vacuous. 
 Note that the condition $($\ref{1.4}$)^*$ holds if the condition  
$($\ref{1.4}$)^{\dagger}$ is satisfied.
%%%
}
%%%%(End of FootNote 5)%%
%%
%%
%%(The condition (1.4*)%%%%%
\begin{enumerate}
\item[$($\ref{1.4}$)^*$]
$d_{\rm min}\geq n\geq 1$.
%\ or \ $n=1$ and $(\dmin,\rmin (\Sigma))\not=(1,2).$
%%%%%%%%(1.4**)%%%%%
\item[$($\ref{1.4}$)^{\dagger}$]
$\dmin \geq n\geq 1$ and $(n,\rmin(\Sigma))\not=(1,2)$.
%One of the following three conditions 
%holds:
%%%%%%%%%
%\begin{enumerate}
%\item[(\ref{1.4}a$)^{\dagger}$]
%%%%%
%$\dmin \geq n\geq 2$.
%%%%%%%
%\item[(\ref{1.4}b$)^{\dagger}$]
% $n=1$, $d_{\rm min}\geq 2$
% and $\rmin (\Sigma)\geq 4$.
%\item[(\ref{1.4}c$)^{\dagger}$]
% $n=d_{\rm min}=1$ and $\rmin (\Sigma)\geq 5$. 
% \end{enumerate}
\end{enumerate}
%%%%
%%%%
%%%%
%%%
%%%
Let $d(D;\Sigma,n,\K)$ denote the positive integer defined by
%%(1.5)%%
\begin{equation}\label{eq: number d(D)}
d(D;\Sigma,n,\K)
=
\begin{cases}
(2n\rmin (\Sigma)  -2)\lfloor d_{\rm min}/n\rfloor -2
&
\mbox{ if }\K=\C,
\\
(n\rmin (\Sigma)  -2)\lfloor d_{\rm min}/n\rfloor -2
&
\mbox{ if }\K=\R.
\end{cases}
\end{equation}
%Indeed,
%\par
Then we can state the main result of this article  as follows.
%%
%%
%%(Theorem 1.5)%%%
%%%%%%%%%%%%%%%%%%%%%%%%%%%%%%%%%%%%%%%
\begin{thm}
[Theorems \ref{thm: I} and \ref{thm: I R}] 
%and Corollary \ref{crl: cor S1}]
\label{thm: I S1} 
%%%%%%%%%%%%%%%%%%%%%%
%%%%%
Let $n\in \N$, let
$D=(d_1,\cdots ,d_r)\in \N^r$,
%be $r$-tuple of positive integers such that $\dmin \geq n$, 
and
let $\XS$ be an $m$ dimensional simply connected non-singular toric variety 
satisfying
the  condition $($\ref{equ: condition dk}$)^*$.
\par
$\I$
If the condition $($\ref{1.4}$)^{*}$ is satisfied,
%%%(i)%%
%$\I$ If
%$\sum_{k=1}^rd_k\textbf{\textit{n}}_k={\bf 0}_m$, then 
the  map $($given by $($\ref{eq: def. jn}$)$ and $($\ref{eq: def of jDnC}$))$
$$
j_{D,n,\C}:\Q^{D,\Sigma}_n(\C) \to 
\Omega \mathcal{Z}_{\KS}(D^{2n},S^{2n-1})
$$ 
is a homotopy equivalence
through dimension $d(D;\Sigma ,n,\C)$. 
\par
$\II$
If the condition $($\ref{1.4}$)^{\dagger}$ is satisfied,
the  map (given by $($\ref{eq: the map jDR}$)$ and
$($\ref{eq: def of jDnR}$))$
$$
j_{D,n,\R}:\Q^{D,\Sigma}_n(\R) \to 
\Omega \mathcal{Z}_{\KS}(D^{n},S^{n-1})
$$ 
is a homotopy equivalence
through dimension $d(D;\Sigma ,n,\R)$.
%\par
%$\III$
%Moreover, if
%the condition $($\ref{1.4}$)^{\dagger}$ is satisfied,
%the map
%$$
%j_{D,n,\C}:\Q^{D,\Sigma}_n(\C) \to 
%\Omega \mathcal{Z}_{\KS}(D^{2n},S^{2n-1})
%$$ 
%is a $\Z_2$-equivariant homotopy equivalence
%through dimension $d(D;\Sigma ,n,\R)$, where
%$\Z_2=\{\pm 1\}$ denotes the multiplicative cyclic group of order $2$ and
%we may regard the spaces
%$\Q^{D,\Sigma}_n(\C)$ and $\mathcal{Z}_{\KS}(D^{2n},S^{2n-1})$
%as $\Z_2$-spaces whose actions induced from the complex conjugation of
%$\C$. 
%%%%
\end{thm}
%%%%%%%%(End of Theorem 1.5)%%
%%
%%
%%
%%%(Theorem 1.6)%%%
%%%%%%%%%%%%%%%%%%%%%%%%%%%%%%%%%%%%%%%%
%\begin{thm}[Theorem \ref{thm: I R}; the case $\K=\R$]\label{thm: IR S1} 
%%%%%%%%%%%%%%%%%%%%%%%
%%%%%%
%Let $n\in \N$, let
%$D=(d_1,\cdots ,d_r)\in \N^r$,
%%be $r$-tuple of positive integers such that $\dmin \geq n$, 
%and
%let $\XS$ be an $m$ dimensional simply connected non-singular toric variety 
%satisfying
%the two conditions $($\ref{equ: condition dk}$)^*$ and $($\ref{1.4}$)^{\dagger}$.
%\par
%%%%(i)%%
%%$\I$ If
%%$\sum_{k=1}^rd_k\textbf{\textit{n}}_k={\bf 0}_m$, then 
%Then
%the  map (given by $($\ref{eq: the map jDR}$)$ and
%$($\ref{eq: def of jDnR}$))$
%$$
%j_{D,n,\R}:\Q^{D,\Sigma}_n(\R) \to 
%\Omega \mathcal{Z}_{\KS}(D^{n},S^{n-1})
%$$ 
%is a homotopy equivalence
%through dimension $d(D;\Sigma ,n,\R)$.
%%where
%%$\lfloor x\rfloor$ denotes the integer part of a real number $x$,
%%$\rmin (\Sigma)$ is the constant positive integer given by
%%(\ref{eq: rmin}),
%%%$\dmin =\min\{d_1,\cdots ,d_r\}$, and
%%the positive integer $d(D;\Sigma,n,\R)$ is given by
%%%%(1.6)%%
%%\begin{equation}\label{eq: dDSigma-R}
%%%%%%%%
%%d(D;\Sigma,n,\R)=(n\rmin (\Sigma)  -2)\lfloor d_{\rm min}/n\rfloor -2.
%%\end{equation}
%%%%
%%%%%
%\end{thm}
%%%%%%%%%(End of Theorem 1.6)%%
%%
%%
%%(Remark 1.6)%%
\begin{rmk}
%%%%%%%%%%%
{\rm
(i)
Recall that a map $g:V\to W$ is called 
{\it a homology $($resp. homotopy$)$
equivalence through dimension }$N$ 
if the induced homomorphism
$
g_*:H_k(V;\Z)\to H_k(W;\Z)$
(resp.
$g_*:\pi_k(V)\to\pi_k(W))$
is an isomorphism for all $k\leq N$.
%%%
\par
(ii)
Similarly, when $G$ is a topological group and
a map $g:V\to W$ is a $G$-equivariant map between $G$-spaces $V$ and $W$,
the map $g$ is called a
{\it  $G$-equivariant homology $($resp. 
$G$-equivariant homology homotopy$)$
equivalence through dimension }$N$
if the restriction $g^H=g\vert V^H:V^H\to W^H$ is a
homology (resp. homotopy) equivalence through dimension $N$
for any subgroup $H\subset G$.
Here, for each $G$-space $X$ and a subgroup $H\subset G$,
let $X^H$ denote the $H$-fixed subspace of $X$ defined by
%%(1.6)%%
\begin{equation}\label{eq: XH}
X^H=\{x\in X:h\cdot x=x
\quad
\mbox{ for any }h\in H\}.
\qquad
\qed
\end{equation}
%%%
}
%%%%%%%%
\end{rmk}
%%(End of Remark 1.6)%%%
%\par\vspace{2mm}\par
%We already investigated well for the homotopy type of $\po^{D,\Sigma}_n(\K)$ 
%for the case $\K=\C$ in
%\cite{KY12}, and
%we would like to study the homotopy type of $\po^{D,\Sigma}(\K)$ for the case
%$\K=\R$.
%Recently we could investigate it for the case $\XS=\CP^m$ in \cite{KY13}, and
%we would like to investigate about the space $\po^{D,\Sigma}_n(\R)$ for
%a general toric variety $\XS$.
%In the subsequent paper, we shall investigate this problem by using the results of 
%this article.

%%
%%
%%
%%
\paragraph{1.4 Organization.}
This paper is organized as follows.
%%(SECTION 2)%%
In \S \ref{section: main results} we recall the basic definitions and facts which is needed for the statements of the results of this article.
After then precise statements of the main results
(Theorems \ref{thm: I}, \ref{thm: I R}, and Corollary \ref{crl: I}) 
are
stated.
%%(SECTION 3)%%
In \S \ref{section: polyhedral products} we recall several basic facts related to polyhedral products and toric varieties.
%%(SECTION 4)%%%
In \S \ref{section: simplicial resolution}, 
we summarize  the definition of the non-degenerate simplicial resolution, and
we construct the Vassiliev spectral sequence.
% converging to  $H_*(\Q^{D,\Sigma}_n(\K);\Z).$ 
%by using it.
%%(SECTION 5, 6)%%
In \S \ref{section: stabilization map} we define the stabilization maps, and
in \S \ref{section: homology stability}, we construct the truncated spectral sequence
induced from the spectral  sequence obtained in \S \ref{section: simplicial resolution}. 
By using this truncated spectral sequence, we  shall prove the homology stability result 
(Theorems \ref{thm: III}, \ref{thm: III (KY16)}, and Corollary \ref{crl: III*}).
%\par
%%(SECTION 7)%%
In \S \ref{section: connectivity}
we investigate about the connectivity of the space $\Q^{D,\Sigma}_n(\K).$
In particular,
we  prove that the space $\Q^{D,\Sigma}_n(\C)$ 
(resp. $\Q^{D,\Sigma}_n(\R)$) is simply connected
if the condition $($\ref{1.4}$)^*$ 
(resp. $($\ref{1.4}$)^{\dagger}$) is satisfied. 
%and that
%the space $\Q^{D,\Sigma}_n(\R)$ is simply connected
%if the condition $($\ref{1.4}$)^{\dagger}$ is satisfied. 
%%(SECTION 8)%%
In \S \ref{section: scanning maps}  
we consider the configuration model 
for the space $\Q^{D,\Sigma}_n(\K)$ and
recall the stabilized horizontal scanning map
(see Theorem \ref{thm: scanning map}).
%%(SECTION 9)%%%
In \S \ref{section: stability} we prove the stability result
(Theorem \ref{thm: V}), and in \S \ref{section: proofs of main results}
we give the proofs of the main results
(Theorems \ref{thm: I}, \ref{thm: I R}, and Corollary \ref{crl: I}) by using it. 
%and
% we consider some examples of the main results.
%we consider two explicit examples of the main results. 

%%%%%%%%%%%%%%%%%%%%%%%%%%%%%%

%%
%%
%%
%%%%%%%%%%%%%%%%%%%%%%%%%%%%%%
%%(SECTION 2)%%%
%%(SECTION 2)%%%
%%(SECTION 2)%%%
%%(SECTION 2)%%%
%%%%%%%%%%%%%%%%%%%%%%%%%%%%%%%%%
%%%(Section 2: Basic related facts and the main results)%%%
\section{Toric varieties and the main results}
\label{section: main results}
%%%%%%%%%%%%%%%%%%%%%%%%%%%%%%%%%%%
%%%
%%%
%%%
%%%
%%%
%%%
In this section we recall several basic definitions and facts related to toric varieties
(convex rational polyhedral cones, toric varieties, a fan of toric variety,
polyhedral products, homogenous coordinate, rational curves on a toric variety etc).
Then by using these definitions and notations we
give precise statements of the main results of this paper.
%\par
From now on,
we always assume that  $\K=\C$ or $\R$.
Moreover, if $\dmin <n$, $\lfloor \dmin/n\rfloor =0$ and
$d(D;\Sigma,n,\K)=-2<0$.
So we also assume that $\dmin \geq n\geq 1$.
%%%%%%%%%%%%%%%%%%%%%%%%%%%%%%%%%%%
\paragraph{2.1 Fans, toric varieties and Polyhedral products.}
%%%%%%%%%%%%%%%%%%%%%%%%%%%%%%%%%%%
%%%
%%%
%%%
%%%%%%%%%%%%%%%%%%%%%%%%%%%%%%%%%%%%%%%%%%%%%%%%%%%%%%%%%%
{\it A convex rational polyhedral cone}
in $\R^m$
is a subset of $\R^m$ of the form
%%%%%%%%
%%%(2.1)%%
\begin{equation}\label{eq: cone}
%%%%%%
\sigma = \mbox{Cone}(S)=
\mbox{Cone}(\textit{\textbf{m}}_1,\cdots,\textit{\textbf{m}}_s)
%=
%\sum_{k=1}^s\R_{\geq 0}\cdot \textbf{\textit{m}}_k
=
\Bigg\{\sum_{k=1}^s\lambda_k\textit{\textbf{m}}_k:\lambda_k\geq 0
%\mbox{ for any }1\leq k\leq s\}
\Bigg\}
\end{equation}
%%%
for a finite set $S=\{\textit{\textbf{m}}_1, \cdots ,
\textit{\textbf{m}}_s\} \subset\Z^m$.
The dimension of $\sigma$ is the dimension of
the smallest subspace of $\R^m$ which contains $\sigma$.
A convex rational polyhedral cone $\sigma$
is called
{\it  strongly convex} if
$\sigma \cap (-\sigma)=\{{\bf 0}_m\}$,
where
we set ${\bf 0}_m={\bf 0}=(0,0,\cdots, 0)\in \R^m.$
%%
%%%
{\it A face} $\tau$ of 
a convex rational polyhedral cone
$\sigma$ is a subset $\tau\subset \sigma$ of the form
$\tau =\sigma\cap \{\textit{\textbf{x}}\in\R^m:L(\textbf{\textit{x}})=0\}$ 
for some linear form $L$ on $\R^m$, such that
$\sigma \subset \{\textit{\textbf{x}}\in\R^m:
L(\textit{\textbf{x}})\geq 0\}.$
%\par
%If we set 
%$\{k:1\leq k\leq s, L(\textit{\textbf{m}}_k)=0\}=\{i_1,\cdots ,i_t\},$
%we easily see that $\tau =\mbox{Cone}(\textit{\textbf{m}}_{i_1},
%\cdots , \textit{\textbf{m}}_{i_t})$.
Note that if $\sigma$ is a strongly convex rational polyhedral cone,
so is any of its faces.\footnote{%
%%%(Footnote 6)%%%
When $S$ is the emptyset $\emptyset$,
we set $\mbox{Cone}(\emptyset)=\{{\bf 0}_m\}$
and we may also regard it as one of strongly convex rational polyhedral cones
in $\R^m$.
}
%%(End of Footnote 6)%%%%%%
%%(Definition 2.1)%%
\begin{dfn}
%%%%%%%%%%%%
{\rm
Let $\Sigma$ be a finite collection of strongly convex rational polyhedral cones
in $\R^m$.
\par
(i)
The set $\Sigma$ 
is called {\it a fan} (in $\R^m$) if
the following two conditions hold:
%(\ref{eq: cone}a) and (\ref{eq: cone}b) 
%(\ref{eq: cone}.1) and (\ref{eq: cone}.2) 
%are satisfied:
%%()%%%
\begin{enumerate}
%%%%(2.1a)%%%%%%
\item[(\ref{eq: cone}a)]
Every face $\tau$ of $\sigma\in \Sigma$ belongs to $\Sigma$.
%%%%(2.1b)%%%%
\item[(\ref{eq: cone}b)]
%%%%
If $\sigma_1,\sigma_2\in \Sigma$, $\sigma_1\cap \sigma_2$ is
a common face of each $\sigma_k$ and
$\sigma_1\cap\sigma_2\in \Sigma$.
\end{enumerate}
%%%%%%%%%%%%%
%%%%%
\par
(ii)
An $m$ dimensional irreducible normal  variety
$X$ (over $\C$) is called {\it a toric variety}
if it has a Zariski open subset
 $\T^m_{\C}=(\C^*)^m$ and the action of $\T^m_{\C}$ on itself
extends to an action of $\T^m_{\C}$ on $X$.
\par
The most significant property of a toric variety
is that it is characterized up to isomorphism entirely by its 
associated fan 
$\Sigma$. 
We denote by $\XS$ the toric variety associated to a fan $\Sigma$
(see \cite{CLS} for the details).
\par
(iii)
Let $K$ be some set of subsets of $[r]$.
Then the set $K$ is called {\it an abstract simplicial complex} on the index set $[r]$ 
 if the following condition $(\dagger)_K$ holds:
 \begin{enumerate}
 \item[$(\dagger)_K$]
 \quad
 $\tau \subset \sigma$ and $\sigma\in K$, then $\tau\in K$.
 \end{enumerate}
}
\end{dfn}
%%%(End of Definition 2.1)%%
%%%%
%%%%%%%%%%%%%%%%%%%%%%%%%%%%%%%%%%%%%%%%%%%%%%%%%%%%%%%%%%%%%%
%\paragraph{2.2 Polyhedral products.}
%%%%%%%%%%%%%%%%%%%%%%%%%%%%%%%%%%%%%%%%%%%%%%%%%%%%%%%%%%%%%
%Now recall the basic definitions concerning polyhedral products and related spaces.
%%%
%%%
%%(Remark 2.2)%%
\begin{rmk}
%%%
{\rm
(i) It is well known that
there are no holomorphic maps
$\CP^1=S^2\to \T^m_{\C}$ except the constant maps, and that
the fan $\Sigma$ of $\T^m_{\C}$ is $\Sigma =\{{\bf 0}_m\}$.
%${\bf 0}_m=(0,0,\cdots ,0,0)\in \C^m.$
Hence, without loss of generality
we always assume that $\XS\not=\T^m_{\C}$, and that
any fan $\Sigma$ in $\R^m$
satisfies the condition
%%%()%%
%\begin{equation}\label{eq: assumption0}
$\{{\bf 0}_m\}\subsetneqq \Sigma .$
%\end{equation}
\par
(ii)
In this paper by a simplicial complex $K$ we always mean  \textit{an abstract simplicial complex}, 
and we always assume that a simplicial complex $K$  contains the empty set 
$\emptyset$.
\qed
}
%%%%%%%%%%%%%%
\end{rmk}
%%(End of Remark 2.2)%%
%%%
%%%%%
%%%
%%%
%%%
%%%(Definition 2.3)%%%
\begin{dfn}\label{dfn: polyhedral product}
%%%%%%%%%%%%%%%%%%%%%
{\rm
Let $K$ be a simplicial complex on the index set $[r]=\{1,2,\cdots ,r\}$,
%%%%%%%%%%%%%%%%%%
%%%%(Footnote 7)%%%%%%
%\footnote{%
%Let $K$ be some set of subsets of $[r]$.
%Then the set $K$ is called {\it an abstract simplicial complex} on the index set $[r]$ if
%the following condition holds:
% if $\tau \subset \sigma$ and $\sigma\in K$, then $\tau\in K$.
%In this paper by a simplicial complex $K$ we always mean  \textit{an abstract simplicial complex}, 
%and we always assume that a simplicial complex $K$  contains the empty set 
%$\emptyset$.
%}
%%%%%%%%%%%%
%%%%(End of FootNote 7)%%%%
%%%%%
and let $(X,A)$
%$(\underline{X},\underline{A})=\{(X_1,A_1),\cdots ,(X_r,A_r)\}$ 
be a  pairs of based spaces.
%such that $A\subset X$.
\par\vspace{1mm}\par
(i) Let $I(K)$ denote the collection of subsets $\sigma\subset [r]$ defined by
%%(2.2)%%
\begin{equation}\label{eq: IK}
I(K)=\{\sigma \subset [r]:\sigma\notin K\}.
\end{equation}
\par
(ii)
Define
 {\it the polyhedral product} $\mathcal{Z}_K
(X,A)$ 
%of an $r$-tuple of pairs of spaces
%$(\underline{X},\underline{A})$ 
with respect to $K$
by
%%%%(2.3)%%%%%%%%
\begin{align}\label{eq: polyhedral p}
%%%%%%%%%%%%%%%%
\mathcal{Z}_K(X,A)
&=
\bigcup_{\sigma\in K}(X,A)^{\sigma},
\qquad
\mbox{where}
\\
(X,A)^{\sigma}
&=
\{(x_1,\cdots ,x_r)\in X^r:
x_k\in A\mbox{ if }k\notin \sigma\}.
\nonumber
%%%%%%%%%%%
\end{align}
%%%%%%%%%%%%%%
%%
%When $(X_i,A_i)=(X,A)$ for each $1\leq i\leq r$, we write
%$\mathcal{Z}_K(X,A)$ for $\mathcal{Z}_K(\underline{X},\underline{A}).$
%%%%
\par
(iii)
For each subset $\sigma =\{i_1,\cdots ,i_s\}\subset [r]$, let
$L_{\sigma}(\K^n)$ denote the subspace of $\K^{nr}$ defined by
%%(2.4)%%
\begin{align}
%%%%%%%
L_{\sigma}(\K^n)
&=
\{(\textbf{\textit{x}}_1,\cdots ,\textbf{\textit{x}}_r)
\in (\K^n)^r=\K^{nr}: 
\textbf{\textit{x}}_{i_1}=\cdots =\textbf{\textit{x}}_{i_s}=
{\bf 0}_n\}
%\\
%&=
%\{(\textbf{\textit{x}}_1,\cdots ,\textbf{\textit{x}}_r)\in\C^{nr}:
%\textbf{\textit{x}}_i\in \C^n,\ 
%\textbf{\textit{x}}_j={\bf 0}_n \mbox{ if }j\in\sigma\}
%\nonumber
\end{align}
%%%%%%
and let $L_n^{K}(\K)$ denote the subspace of $\K^{nr}$ defined by
%%(2.5)%%
\begin{equation}\label{eq: L(Sigma)}
%%%%%
L_n^{K}(\K )=\bigcup_{\sigma\in I(K)}L_{\sigma}(\K^n)
=\bigcup_{\sigma\subset [r],\sigma\notin K}L_{\sigma}(\K^n).
%%%
\end{equation}
%%%%%
%where we set
%%%%(2.5)%%
%\begin{equation}\label{eq: IK}
%%%%
%I(K)=\{\sigma\subset [r]:\sigma\notin K\}.
%\end{equation}
%%%
Then it is easy to see that}
%%(2.6)%%
\begin{equation}\label{eq: LSigma}
%%%%%%%
\mathcal{Z}_K(\K^n,(\K^{n})^*)
=
\K^{nr}\setminus
L_n^{K}(\K),
\ \ 
\mbox{\rm where }(\K^{n})^*=\K^n\setminus \{{\bf 0}_n\}.
%%%%%%%%%%% 
\end{equation}
%%%%%%%%%%
%%%
\end{dfn}
%%%(End of Definition 2.3)%%%%%%
%%
%%
%%
%%%(Homogenous coordinates)%%%%%
\paragraph{2.2 Homogenous coordinates.}
%%%%%%%%%%%%%%%%%%%%%%%%%%%%%%%%
%%%
%%%
%%%
%%%
Next we recall the basic facts about homogenous coordinates on
toric varieties.
\par\vspace{1mm}\par
%%%
%%
%%%(Definition 2.4)%%%%%%%%
\begin{dfn}\label{dfn: fan}
%%%%%%%%%%%%%%%%%%%%%%%%%%%
{\rm
Let $\Sigma$ be a fan in $\R^m$
such that $\{{\bf 0}_m\}\subsetneqq \Sigma$, and let
%%%%%%
%(2.7)%%
\begin{equation}\label{eq: one dim cone}
%%%%%%%
\Sigma (1)=\{\rho_1,\cdots ,\rho_r\}
\end{equation}
%%%%%
denote the set of all
one dimensional cones in $\Sigma$. 
\par
%%(i)%%
(i)
For each $1\leq k\leq r$,
we denote by 
%%%
$
\textbf{\textit{n}}_k\in\Z^m
$
\textit{the primitive generator} of
$\rho_k$, such that %that is, we have
%%()%%
%\begin{equation}\label{eq: primitive}
%%%
$\rho_k \cap \Z^m=\Z_{\geq 0}\cdot \textbf{\textit{n}}_k.$
%\end{equation}
%%%%%%%%
%
Note that $\rho_k 
=\mbox{Cone}(\textit{\textbf{n}}_k).$
%%%%%%%%%%%
\par
(ii)
%%%%%%%%%%%
Let $\mathcal{K}_{\Sigma}$ denote {\it the underlying simplicial complex of}  $\Sigma$ 
%on the index set $[r]$
defined by
%%%%%
%%%(2.8)%%%%%%
\begin{equation}\label{eq: KS}
%%%%%%%%%%%%%
\KS =
\Big\{\{i_1,\cdots ,i_s\}\subset [r]:
\textbf{\textit{n}}_{i_1},\textbf{\textit{n}}_{i_2},\cdots
,\textbf{\textit{n}}_{i_s}
\mbox{ span a cone in }\Sigma\Big\}.%\cup \{\emptyset\}.
\end{equation}
%%%%%%
It is easy to see that $\KS$ is a simplicial complex on the index set $[r]$.
\par
%%(iii)%%
(iii)
Let 
$\GSK\subset \T^r_{\K}=(\K^*)^r$  be the subgroup 
%%%%%%%%
%%(2.9)%%
\begin{equation}
%%%%%%%%
\GSK =
\{(\mu_1,\cdots ,\mu_r)\in \T^r_{\K}:
\prod_{k=1}^r(\mu_k)^{\langle \textbf{\textit{n}}_k,\textbf{\textit{m}}
\rangle}=1
\mbox{ for all }\textbf{\textit{m}}\in\Z^m\},
%%%%%%%%%%%%
\end{equation}
%%%
where
$\langle \textbf{\textit{u}}, \textbf{\textit{v}}\rangle=
\sum_{k=1}^m u_kv_k$ for
$\textbf{\textit{u}}=(u_1,\cdots ,u_m)$
and  $\textbf{\textit{v}}=(v_1,\cdots ,v_m)\in\R^m$.
%%%%
%Similarly,
%let $G_{\Sigma,\R}\subset \GS$ denote the subgroup of $\GS$ defined by
%%%(2.10)%%
%\begin{align}
%G_{\Sigma,\R}&=\GS \cap \T^r_{\R}
%\\
%\nonumber
%&=
%\{(\mu_1,\cdots ,\mu_r)\in \T^r_{\R}:
%\prod_{k=1}^r(\mu_k)^{\langle \textbf{\textit{n}}_k,\textbf{\textit{m}}
%\rangle}=1
%\mbox{ for all }\textbf{\textit{m}}\in\Z^m\},
%%%
%\end{align}
%%%
%where
%$\R^*=\R\setminus\{0\}$ and
%$\T^r_{\R}=(\R^*)^r.$
%%%(v)%%%%%
\par
(iv)
There is a natural $\GSK$-action on
$\mathcal{Z}_{\mathcal{K}_{\Sigma}}(\K^n,(\K^{n})^*)$  by
coordinate-wise multiplication, 
%%%%%%%%%
%%%(2.10)%%%%%
\begin{equation}\label{eq: coordinate multiplication}
%%%%%%%%%%%%
(\mu_1,\cdots ,\mu_r)
\cdot(\textbf{\textit{x}}_1,\cdots ,\textbf{\textit{x}}_r)=
(\mu_1\textbf{\textit{x}}_1,\cdots ,\mu_r\textbf{\textit{x}}_r)
\end{equation}
%%%%%
for
$((\mu_1,\cdots ,\mu_r),
(\textbf{\textit{x}}_1,\cdots ,\textbf{\textit{x}}_r))
\in \GSK \times \mathcal{Z}_{\mathcal{K}_{\Sigma}}(\K^n,(\K^{n})^*),$
where we set
%%%(2.11)%%
\begin{equation}
%%%
\mu \textbf{\textit{x}}=
(\mu x_1,\cdots ,\mu x_n)
\quad
\mbox{if }(\mu ,
\textbf{\textit{x}})=(\mu, (x_1,\cdots ,x_n))\in \K^*\times \K^n.
%%%
\end{equation}
%%%%%%%%%%%%%
%%%
\par
(v)
Let 
$X_{\Sigma,\K}(n)$
denote
the corresponding orbit space
%%%(2.12)%%
\begin{equation}\label{eq: XSigma(n)}
%%%%%%
X_{\Sigma,\K}(n)=
\mathcal{Z}_{\KS}(\K^n,(\K^n)^*)/G_{\Sigma,\K},
\quad
\mbox{where}
\end{equation}
%%%%%%%
%where
%%(2.13)%%
\begin{equation}
q_{n,\K}:\mathcal{Z}_{\mathcal{K}_{\Sigma}}(\K^n,(\K^{n} )^*)
\to
X_{\Sigma,\K}(n)=
\mathcal{Z}_{\mathcal{K}_{\Sigma}}(\K^n,(\K^{n} )^*)/G_{\Sigma,\K}
%%%
\end{equation}
%%%%
denotes the corresponding canonical projection.
%\par
In particular, we also write
%%(2.14)%%
\begin{equation}
\XS (n)=X_{\Sigma,\C}(n) \ \mbox{ and }\ G_{\Sigma}=G_{\Sigma,\C}
\ \ \mbox{ if }\K=\C.
\end{equation}
}
%%%%%
\end{dfn}
%%%(End of Definition 2.4)%%%%%
%%%
%%%(Theorem 2.5: Theorem of Cox)
%%%
\begin{thm}[\cite{Cox1}, Theorem 2.1]\label{prp: Cox}
If the set $\{\textit{\textbf{n}}_k\}_{k=1}^r$ 
of all primitive generators  spans $\R^m$
$($i.e. $\sum_{k=1}^r\R\cdot \textbf{\textit{n}}_k=\R^m)$,
there is a natural isomorphism
%%(2.15)%%
\begin{equation}\label{equ: homogenous}
\XS \cong \mathcal{Z}_{\KS}(\C,\C^*)/\GS=\XS (1)=X_{\Sigma,\C}(1).
%%%
\end{equation}
%%%
Hence, we can identify  $\XS (n)$ with the toric variety $\XS$ if $n=1$.
\qed
\end{thm}
%%%(End of Theorem 2.5)%%%
%%%
%%%
%%%
%%%(Remark 2.6)%%%
\begin{rmk}\label{rmk: fan}
%%%%%%%%%%%%%%%%%%%%%%%%
%%%%%%%%%%%%%%%%%%
{\rm
Let $\Sigma$ be a fan in $\R^m$ 
as in Definition \ref{dfn: fan}.
%as in (\ref{eq: one dim cone}).
Then the fan $\Sigma$ is completely determined by
the pair
$(\mathcal{K}_{\Sigma},\{\textit{\textbf{n}}_k\}_{k=1}^r)$ 
(see \cite[Remark 2.3]{KY9} for the details).
\qed
%%%%
%%%%%%
}
%%%%%%%
\end{rmk}
%%%%%%%(End of Remark 2.6)%%%%%%%%
%Let $\K=\R$ or $\C$.
For each $1\leq i\leq r$,
let 
$F_i=(f_{1;i},\cdots ,f_{n;i})
\in \K [z_0,\cdots ,z_s]^n$
 be an $n$-tuple of homogenous polynomials of
the same degree $d_i$ satisfying the following condition:
%%%%(2.16*)%%
\begin{enumerate}
\item[(\ref{equ: map F}$)^*$]
For each $\sigma\in I(\KS)$,
the homogenous polynomials $\{f_{k;i}\}_{k\in \sigma}$ have no common 
{\it real} root
except ${\bf 0}_{s+1}\in \R^{s+1}.$
\end{enumerate}
%%%%%
In this situation, 
consider the map
%%(2.16)%%
\begin{equation}\label{equ: map F}
%%%
F=(F_1,\cdots,F_r):\R^{s+1}\setminus \{{\bf 0}_{s+1}\}\to(\K^n)^r=\K^{rn}
\quad
\mbox{given by}
%%%%%%
\end{equation}
%%%
%given by
%%%()%%
\begin{equation*}
%\nonumber
\begin{cases}
F(\textbf{\textit{x}})
&=
(F_1(\textbf{\textit{x}}),\cdots F_r(\textbf{\textit{x}}))
\quad
\mbox{ for }\textbf{\textit{x}}\in \R^{m+1}\setminus \{{\bf 0}_{m+1}\},
%\ \ \mbox{where}
\\
F_i(\textbf{\textit{x}})&=(f_{1;i}(\textbf{\textit{x}}),
f_{2;i}(\textbf{\textit{x}}),\cdots ,f_{n;i}(\textbf{\textit{x}}))
\quad
\mbox{for }1\leq i\leq r.
\end{cases}
\end{equation*}
%%%%%%
By  the assumption (\ref{equ: map F}$)^*$, 
homogenous polynomials
$\{f_{k;i}\}_{k\in \sigma}$
have no common real root except ${\bf 0}_{s+1}\in \R^{s+1}$
for each $1\leq i\leq r$
and $\sigma\in I(\KS)$.
Thus, we see that the image of the map $F$ is contained in
$\mathcal{Z}_{\KS}(\K^n,(\K^n)^*)$, and
we may regard
the map $F$  as the map
%%(2.17)%%
\begin{equation}\label{eq: the map F}
F=(F_1,\cdots ,F_r):\R^{s+1}\setminus \{{\bf 0}_{s+1}\}\to 
\mathcal{Z}_{\KS}(\K^n,(\K^n)^*).
\end{equation}
%%%%%%%%%
%%%
%%%
%%%
%%%%%%%%%%%%%%%%%%%%%%%%
 The following lemma, whose proof we
 postpone until the end of \S \ref{section: polyhedral products},
 plays a crucial role in the proof of the main result of this paper.
%%% 
%%%(Lemma 2.7)%%
\begin{lmm}[cf. \cite{Cox2}, Theorem 3.1; \cite{KOY1}, Lemma 2.6]
\label{lmm: Cox tric}
%%%%%%%%%%%%
Suppose that
the set $\{\textit{\textbf{n}}_k\}_{k=1}^r$ 
of all primitive generators  spans $\R^m$.
For each $1\leq i\leq r$ and $\sigma\in I(\KS)$, let
$F_i=(f_{1;i},\cdots ,f_{n;i})
\in \K [z_0,\cdots ,z_s]^n$
 be an $n$-tuple of homogenous polynomials of
the same degree $d_i$ satisfying the condition
(\ref{equ: map F}$)^*$.
\par
Then 
there is a unique map
$f:\RP^s\to X_{\Sigma,\K} (n)$ such that the diagram
%%(2.18)%%
\begin{equation}\label{equ: condition dk}
%{equ: homogenous-CD}
%%%%
\begin{CD}
%%%%%%%%%
\R^{s+1}\setminus \{{\bf 0}_{s+1}\} @>(F_1,\cdots ,F_r)>> 
\mathcal{Z}_{\KS}(\K^n,(\K^n)^*)
\\
@V{\gamma_s}VV @V{q_{n,\K}}VV
\\
\RP^s @>f>> \mathcal{Z}_{\KS}(\K^n,(\K^n)^*)/G_{\Sigma,\K}=X_{\Sigma,\K} (n)
\end{CD}
\end{equation}
%%%%%%%%
is commutative
if and only if  the  condition 
$\sum_{k=1}^rd_k\textit{\textbf{n}}_k={\bf 0}_m.$ holds.
%%%(2.21)%%
%\begin{equation}\label{equ: condition dk}
%%%%%%%%
%\sum_{k=1}^rd_k\textit{\textbf{n}}_k={\bf 0}_m.
%\end{equation}
%%%
\par
Here, $\gamma_s:\R^{s+1}\setminus \{{\bf 0}_{s+1}\}\to \RP^s$ denotes
the canonical double covering, and the map $F=(F_1,\cdots ,F_r)$ is given by
(\ref{eq: the map F}).
\end{lmm}
%%(end of Lemma 2.7)%%
%%
%%
%%
%%
%%
%%
%%%
%%%
%%%%(Assumptions)%%%%%%%
\paragraph{2.3 Assumptions.}
%%%%%%%%%%%%%%%%%%%%%%%
%%
%%
%%
%%%%%%%%%%%%%%%%%%%%%%%
From now on, let $\Sigma$ be a fan in $\R^m$
%(\ref{eq: one dim cone}),
as in Definition \ref{dfn: fan}, 
and we always assume that $\XS$ is simply connected and non-singular.
Moreover,
we shall assume  the following condition holds.
%%(2.18*)%%
\begin{enumerate}
%%%%%%%%%%
\item[$($\ref{equ: condition dk}$)^*$]
There is an $r$-tuple
$D_*=(d_1^*,\cdots ,d_r^*)\in \N^r$ such that
$\sum_{k=1}^rd_k^*\textit{\textbf{n}}_k={\bf 0}_m.$
%%%
\end{enumerate}
%%%%%%%%%%%%%%%%%%%%%%%%
%%
%%
%%
%%
%%
%%
%%%%%%%%%%%%%%%%%%%%%%%%%
%%%(Remark 2.8)%%%
\begin{rmk}\label{rmk: assumption}
%%%%%%%%%%%%%%%%%%%%%%%%%
%%
%%
%%%%%%%%%%%%%%%%%%%%%%
{\rm
(i) 
It follows from \cite[Theorem 12.1.10]{CLS} that
$\XS$ is simply connected if and only if the following condition $(\dagger\dagger)$
holds:
%%%%%%%%%%
\begin{enumerate}
%%%%%%%%%%%
\item[($\dagger\dagger$)]
%%%%%%%%%%%
The set $\{\textbf{\textit{n}}_k\}_{k=1}^r$ of all primitive generators spans
$\Z^m$ over $\Z$, i.e.
\newline
$\sum_{k=1}^r\Z\cdot \textbf{\textit{n}}_k=\Z^m$.
\end{enumerate}
%%%%%%%%%%
%%
%%
%%
%%%%%%%%%%%%%%%%%
Thus we see that if $\XS$ is simply connected then
the set $\{\textit{\textbf{n}}_k\}_{k=1}^r$ 
of all primitive generators  spans $\R^m$.
In particular,  if $\XS$ is a compact smooth toric variety then $\XS$ is simply connected
(see Lemma \ref{lmm: toric}).
\par
(ii)
We make the identification $\RP^1=S^1=\R \cup \infty$ and choose the points $\infty$ and $[1,1,\cdots ,1]$ as the base points of
$\RP^1$ and $\XS$, respectively.
Then, by setting $z=\frac{z_0}{z_1}$,  for each $1\leq k \leq r$,
we can view $f_k$ as a monic polynomial  $f_k(z)\in\K [z]$
of degree $d_k$ in the real variable  $z$.
\qed
}
%%%%%%%%%%%%%%%%%%%%%%%
\end{rmk}
%%(End of Remark 2.8)%%
%%%%%
%%
%%
%%

%%%(Spaces of resultants of bounded multiplicity)%%
\paragraph{2.4 Spaces of algebraic maps of real bounded multiplicity.}
%%%%%%%%%%%%%%%%%%%%%%%%%%%%%%%%%%%%%%%%%%%%%%%%%%%
Now
 we can define the space of algebraic maps as follows.
 
%%(Definition 2.9)%%
\begin{dfn}\label{dfn: poly}
{\rm
%%%
From now on, let $\K=\C$ or $\R$ as before.
\par
%%(i)%%
(i)
%%%%%%%%
For a monic polynomial $f(z)\in \K[z]$
of degree $d$, let $F_n(f)(z)$ denote the $n$-tuple of monic polynomials 
of the same degree $d$ defined by
%%(2.19)%%%
\begin{equation}\label{eq: Fn}
%%%%%%%%%
F_n(f)(z)=(f(z),f(z)+f^{\p}(z),f(z)+f^{\p\p}(z),\cdots ,f(z)+f^{(n-1)}(z)).
\end{equation}
%%
%%%%%%
Note that a monic polynomial $f(z)\in \K [z]$ has a root $\alpha\in \C$
of multiplicity $\geq n$ iff
$F_n(f)(\alpha)={\bf 0}_n\in \C^n.$
\par\vspace{1mm}\par
(ii)
For each $D=(d_1,\cdots ,d_r)\in \N^r$
and a fan $\Sigma$ in $\R^m$,
let $\Q^{D,\Sigma}_n(\K)$ denote the space of
$r$-tuples
$
(f_1(z),\cdots ,f_r(z))\in  \K [z]^r
$
of $\K$-coefficients
monic polynomials satisfying the  conditions
(\ref{eq: dpoly}a) and (\ref{eq: dpoly}b)
(as in Definition \ref{dfn: Pol(K)}).
%%%%%%
%\begin{enumerate}
%%%%%%%%%()%%%%
%%%%(1.2a)%%
%\item[(\ref{eq: dpoly}a)]
%%%%%%%%%%%%%
%For each $1\leq i\leq r$,
%$f_i(z)\in \K [z]$ is an $\mathbb{K}$-coefficients monic polynomial of the degree $d_i$.
%%%(1.2b)%%
%\item[(\ref{eq: dpoly}b)]
%%%%%%%%%%%%
%For each $\sigma =\{i_1,\cdots  ,i_s\}\in I(\KS)$,
%polynomials $f_{i_1}(z),\cdots ,f_{i_s}(z)$ have no common {\it real} root 
%$\alpha\in \R$ of multiplicity
%$\geq n$ (but may have a common root $\alpha \in \C\setminus\R$ of any multiplicity).
%%where
%%$\KS$ denotes the underlying simplicial complex of $\XS$ on the index set
%%$[r]=\{1,2,\cdots ,r\}$ defined by (\ref{eq: KS}) and
%%$I(\KS)$ is the set
%%$I(\KS)=\{\sigma\subset [r]:\sigma \not\in \KS\}$
%%defined by (\ref{eq: IK}).
%%
%%\item[]
%%\item[$(\dagger_{\Sigma,n})$]
%%%%%%%%%%%%%%%%%%%
%%For any $\sigma =\{i_1,\cdots ,i_s\}\in I(\KS)$,
%%polynomials $f_{i_1}(z),\cdots ,f_{i_s}(z)$ have no common 
%%{\it real} root of
%%multiplicity $\geq n$
%%(but they may have a common complex root $\alpha\in \C\setminus \R$ of any multiplicity).
%%%%%%%%%%%%%%%%%%
%\end{enumerate}
}
\end{dfn}
%%%%%(End of Definition 2.9)%%%
%%%
%%
%%
%%(Remark 2.10)%%
\begin{rmk}\label{rmk: 2.8}
{\rm 
(i)
Note that $\Q^{D,\Sigma}_n(\C)$ is path-connected, and that
$\Q^{D,\Sigma}_n(\R)$ is path-connected if
$(n,\rmin(\Sigma))\not=(1,2)$ 
%$n\geq 2$ 
%or if $n=1$ and $\rmin(\Sigma)\geq 3$
(which will be explained in Remark \ref{rmk: ESigma}).
\par
(ii)
Let $\Z_2=\{\pm 1\}$ denote the multiplicative cyclic group of order $2$, and
let  $X^{\Z_2}$ denote the $\Z_2$-fixed point set of a $\Z_2$-space $X$
as in (\ref{eq: XH}).
Complex conjugation on $\C$  extends to a $\Z_2$-actions on
the spaces
$\mathcal{Z}_{\KS}(\C^n,(\C^n)^*)$ and
$\Q^{D,\Sigma}_n(\C)$ such that
%%(2.20)%%
\begin{equation}\label{eq: 2.23}
%%%%
\mathcal{Z}_{\KS}(\R^n,(\R^n)^*)
=\mathcal{Z}_{\KS}(\C^n,(\C^n)^*)^{\Z_2},
\quad
\Q^{D,\Sigma}_n(\R)=\Q^{D,\Sigma}_n(\C)^{\Z_2}.
\end{equation}
%%%
%where we set $(\R^n)^*=\R^n\setminus \{{\bf 0}_n\}.$
It is easy to see that complex conjugation on $\C$ also
naturally extends to a $\Z_2$-action on
$X_{\Sigma}(n)=\mathcal{Z}_{\KS}(\C^n,(\C^n)^*)/G_{\Sigma}$ such that
%%(2.21)%%
\begin{equation}\label{eq: 2.24}
%%%%
X_{\Sigma,\R}(n)=\XS (n)^{\Z_2}.
\end{equation}
%%%%%%%%
It easily follows from the definition of the above actions that the following diagram is commutative:
%%(2.22)%%
\begin{equation}\label{CD: Z2-equivarinat maps}
\begin{CD}
\Q^{D,\Sigma}_n(\R) 
@>q_{n,\R}>> X_{\Sigma,\R}(n)
=\mathcal{Z}_{\KS}(\R^n,(\R^n)^*)/G_{\Sigma,\R}
\\
@V{i_n^D}V{\cap}V @V{i_n^X}V{\cap}V
\\
\Q^{D,\Sigma}_n(\C) @>q_{n,\C}>> X_{\Sigma}(n)
=\mathcal{Z}_{\KS}(\C^n,(\C^n)^*)/G_{\Sigma,\C}
\end{CD}
\end{equation}
where let $i_n^D$ and $i_n^X$
denote he corresponding inclusion maps.

Remark that $\Q^{D,\Sigma}_n(\C)$ (resp. $\Q^{D,\Sigma}_n(\R)$) is simply connected 
if the condition $($\ref{1.4}$)^*$ 
(resp. $($\ref{1.4}$)^{\dagger}$) is satisfied
(which will be proved in Corollary \ref{crl: 1-connected}).
}
%%%%%%%%%%%%%%%
\end{rmk}
%%(End of Remark 2.10)%%%
%%%%%%
%%
%%
%%%(Definition 2.11)%%
\begin{dfn}
{\rm
Suppose that
the condition (\ref{equ: condition dk}$)^*$ holds, and let
$D=(d_1,\cdots ,d_r)\in \N^r$ be an $r$-tuple of positive integers satisfying the condition
%%(2.23)%%
\begin{equation}
\sum_{k=1}^rd_k\n_k={\bf 0}_m.
\end{equation}
%%%%%%%%%
%%
\par
(i)
First, consider the case $\K=\C$.
By Lemma \ref{lmm: Cox tric},
one can define a map
%%(2.24)%%
\begin{equation}\label{eq: iD-map}
%%%%%%%
i_{D,n,\C}:\Q^{D,\Sigma}_n(\C)\to \Omega \XS (n)
\qquad
\mbox{by}
\end{equation}
%%%
%%%()%%%
\begin{equation*}
i_{D,n,\C}(f)(\alpha)
=
\begin{cases}
[F_n(f_1)(\alpha),F_n(f_2)(\alpha),\cdots ,F_n(f_r)(\alpha)]
& \mbox{if }\alpha \in \R
\\
[\textbf{\textit{e}},\textbf{\textit{e}},
\cdots ,\textbf{\textit{e}}]
& \mbox{if }\alpha =\infty
\end{cases}
\end{equation*}
%%%
for $f=(f_1(z),\cdots ,f_r(z))\in \Q^{D,\Sigma}_n(\C)$ and
$\alpha \in \R \cup \infty =S^1$,
where 
we set
$\textbf{\textit{e}}=(1,1,\cdots ,1)\in \C^n$.
\par\vspace{2mm}\par
 Since
%%%%
the space $\Q^{D,\Sigma}_n(\C)$ is simply connected and 
$\Omega q_{n,\C}$ is a universal covering (by (iv)  of Remark \ref{rmk: 2.8}
and (ii) of Corollary \ref{crl: universal covering}),
the map $i_{D,n,\C}$ lifts to the space
$\Omega \mathcal{Z}_{\KS}(D^{2n},S^{2n-1})$, and there is a based map
%%(2.25)%%
\begin{equation}\label{eq: def. jn}
j_{D,n,\C}:\Q^{D,\Sigma}_n(\C)\to
\Omega \mathcal{Z}_{\KS}(D^{2n},S^{2n-1})
\simeq
\Omega  \mathcal{Z}_{\KS}(\C^n,(\C^n)^*)
\end{equation}
%%%%%%%%%
such that the following equality holds:
%%(2.26)%%
\begin{equation}
\Omega q_{n,\C}\circ j_{D,n,\C}=i_{D,n,\C}.
\end{equation}
%%%%%%%%
\par
(ii)
Next, consider the case $\K=\R$.
\par\vspace{1mm}\par
Recall the $\Z_2$-action on the spaces
$\Q^{D,\Sigma}_n(\C)$ and $\XS$ induced from complex conjugation on $\C$, and
remark that
the map $i_{D,n,\C}$ is a $\Z_2$-equivariant map.
Then, by (\ref{eq: 2.23}) and (\ref{eq: 2.24}), we see that
%%(2.27)%%
\begin{equation}
i_{D,n,\C}(\Q^{D,\Sigma}_n(\R))\subset \Omega \XS (n)^{\Z_2}
=\Omega X_{\Sigma,\R}(n).
\end{equation}
%%%
Thus, the restriction  $i_{D,n,\C}\vert \Q^{D,\Sigma}_n(\R)$ defines the map
%%(2.28)%%
\begin{equation}
i_{D,n,\R}=i_{D,n,\C}\vert \Q^{D,\Sigma}_n(\R): \Q^{D,\Sigma}_n(\R)\to \Omega X_{\Sigma,\R}(n)
\end{equation}
%%%
such that the following diagram is commutative:
%%(2.29)%%
\begin{equation}
\begin{CD}
\Q^{D,\Sigma}_n(\R)@>i_{D,n,\R}>> \Omega X_{\Sigma,\R}(n)
@<\Omega q_{n,\R}<\simeq< 
\Omega \mathcal{Z}_{\KS}(\R^n,(\R^n)^*)
\\
@V{i_{n}^D}VV @V{\Omega i_{n}^X}VV @V{\Omega j_n}VV
\\
\Q^{D,\Sigma}_n(\C)@>i_{D,n,\C}>> \Omega X_{\Sigma}(n)
@<\Omega q_{n,\C}<< \Omega \mathcal{Z}_{\KS}(\C^n,(\C^n)^*)
\end{CD}
\end{equation}
%%%%%%%%%
where the $j_n:\mathcal{Z}_{\KS}(\R^n,(\R^n)^*)
\stackrel{\subset}{\longrightarrow}
\mathcal{Z}_{\KS}(\C^n,(\C^n)^*)$ denotes the inclusion map.
%%%(2.34)%%%
%\begin{equation}
%\begin{cases}
%i_{n}^D:\Q^{D,\Sigma}_n(\R)
%\stackrel{\subset}{\longrightarrow}
%\Q^{D,\Sigma}_n(\C),
%\\
%i_{n}^X:X_{\Sigma,\R}(n)
%\stackrel{\subset}{\longrightarrow}
%X_{\Sigma,\C}(n)=X_{\Sigma}(n),
%\\
%j_n:\mathcal{Z}_{\KS}(\R^n,(\R^n)^*)
%\stackrel{\subset}{\longrightarrow}
%\mathcal{Z}_{\KS}(\C^n,(\C^n)^*).
%\end{cases}
%\end{equation}
%%%%%%%%%%%%%%%%
%\par
%%%%%%%%%%%
Note that $\Omega q_{n,\R}$ is a homotopy equivalence
(which will be proved in Corollary \ref{crl: universal covering}).
Thus, there is a based map
%%(2.30)%%
\begin{equation}\label{eq: the map jDR}
j_{D,n,\R}:\Q^{D,\Sigma}_n(\R) \to 
\Omega \mathcal{Z}_{\KS}(\R^n,(\R^n)^*)
\simeq
\Omega \mathcal{Z}_{\KS}(D^n,S^{n-1})
\end{equation}
%%%
which satisfies the following equality:
%%(2.31)%%
\begin{equation}\label{eq: j-map K=R}
%%%%%%%
\Omega q_{n,\R}\circ j_{D,n,\R}=i_{D,n,\R}
\quad
(\mbox{up to homotopy}).
\end{equation}
%%%
}
\end{dfn}
%%(End of Definition 2.11)%%%
%%
%%(Remark 2.12)%%
\begin{rmk}
%%%%
{\rm
When $\sum_{k=1}^rd_k\n_k={\bf 0}_n$, 
by (\ref{eq: def. jn}) and (\ref{eq: j-map K=R}), we obtain the map
%%(2.32)%%
\begin{equation}
j_{D,n,\K}:\Q^{D,\Sigma}_n(\K)
\to
\Omega \mathcal{Z}_{\KS}(\K^n,(\K^n)^*)\simeq
\Omega \mathcal{Z}_{\KS}(D^{d(\K)n},S^{d(\K)n-1}),
\end{equation}
%%%
where the number $d(\K)$ is defined by
%%(2.36)%%
\begin{equation}\label{eq: dimK}
d(\K)=\dim_{\R} \K=
\begin{cases}
2 & \mbox{ if }\K=\C,
\\
1 & \mbox{ if }\K=\R.
\end{cases}
\end{equation}
}
%%%%
\end{rmk}
%%(End of Remark 2.11)%%
%%
%%
%%
%%%%%%(The previous result)%%
\paragraph{2.5 The numbers $\rmin (\Sigma)$ and
$d(D;\Sigma,n,\K).$}
%%%%%%%%%%%%%%%%%%%%%%%%%%%%%
%%
%%
%%%%%%%%%%%%%%
Before stating the main results of this paper, we need to define the positive integers
$\rmin (\Sigma)$ and
$d(D;\Sigma,n,\K)$ (which already appeared in the statements of our results). 
%%
%%
%%%%%%%%%%%%%%%%%%%%%%%%%
%%(Definition 2.12)%%%%%%%
\begin{dfn}\label{dfn: rnumber}
%%%%%%%%%%%%%%%%%%%%%%%%%
{\rm
%%%%%%%%%%
(i)
We say that a set $S=\{\textbf{\textit{n}}_{i_1},\cdots ,\textbf{\textit{n}}_{i_s}\}$
%of some primitive generators
is {\it a primitive collection} 
 if  $\mbox{Cone}(S)\notin\Sigma$ and $\mbox{Cone}(T)\in \Sigma$ for any
 proper subset $T\subsetneqq S$.
%%()%%
\par
(ii) For each $r$-tuple $D=(d_1,\cdots ,d_r)\in \N^r$,
%%%%%%%
 define  the positive integer $d(D,\Sigma,n,\K)$  by
%%(2.34)%%%%%%%}
\begin{align}\label{eq: dDSigma}
%%%%%%%%%%%%%%%
d(D;\Sigma ,n,\K)
%&=
%(nd(\K)\rmin (\Sigma)  -2)\lfloor d_{\rm min}/n\rfloor -2
%\\
%\nonumber 
&=
\begin{cases}
(2n\rmin (\Sigma)  -2)\lfloor \frac{\dmin}{n}\rfloor -2 & \mbox{ if }\K=\C
\\
(n\rmin (\Sigma)  -2)\lfloor \frac{\dmin}{n}\rfloor -2 & \mbox{ if }\K=\R
\end{cases}
\end{align}
%%%%%%
as in (\ref{eq: number d(D)}),
where  
$\rmin (\Sigma)$ and $d_{\rm min}$ are the positive integers given by
%%%(2.35)%%
\begin{equation}\label{eq: rmin}
%%()%%
\begin{cases}
\rmin (\Sigma)
=\min\{s\in\N :
\{\textbf{\textit{n}}_{i_1},\cdots ,\textbf{\textit{n}}_{i_s}\}
\mbox{ is a primitive collection}\},
\\
d_{\rm min}=
%d_{\rm min}(D)=
\min\{d_1,d_2,\cdots ,d_r\}.
\end{cases}
\end{equation}
%%%%%
Note that 
%%(2.36)%%
\begin{equation}\label{eq: rmin >1}
\rmin (\Sigma)\geq 2.
\end{equation}
%%%%%%%
}
%%%%%%
\end{dfn}
%%%%%%%%(End of Definition 2.13)%%%
%%%%
%%%%
%%%%
%%%
%%%
%%%

%%
%%
%%
%%
%%
%%
%%
%%
%%%%%%%%%%%%%%%%%%%%%%%%%%%%%
\paragraph{2.6 The main results.}
%%%%%%%%%%%%%%%%%%%%%%%%%%%%%
Note that the space $\Q^{D,\Sigma}_n(\C)$ has already been extensively
studied in the case $n=1$  in \cite{KY11}.
%\footnote{%
%%%%%(FootNote 7)%%
%The proof of \cite[Proposition 5.9]{KY11} contains a mistake which
%is corrected in  Remark \ref{rmk: prop. 5.9 in KY11}.
% We will assume the condition $(\dmin,\rmin (\Sigma))\not= (1,2)$
%in the statements below 
%(see Remark \ref{rmk: prop. 5.9 in KY11}).
%Moreover, in \cite{KY11}, we write
%$\Q^{D,\Sigma}_1(\C)={\rm Pol}^*_D(S^1,\XS)$.
%}
%%(End of FootNote 7)%%%
%%%
The main purpose of this paper is to generalize the  results
of \cite{KY11}
to the space $\Q^{D,\Sigma}_n(\K)$
%of real non-resultant systems of real bounded multiplicity
for $\K=\C$ or $\R$ and for any $n\geq 1$.
Indeed,
the main results of this article are stated as follows.
%From now on, we shall only consider the case $n\geq 2$.
%\par
%%%%
%%%%
%%%%
%%%%
%%%%%%%%%%%%%%%%%%%%%%%%%%%%%%%%%%%%%%%
%%%(Theorem 2.14: The main result :No.1)%%%%%
%%%%%%%%%%%%%%%%%%%%%%%%%%%%%%%%%%%%%%%
\begin{thm}[The case $\K=\C$]\label{thm: I} 
%%%%%%%%%%%%%%%%%%%%%%
Let $n\in \N$, let
$D=(d_1,\cdots ,d_r)\in \N^r$,
%be $r$-tuple of positive integers such that $\dmin \geq n$, 
%$\dmin \geq n\geq 2$, 
and let $\XS$ be an $m$ dimensional simply connected non-singular toric variety 
such that 
the  two conditions $($\ref{equ: condition dk}$)^*$ and $($\ref{1.4}$)^*$ holds.
\par
%%%(i)%%
$\I$ If
$\sum_{k=1}^rd_k\textbf{\textit{n}}_k={\bf 0}_m$, then 
the  map $($given by $($\ref{eq: def. jn}$))$
$$
j_{D,n,\C}:\Q^{D,\Sigma}_n(\C)\to 
\Omega \mathcal{Z}_{\KS}(D^{2n},S^{2n-1})
$$ 
is a homotopy equivalence
through dimension $d(D;\Sigma ,n,\C)$.
%%%
\par
%%(ii)%%
$\II$
%%%%%%
If
$\sum_{k=1}^rd_k\textbf{\textit{n}}_k\not= {\bf 0}_m$,  
there is a map
$$
j_{D,n,\C}:\Q^{D,\Sigma}_n(\C) \to  
%\Omega^2_0\XS (n)\simeq
\Omega \mathcal{Z}_{\KS}(D^{2n},S^{2n-1})
$$ 
which
is a homotopy equivalence
through dimension $d(D;\Sigma ,n,\C)$.\footnote{%
%%%(FootNote 7)%%%
This map has to be constructed in a slightly different way from the one in $\I$ but we shall use the same notation for both.}
%%%%%(End of FootNote 7)%%%%
%%%%
\end{thm}
%%%%%%%%(End od Theorem 2.14)%%
%%
%%
%%%(Theorem 2.15: The main result :No.2)%%%%%
%%%%%%%%%%%%%%%%%%%%%%%%%%%%%%%%%%%%%%%
\begin{thm}[The case $\K=\R$]\label{thm: I R} 
%%%%%%%%%%%%%%%%%%%%%%
Let $n\in \N$, let
$D=(d_1,\cdots ,d_r)\in \N^r$,
%be $r$-tuple of positive integers such that $\dmin \geq n$, 
%$\dmin \geq n\geq 2$, 
and let $\XS$ be an $m$ dimensional simply connected non-singular toric variety 
such that 
the two conditions $($\ref{equ: condition dk}$)^*$ and $($\ref{1.4}$)^{\dagger}$ hold.
\par
%%%(i)%%
$\I$ If
$\sum_{k=1}^rd_k\textbf{\textit{n}}_k={\bf 0}_m$, then 
the  map $($given by $($\ref{eq: the map jDR}$))$
$$
j_{D,n,\R}:\Q^{D,\Sigma}_n(\R)\to 
\Omega \mathcal{Z}_{\KS}(D^{n},S^{n-1})
$$ 
is a homotopy equivalence
through dimension $d(D;\Sigma ,n,\R)$.
%%%
\par
%%(ii)%%
$\II$
%%%%%%
If
$\sum_{k=1}^rd_k\textbf{\textit{n}}_k\not= {\bf 0}_m$,  
there is a map
$$
j_{D,n,\R}:\Q^{D,\Sigma}_n(\R) \to  
%\Omega^2_0\XS (n)\simeq
\Omega \mathcal{Z}_{\KS}(D^{n},S^{n-1})
$$ 
which
is a homotopy equivalence
through dimension $d(D;\Sigma ,n,\R)$.
%%%%
\end{thm}
%%%%%%%%(End of Theorem 2.15)%%
%%
%%
%%(Corollary 2.16)%%
\begin{crl}\label{crl: I}
Let $n\in \N$, let
$D=(d_1,\cdots ,d_r)\in \N^r$.
%be $r$-tuple of positive integers such that $\dmin \geq n$, 
%$\dmin \geq n\geq 2$ 
and let $\XS$ be an $m$ dimensional simply connected non-singular toric variety 
such that 
the two conditions $($\ref{equ: condition dk}$)^*$ and $($\ref{1.4}$)^*$ hold.
\par
%%%(i)%%
$\I$ If
$\sum_{k=1}^rd_k\textbf{\textit{n}}_k={\bf 0}_m$, then 
the  map
$
i_{D,n,\C}:\Q^{D,\Sigma}_n(\C) \to 
\Omega \XS (n)
$ 
induces an isomorphism
$$
(i_{D,n,\C})_*:\pi_s(\Q^{D,\Sigma}_n(\C)) \stackrel{\cong}{\longrightarrow} 
\pi_s(\Omega \XS)\cong \pi_{s+1}(\XS (n))
$$
for any $2\leq s\leq d(D;\Sigma,n,\C)$.
%%%
\par
%%(ii)%%
$\II$
%%%%%%
If
$\sum_{k=1}^rd_k\textbf{\textit{n}}_k\not= {\bf 0}_m$,  
the map 
$i_{D,n,\C}:\Q^{D,\Sigma}_n(\C)\to \Omega \XS(n)$ defined by
%%(2.37)%%
\begin{equation}
i_{D,n,\C}:=\Omega q_{n,\C}\circ j_{D,n,\C}
\end{equation}
induces an isomorphism 
$$
(i_{D,n,\C})_*:\pi_s(\Q^{D,\Sigma}_n(\C)) \stackrel{\cong}{\longrightarrow} 
\pi_s(\Omega \XS (n))\cong \pi_{s+1}(\XS (n))
$$
for any $2\leq s\leq d(D;\Sigma,n,\C)$.
%%%%
\end{crl}
%%(End of Corollary 2.15)%%%
%%
%%
%Let $\Z_2=\{\pm1\}$ and
%be the multiplicative cyclic group of order $2$.
Consider the $\Z_2$-action on the spaces
$\Q^{D,\Sigma}_n(\C)$ and $\mathcal{Z}_{\KS}(D^{2n},S^{2n-1})$
induced from the complex conjugation on $\C$, where we identify
%%(2.38)%%
\begin{equation}
D^{2n}=\{(x_1,\cdots ,x_n)\in \C^n:\sum_{k=1}^n|x_k|^2\leq 1\}.
\end{equation}
%%%%
Note that we can regard the space $D^{2n}$ as a $\Z_2$-space whose $\Z_2$ action is given by
the complex conjugation.
%%(2.39)%%
\begin{equation}
(-1)\cdot (x_1,\cdots ,x_n)=
(\overline{x_1},\cdots ,\overline{x_n})
\quad
\mbox{for }(x_1,\cdots ,x_n)\in D^{2n}.
\end{equation}
\par
Since $\Q^{D,\Sigma}_n(\R)=\Q^{D,\Sigma}_n(\C)^{\Z_2},$ 
$\mathcal{Z}_{\KS}(D^{n},S^{n-1})
=\mathcal{Z}_{\KS}(D^{2n},S^{2n-1})^{\Z_2}$,
and 
$j_{D,n,\R}=(j_{D,n,\C})^{\Z_2}$,
we also obtain the following result.
%%
%%
%%(Corollary 2.17)%%
\begin{crl}\label{crl: cor S1}
%%%
Let $n\in \N$, let
$D=(d_1,\cdots ,d_r)\in \N^r$, 
and
let $\XS$ be an $m$ dimensional simply connected non-singular toric variety 
satisfying
the two conditions $($\ref{equ: condition dk}$)^*$ and $($\ref{1.4}$)^{\dagger}$.
Then the map
$$
j_{D,n,\C}:\Q^{D,\Sigma}_n(\C) \to 
\Omega \mathcal{Z}_{\KS}(D^{2n},S^{2n-1})
$$ 
is a $\Z_2$-equivariant homotopy equivalence through dimension
$d(D;\Sigma,n,\R)$.
\qed
\end{crl}
%%(End of Corollary 2.17)%%
%%
%%%%%%%%%%%%%%%%%%%%%
%%
Finally, we easily obtain the following two corollaries.
%%
%%(Corollary 2.18)%%
\begin{crl}\label{crl: II}
%%%%%%%%%%%%%%%%%%%%%
Let $n\in \N$, let
$D=(d_1,\cdots ,d_r)\in \N^r$,
%be $r$-tuple of positive integers such that $\dmin \geq n$, 
and let $\XS$ be a simply connected compact non-singular toric variety
such that 
the the condition $($\ref{equ: condition dk}$)^*$  holds.
Let $\Sigma (1)$ denote the set of all one dimensional cones in $\Sigma$, 
and let $\Sigma_1$ be any fan in $\R^m$ satisfying the condition
%%%%%(2.40)%%%
\begin{equation}\label{eq: fan Sigma_1}
%%%%%%%%%%%%
\Sigma (1)\subset \Sigma_1\subsetneqq \Sigma.
%%%%%%%
\end{equation}
%%%%%%%
%\begin{enumerate}
%\item[$\I$]
\par
$\I$
Then  $X_{\Sigma_1}$ is a non-compact smooth 
toric subvariety of $\XS$.
%%(ii)%%
\par
$\II$
%\item[$\II$]
If the condition $($\ref{1.4}$)^*$ holds and
$\sum_{k=1}^rd_k\textbf{\textit{n}}_k={\bf 0}_m$, then
the map
$$
j_{D,n,\C}:\Q^{D,\Sigma_1}_n(\C)\to \Omega \mathcal{Z}_{\Sigma_1}(D^{2n},S^{2n-1})
$$ 
 is a
homotopy equivalence through the dimension $d(D;\Sigma_1,n,\C)$.
\par
Moreover, the map
$i_{D,n,\C}:\Q^{D,\Sigma_1}_n(\C)\to \Omega X_{\Sigma_1}$ induces an isomorphism
$$
(i_{D,n,\C})_*:\pi_s(\Q^{D,\Sigma_1}_n(\C)) \stackrel{\cong}{\longrightarrow} 
\pi_s(\Omega X_{\Sigma_1}(n))\cong \pi_{s+1}(X_{\Sigma_1}(n))
$$
for any $2\leq s\leq d(D;\Sigma_1,n,\C)$.
%%%
\par
$\III$
%\item[$\III$] 
If the condition $($\ref{1.4}$)^*$ holds and
$\sum_{k=1}^rd_k\textbf{\textit{n}}_k\not={\bf 0}_m$, then
there is a map
$$
j_{D,n,\C}:\Q^{D,\Sigma_1}_n(\C)\to 
\Omega \mathcal{Z}_{\mathcal{K}_{\Sigma_1}}(D^{2n},S^{2n-1})
$$ 
which is a
homotopy equivalence through dimension $d(D;\Sigma_1,n,\C)$.
%%%
\par
Moreover,
the map 
$i_{D,n,\C}:\Q^{D,\Sigma_1}_n(\C)\to \Omega X_{\Sigma_1}$ defined by
%%(2.41)%%
\begin{equation}
%%%%%%%%
i_{D,n,\C}:=\Omega q_{n,\C}\circ j_{D,n,\C}
\end{equation}
induces an isomorphism 
$$
(i_{D,n,\C})_*:\pi_s(\Q^{D,\Sigma_1}_n(\C)) \stackrel{\cong}{\longrightarrow} 
\pi_s(\Omega X_{\Sigma_1}(n))\cong \pi_{s+1}(X_{\Sigma_1}(n))
$$
for any $2\leq s\leq d(D;\Sigma_1,n,\C)$.
%%%%
\par
$\IV$
If the condition $($\ref{1.4}$)^{\dagger}$ holds and
$\sum_{k=1}^rd_k\textbf{\textit{n}}_k={\bf 0}_m$, then
the map
$$
j_{D,n,\R}:\Q^{D,\Sigma_1}_n(\R)\to \Omega \mathcal{Z}_{\Sigma_1}(D^{n},S^{n-1})
$$ 
 is a
homotopy equivalence through the dimension $d(D;\Sigma_1,n,\K)$.
\par
$\V$
If the condition $($\ref{1.4}$)^{\dagger}$ holds and
$\sum_{k=1}^rd_k\textbf{\textit{n}}_k\not={\bf 0}_m$, there is a map
$$
j_{D,n,\R}:\Q^{D,\Sigma_1}_n(\R)\to 
\Omega \mathcal{Z}_{\mathcal{K}_{\Sigma_1}}(D^{n},S^{n-1})
$$ 
which is a
homotopy equivalence through dimension $d(D;\Sigma_1,n,\R)$.
%\end{enumerate}
%%%%%%%%%
\end{crl}
%%%%%%%%%%%%(End of Corollary 2.18)%%%%
%%
%%%
%%%(Corollary 2.18)%%
%\begin{crl}\label{crl: II K=R}
%%%%%%%%%%%%%%%%%%%%%%
%Let $n\in \N$, let
%$D=(d_1,\cdots ,d_r)\in \N^r$,
%and let $\XS$ be a simply connected compact non-singular toric variety 
%such that 
%the two conditions $($\ref{equ: condition dk}$)^*$ and $($\ref{1.4}$)^{\dagger}$ hold.
%%%
%Let $\Sigma_1$ be any fan in $\R^m$ satisfying the condition (\ref{eq: fan Sigma_1}).
%%%%%%
%%%%%%%%
%%\end{equation}
%%%%%%%%
%%\begin{enumerate}
%%\item[$\I$]
%%Then  $X_{\Sigma_1}$ is a non-compact smooth 
%%toric subvariety of $\XS$.
%%%(ii)%%
%%\item[$\I$]
%\par
%$\I$
% Let $\sum_{k=1}^rd_k\textbf{\textit{n}}_k={\bf 0}_m$.
% Then the map
%$$
%j_{D,n,\R}:\Q^{D,\Sigma_1}_n(\R)\to \Omega \mathcal{Z}_{\Sigma_1}(D^{n},S^{n-1})
%$$ 
% is a
%homotopy equivalence through the dimension $d(D;\Sigma_1,n,\R)$.
%%%%
%%\item[$\II$] 
%\par
%$\II$
%If $\sum_{k=1}^rd_k\textbf{\textit{n}}_k\not={\bf 0}_m$, there is a map
%$$
%j_{D,n,\R}:\Q^{D,\Sigma_1}_n(\R)\to 
%\Omega \mathcal{Z}_{\mathcal{K}_{\Sigma_1}}(D^{n},S^{n-1})
%$$ 
%which is a
%homotopy equivalence through dimension $d(D;\Sigma_1,n,\R)$.
%%%%%
%\qed
%%\end{enumerate}
%%%%%%%%%%
%\end{crl}
%%%%%%%%%%%%%(End of Corollary 2.17)%%%%
%%%
%%
%%%
%%%
%%%(End of section 2)%%%%%%%%%%%
%%%(End of section 2)%%%%%%%%%%%
%%%(End of section 2)%%%%%%%%%%%
%%
%%
%%%(SECTION 3)%%
%%%(SECTION 3)%%%%
%%%(SECTION 3)%%%%
%%%%%%%%%%%%%%%%%%%%%%%%%%
\section{Basic facts about toric varieties}
\label{section: polyhedral products}
%%%%%%%%%%%%%%%%%%%%%%%%%%

%\paragraph{3.1. Polyhedral products.}
In this section, we recall some basic definitions and known results.

%%%(Definition 3.1)%%%
\begin{dfn}[\cite{BP}, Definition 6.27, Example 6.39]\label{def: moment}
%%%
{\rm
Let $K$ be a simplicial complex on the index set $[r]$, and
let $I(K)=\{\sigma\subset [r]:\sigma\notin K\}$ as in (\ref{eq: IK}).
\par
(i) 
An element $\sigma\in I(K)$ is called
{\it a minimal non-face of $K$}
if $\tau\in K$ for any proper subset $\tau\subsetneqq \sigma$.
%%%
%\par\noindent
%%%(ii)%%%%%
\par
(ii)
Then 
we denote by $I_{\rm min}(K)$ the set of all minimal non-faces of $K$.
It is easy to see that the following equality holds.
%%(3.1)%%
\begin{align}\label{eq: minimal non-face}
K&=
\{\sigma\subset [r]:
\tau\not\subset \sigma\mbox{ for any }\tau\in I_{\rm min}(K)\}.
%%%%%%
\end{align}
%%(ii)%%
\par
(iii)
We denote by $\mathcal{Z}_K$ and $DJ(K)$
{\it the moment-angle complex} of $K$ and 
{\it the Davis-Januszkiewicz space} of $K$ 
which are
defined by}
%%(3.2)%%%
\begin{equation}\label{DJ}
%%%%%%%%
\mathcal{Z}_K=\mathcal{Z}_K(D^2,S^1),\quad
DJ(K)=\mathcal{Z}_K(\CP^{\infty},*).
\qquad\qquad
\qquad
\qed
%%%
\end{equation}
%%%%%
\end{dfn}
%%(End of Definition 3.1)%%%
%%
%%%%(Remark 3.2)%%
\begin{rmk}
%%%%
{\rm 
Let $\Sigma$ be a fan in $\R^m$ and let
$\XS$ be a smooth toric variety such that the condition 
$($\ref{equ: condition dk}$)^*$ holds.
Then it is easy to see that 
$\{\n_{i_1},\n_{i_2},\cdots ,\n_{i_s}\}$ is primitive
if and only if
$\sigma=\{i_1,i_2,\cdots ,i_s\}\in
I_{\rm min}(\KS).$
%the following equality holds:
%$$
%\{\n_{i_1},\n_{i_2},\cdots ,\n_{i_s}\}
%\mbox{ is primitive}
%\ \Leftrightarrow\ 
%\sigma=\{i_1,i_2,\cdots ,i_s\}\in
%I_{\rm min}(\KS).
%$$
Thus, we also obtain the following equality:
%%(3.3)%%
\begin{equation}\label{eq: rmin min=mini sigma}
%%%%%%%
\rmin (\Sigma)=\min\{\mbox{card}(\sigma):\sigma\in I(\KS)\}.
\qquad
\qed
\end{equation}
}
%%%%%%%%%%%%%%%%
\end{rmk}
%%%(End of Remark 3.2)%%
%%%
%%%%%(Lemma 3.3)%%%
\begin{lmm}[\cite{BP}; Corollary 6.30, Theorems 6.33,
8.9]\label{Lemma: BP}
%%%%%%%%%%%%%%%%%%%
Let $K$ be a simplicial complex on the index set $[r]$.
\par
$\I$
The space $\mathcal{Z}_K$ is $2$-connected, and
there is a fibration sequence
%%(3.4)%%
\begin{equation}
\mathcal{Z}_K \stackrel{}{\longrightarrow} DJ(K)
\stackrel{\subset}{\longrightarrow} (\CP^{\infty})^r.
\end{equation}
\par
$\II$
There are $\T^r$-equivariant deformation retraction
%%(3.5)%%
\begin{equation}\label{eq: retr}
ret:\mathcal{Z}_K(\K^n,(\K^n)^*)\stackrel{\simeq}{\longrightarrow}
\mathcal{Z}_K(D^{d(\K)n},S^{d(\K)n-1}).
\end{equation}
where
we set $\T^r=(S^1)^r$.
\qed
\end{lmm}
%%%%(End of Lemma 3.3)%%
%%
%%
%%
%%
%%(Lemma 3.4)%%%%%%
\begin{lmm}[\cite{Pa1}]
\label{lmm: principal}
%%%%%%%%%%%%%%%
Let $\Sigma$ be a fan in $\R^m$ and let
$\XS$ be a smooth toric variety such that the condition 
$($\ref{equ: condition dk}$)^*$ holds.
\par
$\I$
There is an isomorphism 
%%(3.6)%%
\begin{equation}\label{eq: GS-eq}
G_{\Sigma,\K}\cong \T^{r-m}_{\K}=(\K^*)^{r-m}.
\end{equation}
%%%%%
\par
$\II$
The group $G_{\Sigma,\K}$ acts on the space 
$\mathcal{Z}_{\mathcal{K}_{\Sigma}}(\K^n,(\K^n)^*)$
freely as in (\ref{eq: coordinate multiplication})
%by the coordinate-wise multiplication,
and
 there is a principal
$G_{\Sigma,\K}$-bundle sequence
%%%%%%(3.7)%%
\begin{equation}\label{eq: principal}
%%%%%%%%%
G_{\Sigma,\K} \stackrel{}{\longrightarrow}
\mathcal{Z}_{\KS}(\K^n,(\K^n)^*)
\stackrel{q_{n,\K}}{\longrightarrow}
X_{\Sigma,\K}.
%%%%%%%%%%%%%%
\end{equation}
%%%%%%%%%%%%%%
\par
$\III$
If $\K=\R$, there is a homotopy equivalence
$\T^r_{\R}\simeq (\Z_2)^{r-m}$ and
the map
$q_{n,\R}$ is a covering projection with fiber 
$(\Z_2)^{r-m}$ (up to homotopy).
%%(Proof of Lemma 3.4)%%%
\begin{proof}
%%%%%%%%%%%%%%%%
%%
First, consider the case $\K=\C$.
Then the assertions (i) and (ii) follow from \cite[(6.2) page 527; Proposition 6.7]{Pa1}.
%(cf. \cite[Lemma 3.3, Corollary 3.4]{KY12}).
\par
Next, let $\K=\R$.
Since $G_{\Sigma}=\GS\cong (\C^*)^{r-m}$ and $G_{\Sigma,\R}=
G_{\Sigma}\cap (\R^*)^r$, we have the isomorphism
$G_{\Sigma,\R}\cong (\R^*)^{r-m}=\T^{r-m}_{\R}$.
Since $\GS$ acts on the space $\mathcal{Z}_{\KS}(\C^n,(\C^n)^*)$ freely,
the subgroup $\G_{\Sigma,\R}$ also acts on the space
$\mathcal{Z}_{\KS}(\R^n,(\R^n)^*)$ freely and we obtain the 
$G_{\Sigma,\R}$-principal fibration sequence (\ref{eq: principal}) for the case $\K=\R$.
This proves (i) and (ii)  for the case $\K=\R$.
Since $G_{\Sigma,\R}\simeq (\Z_2)^{r-m}$, $q_{n,\R}$ is a covering projection with fiber
$(\Z_2)^{r-m}$, and we obtain (iii).
\end{proof}
%%%%%
\end{lmm}
%%%(End of Lemma 3.4)%%
%%
%%
%%(Definition 3.5)%%
\begin{dfn}[c.f. \cite{KY12}, (5.26)]
%%%%
{\rm
Let $\Sigma$ be a fan in $\R^m$ and let
$\XS$ be a smooth toric variety such that the condition 
$($\ref{equ: condition dk}$)^*$ holds.
 \par
(i)
Let $\KS (n)$ denote the simplicial complex on the index set $[r]\times [n]$ defined by
%%%(3.8)%%
\begin{equation}\label{eq: def. KS(n)}
%%%%%
\KS (n)=\{\tau \subset [r]\times [n]:\sigma \times [n]\not\subset \tau
\mbox{ for any }\sigma \in I(\KS)\}.
\end{equation}
%%%
%%\end{dfn}
\par
(ii)
For each $(i,j)\in [r]\times [n]$,
let $\n_{i,j}\in \Z^{mn}$
denote the lattice vector
defined by
%%(3.9)%%
\begin{equation}
\n_{i,j}=(\textbf{\textit{a}}_1,\cdots ,\textbf{\textit{a}}_n),
\mbox{ where we set }
\textbf{\textit{a}}_k=
\begin{cases}
\n_i & (k=j)
\\
{\bf 0}_m & (k\not= j)
\end{cases}
\end{equation}
%%%%%%%%%%%
and define a fan $F_n(\Sigma)$ in $\R^{mn}$ by
%%(3.10)%%
\begin{equation}\label{eq: Fn}
%%%%%
F_n(\Sigma)=\{c_{\tau}:\tau\in \KS (n)\},
\end{equation}
%%%%%
where $c_{\tau}$ denotes the cone in $\R^{mn}$
given by
%generated by $\{ \n_{i,j}:(i,j)\in \tau\}$, i.e.
%%(3.11)%%
\begin{equation}
c_{\tau}=\mbox{Cone}(\{ \n_{i,j}:(i,j)\in \tau\})
=\Big\{\sum_{(i,j)\in \tau}x_{i,j}\nn_{i,j}:x_{i,j}\geq 0\Big\}.
\end{equation}
}
\end{dfn}
%%%(End of Definition 3.5)%%
%%%%
%%%%
%%%%
%%%%
%%(Lemma 3.6)%%
\begin{lmm}\label{lmm: simplicial complex of KS(n)}
%%%%
\begin{enumerate}
\item[$\I$]
If $\T^r=(S^1)^r$, there is a $\T^r$-equivariant homeomorphism
%%(3.12)%%
\begin{equation}
\mathcal{Z}_{\KS}(D^{2n},S^{2n-1})\cong
\mathcal{Z}_{\KS (n)}(D^2,S^1).
\end{equation}
%%%
\item[$\II$]
If $\T^r_{\C}=(\C^*)^r$,
there is a $\T^r_{\C}$-equivariant homeomorphsim
%%(3.13)%%
\begin{equation}
\mathcal{Z}_{\KS}(\C^n,(\C^n)^*)\simeq
\mathcal{Z}_{\KS (n)}(\C,\C^*).
\end{equation}
%%(ii)%%
\item[$\III$]
The space $\mathcal{Z}_{\KS}(D^{2n},S^{2n-1})$ is $2$-connected.
\end{enumerate}
\end{lmm}
%%(Proof of Lemma 3.6)%%%%
\begin{proof}
%%%%%%%%%%%
(i)
Let $J=(n,n,\cdots ,n)\in \N^r$ and let $\KS (J)$ denote the simplical complex
on the index set $[r]\times [n]$ defined by \cite[Definition 2.1]{BBCG}.\footnote{%
%%%%%%%%%%%%%%%%%%%%%
%%%(FootNote 8)%%%%%%%%%%%
More precisely, if we set
$(K,m)=(\KS,r)$ and $J=(j_1,j_2,\cdots ,j_r)=(n,n,\cdots ,n)$
$(r$-times)
 in the notation of \cite[Definition 2.1]{BBCG}, we get the simplicial complex
$K(J)$ on the index set $[r]\times [n]$.
%%%%%%%%%%%%%%%%%%%%%
}
%%(End of FootNote 8)%%%%
%%%%%
Then it follows from \cite[Definition 2.1]{BBCG} that
the following equality holds:
%%(3.14)%%
\begin{equation}\label{eq: minimal non-face condition}
%%%%%%%%
I_{\rm min}(\KS (J))=\{\tau\times [n]:\tau\in I_{\rm min}(\KS)\}.
\end{equation}
%%%
Hence, by  (\ref{eq: minimal non-face}) and (\ref{eq: def. KS(n)}),
 we obtain  the following equality:
%%%%()%%%%
\begin{align*}\label{eq: KJ=KS(n)}
\KS (J)&=\{\sigma\subset [r]\times [n]:
\tau\times [n]\not\subset \sigma
\mbox{ for any }\tau\in I_{\rm min}(\KS)\}.
%=
%\KS (n).
\end{align*}
%%%
Thus, we we have $\KS (J)=\KS (n)$ by (\ref{eq: minimal non-face condition}).
Hence,
it follows from \cite[Theorem 7.5]{BBCG} that there is a
$\T^r$-equivariant homeomorphism
$
\mathcal{Z}_{\KS}(D^{2n},S^{2n-1})\cong
\mathcal{Z}_{\KS (n)}(D^2,S^1),
$
and the assertion (i) follows.
\par
(ii)
It follows from \cite[(5.32)]{KY12} that there is a homeomorphism
$$
\mathcal{Z}_{\KS}(\C^n,(\C^n)^*)\cong
\mathcal{Z}_{\mathcal{K}_{\KS (n)}}(\C,\C^*).
$$
One can easily check that the above homeomorphism is $\T^r_{\C}$-equivariant, and the assertion (ii) follows.
%%%%
%Alternatively, one can prove the assertion (ii) by using the equality 
%$\KS (J)=\KS (n)$
%%(\ref{eq: KJ=KS(n)}) 
%and the method given in the proof of
%\cite[Theorem 7.5]{BBCG}.
\par
(iii)
It follows from (i) and (ii) that there is the following homotopy equivalence
%%()%%
\begin{align*}
%%%%
\mathcal{Z}_{\KS}(D^{2n},S^{2n-1})
&\simeq
\mathcal{Z}_{\KS}(\C^n,(\C^n)^*)
\cong
\mathcal{Z}_{\KS (n)}(\C,\C^*)
\simeq
\mathcal{Z}_{\KS (n)}(D^2,S^1).
\end{align*}
Since the moment-angle complex 
$\mathcal{Z}_{\KS (n)}(D^2,S^1)$ is 
$2$-connected by \cite[Theorem 6.33]{BP}, the space
$\mathcal{Z}_{\KS}(D^{2n},S^{2n-1})$ is also $2$-connected.
\end{proof}
%%%(End of proof of Lemma 3.6)%%
%%%
%%%
%%%
%%(Definition 3.7)%%%
\begin{dfn}[\cite{CLS}]\label{dfn: smooth}
%%%
{\rm 
Let $\Sigma$ be a fan in $\R^m$.
Then a cone $\sigma\in \Sigma$ is called
{\it smooth} if it is generated by a subset of a basis of $\Z^m$.
}
%%%%%%%
\end{dfn}
%%(End of Definition 3.7)%%
%%
%%
%%
%%%(Lemma 3.8)%%
\begin{lmm}[\cite{CLS}]\label{lmm: toric}
%%%%%%%%%%%%%%%%
Let $\XS$ be a toric variety determined by a fan $\Sigma$ in $\R^m$.
\begin{enumerate}
%%(i)%%
\item[$\I$]
$\XS$ is compact if and only if $\R^m=\bigcup_{\sigma\in \Sigma}\sigma$.
%%(ii)%%
\item[$\II$]
$\XS$ is smooth if and only if every cone $\sigma\in \Sigma$ is smooth.
\qed
\end{enumerate}
%%%%%
\end{lmm}
%%%%%%%(End of Lemma 3.8)%%%
%%
%%
%%%%%%%%%%%%%%%%%%
%The following result was stated in \cite[Remark 5.11]{KY12} without its proof.
%Although we do not need this result,
%it may be interesting  and we give its proof.
%%%(Lemma 3.9)%%%%
\begin{lmm}\label{lmm: toric XS(n)}
%%%%%%%%%%%%%%
The space
 $\XS (n)$ is a non-singular toric variety associated to the fan $F_n(\Sigma)$.
 Moreover,
 there is an isomorphism $\XS (n)\cong X_{F_n(\Sigma)}$, and 
 $\KS (n)$ is the underlying simplicial complex of the fan $F_n(\Sigma)$.
 \end{lmm}
 \begin{proof}
 To see this, consider the toric variety $X(\Sigma)$ determined by the fan 
 $F_n(\Sigma)$. 
By considering the homogenous coordinate representation of $X(\Sigma)$,
we easily see that it is isomorphic to $\XS (n)$.
Moreover, one can easily show that $\XS (n)$ is non-singular
(by using Lemma \ref{lmm: toric}).
Thus, $\XS (n)$ is a non-singular toric variety associated to the fan
$F_n(\Sigma)$.
Moreover, by (\ref{eq: Fn}) we easily see that 
$\KS (n)$ is the underlying simplicial complex of $F_n(\Sigma)$.
%%%%
\end{proof}
%%(End of proof of Lemma 3.9)%%%
%%
%%
%%(Corollary 3.10)%%
\begin{crl}\label{crl: universal covering}
%%%%%%%%
Let $\Sigma$ be a fan in $\R^m$ and let
$\XS$ be a smooth toric variety such that the condition 
$($\ref{equ: condition dk}$)^*$ holds.
\par
%%(i)%%
$\I$
The map
%%(3.15)%%
%\begin{equation}
%%%%%%%
$
\Omega q_{n,\C}:\Omega\mathcal{Z}_{\KS}(\C^n,(\C^n)^*)
\stackrel{}{\longrightarrow}
\Omega\XS (n)
$
%%
%%%%%%%%%%%
%\end{equation}
%%%%%%%%%%%
is a universal covering with fiber $\Z^{r-m}.$
\par
%$\II$
%The map
%$q_{n,\R}:
%\mathcal{Z}_{\KS}(\R^n,(\R^n)^*)\to
%X_{\Sigma,\R}(n)$ is a covering projection with fiber $(\Z_2)^{r-m}$
%(up to homotopy equivalence).
%\par
$\II$
 The map
%%(3.16)%%
%\begin{equation}
$
\Omega q_{n,\R}:
\Omega\mathcal{Z}_{\KS}(\R^n,(\R^n)^*)
\stackrel{\simeq}{\longrightarrow}
\Omega X_{\Sigma,\R}(n)
$
%\end{equation}
%%
is a homotopy equivalence.
%%%
\par
$\III$
There is the following fibration sequence (up to homotopy)
%%(3.15)%%
\begin{equation}\label{eq: Segal fibration seq}
%%%%%%%
\T^{mn}_{\C}\stackrel{}{\longrightarrow} \XS (n)
\stackrel{}{\longrightarrow}
DJ(\KS (n)).
\end{equation}
\end{crl}
%%%%(Proof of Corollary 3.10)%%%
\begin{proof}
%%%%%%%% 
(i)
It follows easily from  Lemma \ref{lmm: principal}, that
the map 
$\Omega q_{n,\C}
%:\Omega \Q^{D,\Sigma}_n(\C)\to \Omega \XS (n)
$ is a covering projection with fiber $\Z^{r-m}$.
Since
$\Omega \Q^{D,\Sigma}_n(\C)$ is simply connected (by (i)), $\Omega q_{n,\C}$ is a
universal covering with fiber $\Z^{r-m}$.
\par
%(ii)
%By using Lemma \ref{lmm: principal}, we  see that 
%$q_{n,\R}$ is a fibration with fiber
%$G_{\Sigma,\R}$.
%However, since $G_{\Sigma,\R}\cong (\R^*)^{r-m}\simeq (\Z_2)^{r-m}$,
%$q_{n,\R}$ is a covering projection with fiber $(\Z_2)^{r-m}$ (up to homotopy).
%\par
(ii) The assertion (ii) easily follows from (iii) of Lemma \ref{lmm: principal}.
\par
(iii) 
The assertion (iii)
follows from  Lemmas \ref{lmm: simplicial complex of KS(n)}, \ref{lmm: toric XS(n)}
and
\cite[Proposition 4.4]{KY9}.
%%%
\end{proof}
%%(End of proof of Corollary 3.10)%%%%
%%
%%
%%
%%
%%
%%%(Lemma 3.11)%%%
\begin{lmm}[\cite{KY9}; Lemma 3.4]\label{lmm: XS}
%%%%%%%%%%%%%%%%%
If the condition $($\ref{equ: condition dk}$)^*$ is
satisfied, 
the space $\XS$ is simply connected and $\pi_2(\XS)=\Z^{r-m}$. 
\qed
%%%
\end{lmm}
%%%%%%%%%%%%%%%%%
%%(End of Lemma 3.11)%%%
%%
%%
%%
%%
%%
We end this section with a proof of Lemma \ref{lmm: Cox tric}.
%%%
%%%
%%%
%%%(Proof of Lemma 2.5)%%
\begin{proof}[Proof of Lemma \ref{lmm: Cox tric}]
%%%%%%%%%%%%%%%%
%%%%%%%%%%%%%%%%
Consider the map $F=(F_1,\cdots ,F_r)$ is given by
(\ref{eq: the map F}).
%%%
We let  $\K=\C$, as the proof for $\K=\R$ is completely analogous. 
It suffices to show that
$F(\lambda \textit{\textbf{x}})=F(\textit{\textbf{x}})$ up to $\GS$-action
for any $(\lambda ,\textit{\textbf{x}})\in \R^*\times (\R^{s+1}\setminus \{{\bf 0}_{s+1}\})$
iff $\sum_{k=1}^rd_k\textit{\textbf{n}}_k={\bf 0}_m$.
\par 
Since all homogenous polynomials $\{f_{k;i}\}_{k=1}^n$ have the same degree $d_i$,
for each $(\lambda,\textbf{\textit{x}})\in \R^*\times\R^{s+1}$,
%%%%
\begin{align*}
F_i(\lambda \textbf{\textit{x}})
&=(f_{1;i}(\lambda \textbf{\textit{x}}),\cdots ,f_{n;i}(\lambda \textbf{\textit{x}}))
=
(\lambda^{d_i}f_{1;i}(\textbf{\textit{x}}),\cdots ,\lambda^{d_i}f_{n;i}(\textbf{\textit{x}}))
\\
&=\lambda^{d_i}
(f_{1;i}(\textbf{\textit{x}}),\cdots ,f_{n;i}(\textbf{\textit{x}}))
=\lambda^{d_i}F_i(\textbf{\textit{x}}).
\end{align*}
%%%%
Thus, 
we have
%%%
\begin{align*}
%%%
F(\lambda \textit{\textbf{x}})
&=
(F_1(\lambda \textit{\textbf{x}}),\cdots ,F_r(\lambda \textit{\textbf{x}}))
=(\lambda^{d_1}F_1(\textit{\textbf{x}}),\cdots ,\lambda^{d_r}F_r(\textit{\textbf{x}}))
\\
&=
(\lambda^{d_1},\cdots ,\lambda^{d_r})\cdot
(F_1(\textit{\textbf{x}}),\cdots ,F_r(\textit{\textbf{x}}))
=(\lambda^{d_1},\cdots ,\lambda^{d_r})\cdot F(\textit{\textbf{x}}).
%%%
\end{align*}
%%%
Hence, it remains to show that
$(\lambda^{d_1},\cdots ,\lambda^{d_r})\in \GS$
for any $\lambda \in \R^*$
iff $\sum_{k=1}^rd_k\textit{\textbf{n}}_k={\bf 0}_m$.
However,
$(\lambda^{d_1},\cdots ,\lambda^{d_r})\in \GS$ 
for any $\lambda \in \R^*$ iff
%%%
%\begin{align*}
%%%
%&(\lambda^{d_1},\cdots ,\lambda^{d_r})\in \GS
%\Leftrightarrow
$$
\prod_{k=1}^r(\lambda^{d_k})^{\langle \textit{\textbf{n}}_k,\textit{\textbf{m}}\rangle}
=\lambda^{\langle \sum_{k=1}^rd_k\textit{\textbf{n}}_k,\textit{\textbf{m}}\rangle }
=1
\ \mbox{ for any }\textit{\textbf{m}}\in \Z^m
\ 
\Leftrightarrow
\ 
\sum_{k=1}^rd_k\textit{\textbf{n}}_k={\bf 0}_m
$$
%\\
%&
%\Leftrightarrow
%\sum_{k=1}^rd_k\textit{\textbf{n}}_k={\bf 0}_n
%%%%
%\end{align*}
%%%
and this completes the proof.
\end{proof}
%%%
%%
%%
%%(End of SECTION 3)%%
%%(End of SECTION 3)%%
%%(End of SECTION 3)%%
%%
%%
%%
%%
%%
%%
%%
%%%%%%%%%%%%%%%%%%%%%%%%%%%%%%%%%%%%
%%(SECTION 4: Vassiliev spectral sequence)%%%
%%%%%%%%%%%%%%%%%%%%%%%%%%%%%%%%%
\section{The
Vassiliev spectral sequence}\label{section: simplicial resolution}
%%%%%%%%%%%%%%%%
%%%
%%%
%%%
\paragraph{4.1 Simplicial resolutions.}
First, recall  the definitions of the non-degenerate simplicial resolution
and the associated truncated simplicial resolution 
(\cite{KY7}, \cite{Mo2},  \cite{Mo3}, \cite{Va}, \cite{Va2}).
%%%%%
%%
%%
%%(Definition 4.1)%%
\begin{dfn}\label{def: def}
%%%%%%%%%%%%%%%%%%%
{\rm
%%(i)%%
(i)
%%%%%%% 
For a finite set $\textbf{\textit{v}} =\{v_1,\cdots ,v_l\}\subset \R^N$,
let $\sigma (\textbf{\textit{v}})$ denote the convex hull spanned by 
$\textbf{\textit{v}}.$
%%%
Let $h:X\to Y$ be a surjective map such that
$h^{-1}(y)$ is a finite set for any $y\in Y$, and let
$i:X\to \R^N$ be an embedding.
Let  $\mathcal{X}^{\Delta}$  and $h^{\Delta}:{\mathcal{X}}^{\Delta}\to Y$ 
denote the space and the map
defined by
%%%
%%(4.1)%%%%%%%%
\begin{equation}
%%%%%%%%%%%%%%%
\mathcal{X}^{\Delta}=
\big\{(y,u)\in Y\times \R^N:
u\in \sigma (i(h^{-1}(y)))
\big\}\subset Y\times \R^N,
\ h^{\Delta}(y,u)=y.
%%%%%%%%%%%%%%
\end{equation}
%%%%%%%%%%%%%%%%
The pair $(\mathcal{X}^{\Delta},h^{\Delta})$ is called
{\it the simplicial resolution of }$(h,i)$.
In particular, it %$(\mathcal{X}^{\Delta},h^{\Delta})$
is called {\it a non-degenerate simplicial resolution} if for each $y\in Y$
any $k$ points of $i(h^{-1}(y))$ span $(k-1)$-dimensional simplex of $\R^N$.
%%%%
\par
%%(ii)%%
(ii)
%%%%%%%%
For each $k\geq 0$, let $\mathcal{X}^{\Delta}_k\subset \mathcal{X}^{\Delta}$ be the subspace
of
the union of the $(k-1)$-skeletons of the simplices over all the points $y$ in $Y$
given by
%%(4.2)%%
\begin{equation}
%%%%%%%%
\mathcal{X}_k^{\Delta}=\big\{(y,u)\in \mathcal{X}^{\Delta}:
u \in\sigma (\textbf{\textit{v}}),
\textbf{\textit{v}}=\{v_1,\cdots ,v_l\}\subset i(h^{-1}(y)),l\leq k\big\}.
\end{equation}
%%%%%%%%%%%
We make the identification $X=\mathcal{X}^{\Delta}_1$ by identifying 
 $x\in X$ with the pair
$(h(x),i(x))\in \mathcal{X}^{\Delta}_1$,
and we note that  there is an increasing filtration
%%%(4.3)%%
\begin{equation}\label{equ: filtration}
%%%
\emptyset =
\mathcal{X}^{\Delta}_0\subset X=\mathcal{X}^{\Delta}_1\subset \mathcal{X}^{\Delta}_2\subset
\cdots \subset \mathcal{X}^{\Delta}_k\subset 
%\mathcal{X}^{\Delta}_{k+1}\subset
\cdots \subset \bigcup_{k= 0}^{\infty}\mathcal{X}^{\Delta}_k=\mathcal{X}^{\Delta}.
%%%%%%%%%%%
\end{equation}
%%%%%%%%%%
%%
%%
%%
%%%%%%%%%
Since the map $h^{\Delta}:\mathcal{X}^{\Delta}\stackrel{}{\rightarrow}Y$
 is a proper map, it extends to the map
 ${h}_+^{\Delta}:\mathcal{X}^{\Delta}_+\stackrel{}{\rightarrow}Y_+$
 between the one-point compactifications,
 where $X_+$ denotes the one-point compactification of a locally compact space $X$.
%%%
%\par
%%%%(iii)%%
%(iii)
%%%%%%%%%%
%A space $X\subset \R^n$ is called  {\it semi-algebraic} if it is a subspace
%of the form
%$X=\bigcup_{i=1}^s\bigcap_{j=1}^{r_i}\{(x_1,\cdots ,x_n)\in\R^n:f_{ij}\ *_{ij}\  0\}$,
%where
%$f_{ij}\in \R [X_1,\cdots ,x_n]$ and
%$*_{ij}$ is either $<$ or $=$, for $i=1,\cdots s$ and
%$j=1,\cdots ,r_i$.
%\par
%Similarly, when $X\subset \R^n$ and $Y\subset \R^m$ are semi-algebraic spaces,
%a map $f:X\to Y$ is {\it a semi-algebraic map} if
%the graph of $f$ is semi-algebraic.
\qed
%%%%%%%%%%%%%%%%
}
%%%%%%%%%%%%%%%%
\end{dfn}
\begin{dfn}\label{def: 2.3}
%%%%%%%%%%%%%%%%%%%%%%
{\rm
Let $h:X\to Y$ be a surjective semi-algebraic map between semi-algebraic spaces, 
$j:X\to \R^N$ be a semi-algebraic embedding, and let
$(\mathcal{X}^{\Delta},h^{\Delta}:\mathcal{X}^{\Delta}\to Y)$
denote the associated non-degenerate  simplicial resolution of $(h,j)$. 
Then
for each positive integer $k\geq 1$, we denote by
%%%%%%%%%%%%
%Following  \cite{Mo3}, we call the map 
$h^{\Delta}_k:X^{\Delta}(k)\to Y$
{\it the truncated $($after the $k$-th term$)$  simplicial resolution} of $Y$ as in \cite{Mo3}.
Note that 
that there is a natural filtration
$$
 X^{\Delta}_0\subset
 X^{\Delta}_1\subset
%X^{\Delta}_2\subset
\cdots 
\subset X^{\Delta}_l\subset X^{\Delta}_{l+1}\subset \cdots
\subset  X^{\Delta}_k\subset X^{\Delta}_{k+1}
=X^{\Delta}_{k+2}
%=X^{\Delta}_{k+3}
=\cdots =X^{\Delta}(k),
$$
where $X^{\Delta}_0=\emptyset$,
$X^{\Delta}_l=\mathcal{X}^{\Delta}_l$ if $l\leq k$ and
$X^{\Delta}_l=X^{\Delta}(k)$ if $l>k$.
\qed
}
%%%%%%%%
\end{dfn}
%%%%(End of Definition 4.4-->4.2)%%%
%%
%%%(Lemma 4.5)%%%%%%%%%
%\begin{lmm}[\cite{Mo3}, c.f. Remark 2.4 and Lemma 2.5 in \cite{KY4}]\label{Lemma: truncated}
%%%%%%%%%%%%%%%%%%%%%%%%
%%%
%Under the same assumptions as in Definition \ref{def: 2.3}, the map
%$h^{\Delta}_k:X^{\Delta}(k)\stackrel{\simeq}{\longrightarrow} Y$ 
%is a homotopy equivalence.
%\qed
%\end{lmm}
%%%%%%%%%%%
%%(End of Lemma 4.5)%%%
%%%%%%%%%%%%%%%%%%%%%
%%%%%%%%%%%%%%%%%%%%%%%%%%%%%%%%%%%%%%%%%%%%%%%%%%%
%%%%(Vassiliev spectral sequences)%%%
\paragraph{4.2 Vassiliev spectral sequences.}
Next, 
%for $\K=\R$ or $\C$,
we shall construct the Vassiliev spectral sequence for computing the homology of
the space $\Q^{D,\Sigma}_n(\K)$.
%However, since the method of the case $\K=\R$ is completely analogous to that of the case
%$\K=\C$, we mainly explain only about the case $\K=\C$. 
%%
\par\vspace{2mm}\par
From now on, we always assume that $\Sigma$ is a fan in $\R^m$ such that
$\XS$ is simply connected toric variety satisfying the condition
$($\ref{equ: condition dk}$)^*$. 
Moreover,  let $D=(d_1,\cdots ,d_r)\in \N^r$ will always  be a fixed
$r$-tuple of positive integers.
%%%%%%%%%%%%%%%%%%%
%%%
%%
%%(Definition 4.6-->4.3)%%%%%%%
\begin{dfn}\label{Def: 3.1}
{\rm
%Let $\K=\R$ or $\C$ as before.
\par
%%(i)%%%
(i)
For each $d\in \N$, let $\P_d^{\K}\subset \K[z]$
denote the space of all monic polynomials
$f(z)=z^d+a_z^{d-1}+\cdots +a_d\in \K [z]$ of degree $d$.
Then for each $D=(d_1,\cdots ,d_r)\in \N^r$, let $\P_D^{\K}$ denote the space of
$r$-tuples of monic polynomials defined by
%%%(4.4)%%
\begin{equation}
\P_D^{\K}=\P_{d_1}^{\K}\times \P_{d_2}^{\K}\times
\cdots \times \P_{d_r}^{\K}.
\end{equation}
%%%%%%%%%
\par
(ii)
For each ${\rm f}=(f_1(z),\cdots ,f_r(z))\in \P_D^{\K}$,
let $F_{(n)}({\rm f})(z)$ denote the $rn$-tuple of monic polynomials
defined by
%%(4.5)%%
\begin{equation}
F_{(n)}({\rm f})(z)=(F_n(f_1)(z),\cdots ,F_n(f_r)(z))\in \K [z]^{rn},
\end{equation}
%%%%%%%
where  we denote by
$F_n(f_i)(z)$  the $n$-tuple of monic polynomials of degree $d_i$
given  by
%%(4.6)%%
\begin{equation}\label{eq: Fnfi}
%%%%%%
F_n(f_i)(z)=(f_i(z),f_i(z)+f^{\p}_i(z),f_i(z)+f^{\p\p}_i(z),\cdots ,f_i(z)+f^{(n-1)}_i(z))
\end{equation}
%%%%
for each $1\leq i\leq r$
(as in
(\ref{eq: Fn})).
\par
(iii)
%%%%%
Let $\Sigma_D$ denote {\it the discriminant} of $\Q^{D,\Sigma}_n(\K)$ in 
$\P_D^{\K}$ 
given by the complement
%%%%
\begin{align*}
%%%%
\Sigma_D&=\P_D^{\K} \setminus \Q^{D,\Sigma}_n(\K)
\\
&=
\{{\rm f}
=(f_1(z),\cdots ,f_{r}(z))
\in \P_D ^{\K}:
%(f_1(x),\cdots ,f_{r}(x))
F_{(n)}({\rm f})(x)
\in L_n^{\KS}(\K)
\mbox{ for some }x\in \R\},
%%%%%%%
\end{align*}
%%%%%%%
where 
$L_n^{\KS}(\K)$ denotes the set 
given by 
$K=\KS$ in (\ref{eq: L(Sigma)}).
%%%(4.7)%%
%\begin{align} 
%L_n^{\Sigma}(\K) &=
%\bigcup_{\sigma\in I(\mathcal{K}_{\Sigma})}L_{\sigma}(\K^n)
%=
%~\bigcup_{\sigma\subset[r],\sigma\notin K_{\Sigma}}
%L_{\sigma}(\K^n)
%\quad
%\mbox{as in (\ref{eq: L(Sigma)}).}
%\end{align}
%%
%as in (\ref{eq: L(Sigma)}).
%%%%%%
\par
%%(iv)%%
(iv)
%%%%%%%
Let  $Z_D\subset \Sigma_D\times \R$
denote %the subspace
{\it the tautological normalization} of 
 $\Sigma_D$ consisting of 
 all pairs 
$({\rm f},x)=((f_1(z),\ldots ,f_{r}(z)),
x)\in \Sigma_D\times\R$
satisfying the condition
$F_{(n)}({\rm f})(x)=(F_n(f_1)(x),\cdots ,F_n(f_{r})(x))\in L_n^{\KS}(\K)$.
%%%%
Projection on the first factor  gives a surjective map
$\pi_D :Z_D\to \Sigma_D$.
\qed
%%%%%%
}
%%%%%%
\end{dfn}
%%%%%%
%%(End of Definition 4.6-->4.3)%%%%
%%
%%
%%
%%
%%
%%
%%(Remark 4.7-->4.4)%%
\begin{rmk}
%%%%%%%%%%%%%%%
{\rm
Let $\sigma_k\in [r]$ for $k=1,2$.
It is easy to see that
$L_{\sigma_1}(\K^n)\subset L_{\sigma_2}(\K^n)$ if
$\sigma_1\supset \sigma_2$.
Letting
%%()%%
\begin{equation*}
%%% 
Pr(\Sigma)=\{\sigma =\{i_1,\cdots ,i_s\} \subset [r]:
\{\textit{\textbf{n}}_{i_1},\cdots ,\textit{\textbf{n}}_{i_s}\}
\mbox{ is a primitive collection}\},
\end{equation*}
%%%
we see that
%%%
%%(4.8)%%
%\begin{equation}
%%%%%%%%%
$\dis L_n^{\KS}(\K)=\bigcup_{\sigma\in Pr(\Sigma)}L_{\sigma}(\K^n)$,
%%%%%%%%
%\end{equation}
%%%%%%%% 
and by using (\ref{eq: rmin}) we obtain the equality 
%%%(4.7)%%
\begin{equation}\label{eq: dim rmin}
%%%
\dim L_n^{\KS}(\K)=
nd(\K)(r-\rmin (\Sigma))
=
\begin{cases}
2n(r-\rmin (\Sigma)) & \mbox{ if }\K=\C,
\\
n(r-\rmin (\Sigma)) & \mbox{ if }\K=\R.
\end{cases}
%%%%%%%%%%%%%%
\end{equation}
}
%%%
\end{rmk}
%%%(End of Remark 4.7-->4.4)%%%%%
%%
%%
%%
Our goal in this section is to construct, by means of the
{\it non-degenerate} simplicial resolution  of the discriminant, 
a spectral sequence converging to the homology of
$\Q^{D,\Sigma}_n(\K)$.
%%%%%
%%%%%
%%%%%
%%(Definition 4.8-->4.5)%%%%
\begin{dfn}\label{non-degenerate simp.}
%%%%%%%%%%%%%%%%%%%%%%
{\rm
%%%'i)%%%
(i)
%%%%%%%%%%%%%%%%%%
For an $r$-tuple $D=(d_1,\cdots ,d_r)\in \N^r$ of positive integers,
let $N(D)$ denote the positive integer given by
%%%%
%%%%
%%(4.8)%%
\begin{equation}\label{eq: ND}
%%%%%%%%
N(D)=\sum_{k=1}^r d_k.
\end{equation}
%%%%%
\par
%%%
%%%
%%%(ii)%%
(ii)
%%%%%%%%
For each based space $X$, let $F(X,d)$ denote 
{\it the ordered configuration space
of distinct $d$ points in $X$}
defined by
%%(4.9)%%
\begin{equation}
%%%%%%%%%
F(X,d)=\{(x_1,\cdots ,x_d)\in X^d:x_i\not= x_j\mbox{ if }i\not= j\}.
%%%%%%%%%%
\end{equation}
%%%%%%%%%
%%
%%
%%%%%%%%
Note that the symmetric group $S_d$ of $d$-letters acts on $F(X,d)$ freely by 
permuting coordinates.
Let $C_d(X)$ denote {\it the unordered configuration space of $d$-distinct
points in $X$} given by the orbit space
%%%(4.10)%%
\begin{equation}
%%%%%%%%%%
C_d(X)=F(X,d)/S_d.
\end{equation}
%%%
\par
%%%(iii)%%
(iii)
%%%%%%%%%
Let
 $L_{k;\Sigma,\K}\subset (\R\times L_n^{\KS}(\K))^k$ 
 denote the subspaces
defined by
%%%%%()
%\begin{equation*}\label{l}
$$
L_{k;\Sigma,\K}=
\{((x_1,s_1),\cdots ,(x_k,s_k))\in
(\R\times L_n^{\KS}(\K))^k:
x_i\not= x_j\mbox{ if }i\not= j\}.
%
%\begin{cases}
%L_{k;\Sigma}&=\{((x_1,s_1),\cdots ,(x_k,s_k))
%\in (\R\times L_n(\Sigma))^k: 
%%x_j\in \R,s_j\in L_n(\Sigma),
%x_i\not= x_j\mbox{ if }i\not= j\},
%\\
%L_{k;\Sigma,\R}&=\{((x_1,s_1),\cdots ,(x_k,s_k))
%\in (\R\times L_n(\Sigma,\R))^k: 
%%x_j\in \R,s_j\in L_n(\Sigma;\R),
%x_i\not= x_j\mbox{ if }i\not= j\}.
%\end{cases}
$$
%\end{equation*}
%%%%%%%%%
The symmetric group $S_k$ on $k$ letters  acts on the space $L_{k;\Sigma,\K}$ 
by permuting $k$-elements., and
let
 $C_{k;\Sigma,\K}$ denote the orbit space
defined by
%%%%%
%%%%(4.11)%%%%
\begin{equation}\label{Ck}
%%%%%%%%%%%%%
C_{k;\Sigma,\K}=L_{k;\Sigma,\K}/S_k.
%%%%%%%%%%%%%
\end{equation}
%%%%%%%%%%%%%%%
%%
%%
%%
Note that the space 
$C_{k;\Sigma,\K}$
is a cell-complex of  dimension (by (\ref{eq: dim rmin}))
%%%(4.12)%%
\begin{equation}\label{eq: dim CSigma}
%%%%%%%%%%
\dim C_{k;\Sigma,\K}=
\begin{cases}
k+2kn(r-\rmin (\Sigma)) & \mbox{ if }\K=\C,
\\
k+kn(r-\rmin (\Sigma)) & \mbox{ if }\K=\R.
\end{cases}
%2k(nr-nr_{\rm min}(\Sigma)).
%%%%%%%%%%
\end{equation}
%%%%
%%
%%%
\par
%%%(iv)%%%
(iv)
%%%%%%%%%
Let 
$(\mathcal{X}^D,{\pi}^{\Delta}_D:\mathcal{X}^D\to\Sigma_D)$ 
%and
%$(\tilde{\SZ}(d),\ ^{\p}\tilde{\pi}_d^{\Delta}:\SZ(d)\to\Sigma_d^*)$ 
be the non-degenerate simplicial resolution associated to the surjective map
$\pi_D:Z_D\to \Sigma_D$ 
with the natural increasing filtration as in Definition \ref{def: def},
$$
\emptyset =
\SZ_0
\subset \SZ_1\subset 
\SZ_2\subset \cdots
\subset 
\SZ=\bigcup_{k= 0}^{\infty}\SZ_k.
\quad
\qed
$$
}
\end{dfn}
%%%(End of Definition 4.8-->4.6)%%%%%%
%%%%%%%%
%%
%%
%%
%%
%%%%%%%%%%%%%%%%% 
By \cite[Lemma 1 (page 90)]{Va},
%Lemma \ref{lemma: simp},
the map
$\pi_D^{\Delta}$
%\SZ\stackrel{\simeq}{\rightarrow}\Sigma_D$
%is a homotopy equivalence which
extends to  a homology equivalence
%%%%%%%%%
$\pi_{D+}^{\Delta}:\SZ_+\stackrel{\simeq}{\rightarrow}{\Sigma_{D+}}.$
%%%%%
%where $X_+$ denotes the one-point compactification of a
%locally compact space $X$.
%%
%\par
%%%
Since
${\mathcal{X}_k^{D}}_+/{\SZ_{k-1}}_+
\cong (\SZ_k\setminus \SZ_{k-1})_+$,
we have a spectral sequence 
%%%%%%%%%%%
%%(4.13)%%%
\begin{equation}
%%%%%%%
\big\{E_{t;D}^{k,s},
d_t:E_{t;D}^{k,s}\to E_{t;D}^{k+t,s+1-t}
\big\}
\Rightarrow
H^{k+s}_c(\Sigma_D;\Z),
%%%
\end{equation}
%%%%%%%
where
$E_{1;D}^{k,s}=H^{k+s}_c(\SZ_k\setminus\SZ_{k-1};\Z)$ and
$H_c^k(X;\Z)$ denotes the cohomology group with compact supports given by 
$
H_c^k(X;\Z)= \tilde{H}^k(X_+;\Z).
$
%%%%%%%%%
\par
%%%
Since there is a homeomorphism
$\P_D^{\K}\cong \K^{N(D)}\cong \R^{d(\K)N(D)}$,
by Alexander duality  there is a natural
isomorphism
%%
%%%(4.14)%%%
\begin{equation}\label{Al}
%%%%%%%%%
\tilde{H}_k( \Q^{D,\Sigma}_n(\K);\Z)\cong
H_c^{d(\K)N(D)-k-1}(\Sigma_D;\Z)
\quad
\mbox{for any }k.
%%%%%%%%%%
\end{equation}
%%%%%%%%%%
%%%
%%%
%%%
By reindexing we obtain a
spectral sequence
%%
%%%(4.15)%%%
\begin{eqnarray}\label{SS}
%%%%%%%%%%%%%%%%%%%
&&
\big\{E^{t;D}_{k,s}, \tilde{d}^{t}:E^{t;D}_{k,s}\to E^{t;D}_{k+t,s+t-1}
\big\}
\Rightarrow H_{s-k}
( \Q^{D,\Sigma}_n(\K);\Z),
%%%%%%%%%%%%%%%%%%
\end{eqnarray}
%%%%%%%%%%%%%%%%
%%
%%
where
$E^{1;D}_{k,s}=
H^{d(\K)N(D)+k-s-1}_c(\SZ_k\setminus\SZ_{k-1};\Z).$
%%%%%%
%%%%
%\par\vspace{2mm}\par
%From now on, we assume that the following condition is satisfied:
%%%(4.18)%%
%\begin{equation}\label{eq: assumption: dmin}
%\dmin \geq n
%\ \Leftrightarrow \
%\lfloor \dmin/n \rfloor \geq 1.
%\end{equation}
%%%%
%%%%
%%%%
%%%%(Lemma 4.9->4.6)%%%%
\begin{lmm}\label{lemma: vector bundle*}
%%%%%%%%%%%%%%%%%%%%%
If $\dmin\geq n$ and 
 $1\leq k\leq \lfloor \frac{d_{\rm min}}{n}\rfloor$,
the space
$\SZ_k\setminus\SZ_{k-1}$
is homeomorphic to the total space of a real affine
bundle $\xi_{D,k,n}$ over $C_{k;\Sigma,\K}$ with rank 
$l_{D,k,n}=d(\K)(N(D)-nrk)+k-1$.
%%%%%%%%%%%%%%%%%%
\end{lmm}
%%%%%%%%%%%%%%%%%%%%%%%%%%%%%%%
%%%%%%%%(Proof of Lemma 4.9-->4.6)%%%
\begin{proof}
%%%%%%%%%%%%%%%%%%%%%%%%%%%%%%%
%%
Since the proof is completely analogous to that of \cite[Lemma 4.9]{KY12},
we omit detail of the proof.
\end{proof}
%%(End of proof of Lemma 4.6)%%%
%%%
%%%
%%%
%%%
%%%
%%%%%%(Lemma 4.10-->4.7)%%
\begin{lmm}\label{lemma: E11}
%%%%%%%%%%%%%%%%%%%
If $\dmin \geq n$ and
$1\leq k\leq  \lfloor \frac{d_{\rm min}}{n}\rfloor$, there is a natural isomorphism
$$
E^{1;D}_{k,s}\cong
H^{d(\K)nrk-s}_c(C_{k;\Sigma,\K};\pm \Z),
$$
where 
the twisted coefficients system $\pm \Z$  comes from
the Thom isomorphism.
\end{lmm}
%%%%
\begin{proof}
%%%(Proof of Lemma 4.10-->4.7)%%
Suppose that $1\leq k\leq \lfloor \frac{d_{\rm min}}{n}\rfloor$.
%%%
By Lemma \ref{lemma: vector bundle*}, there is a
homeomorphism
$
(\SZ_k\setminus\SZ_{k-1})_+\cong T(\xi_{D,k}),
$
where $T(\xi_{D,k,n})$ denotes the Thom space of
%one-point compactification of
$\xi_{D,k,n}$.
%%%%%%%
Since $(d(\K)N(D)+k-s-1)-l_{D,k,n}
=
d(\K)nrk-s,$
%$$
%%%%%
the Thom isomorphism gives a natural isomorphism 
%%%
$$
E^{1;d}_{k,s}
\cong 
\tilde{H}^{d(\K)N(D)+k-s-1}(T(\xi_{d,k,n});\Z)
\cong
H^{d(\K)nrk-s}_c(C_{k;\Sigma,\K};\pm \Z),
$$
and the assertion follows.
\end{proof}
%%%(End of proof of Lemma 4.10-->4.7)%%%%%%
%%
%%
%%
%%
%%
%%%%

%%%
%%%(SECTION 5)%%%%
\section{Stabilization maps}\label{section: stabilization map}
%%%%%%%%%%%%%%

We will now define two stabilization maps
%%(5.1)%%
\begin{equation}
\begin{cases}
s_{D,D+\textbf{\textit{a}}}:\Q^{D,\Sigma}_n(\C)\to
\Q^{D,D+\textbf{\textit{a}}}_n(\C) 
\\
s_{D,D+\textbf{\textit{a}}}^{\R}:\Q^{D,\Sigma}_n(\R)\to
\Q^{D,D+\textbf{\textit{a}}}_n(\R)
\end{cases}
\ \ 
\mbox{for each $\textbf{\textit{a}}\not= {\bf 0}_r\in (\Z_{\geq 0})^r$.}
\end{equation}
%for each $\textbf{\textit{a}}\not= {\bf 0}_r\in (\Z_{\geq 0})^r$.
%%
%%
%%%(Definition 5.1)%%%
\begin{dfn}\label{dfn: stab}
%%%%%%%%%%%%%%%%%%%%%%
{\rm
(i)
For  an $r$-tuple 
$D=(d_1,\cdots ,d_r)\in \N^r$,
let $U_D\subset \C$ denote the subspace defined by
%%(5.2)%%
\begin{equation}
%%%%%%%%%
U_D=\{w\in \C:\mbox{Re}(w)<N(D)\},
%%%%%%%%
\end{equation}
%%%%%%%%
%%
%%
%%
%%
and let $\varphi_D:\C\stackrel{\cong}{\longrightarrow}
U_D$ be any  homeomorphism (which we now fix)
satisfying the  following two conditions:
%%(5.3)%%%
\begin{equation} \label{eq: real condition}
%\begin{cases}
\varphi_D (\R)=(-\infty,N(D))\ \ \mbox{and}\ \ 
%\\
\varphi_D (\overline{\alpha})=\overline{\varphi_D (\alpha)}
%&
\quad 
\mbox{for any }\alpha\in {\rm H}_+,
%\end{cases}
\end{equation}
%%%
where ${\rm H}_+\subset \C$ denotes the upper half plane in $\C$ given by
%%(5.4)%%
\begin{equation}
{\rm H}_+=\{\alpha \in \C: \mbox{ Im }\alpha >0\}.
\end{equation}
\par
(ii)
Now let us choose and fix any
 $r$ points 
$(x_1,\cdots ,x_r)\in (\C\setminus U_D)^r$
satisfying the condition $x_i\not=x_j$ if $i\not= j$.
%%%%
\par
%%%(i)%%
%%%%%%%%
For each monic polynomial $f(z)=\prod_{k=1}^d(z-\alpha_k)\in \C [z]$
of degree $d$, let $\varphi_D(f)$
denote the monic polynomial of the same degree $d$ given by
%%(5.5)%%
\begin{equation}
%%%%
\varphi_D(f)=\prod_{k=1}^d(z-\varphi_D(\alpha_k)).
%%%
\end{equation}
%%%
\par
%%%%%%
%%%(ii)%%
(iii)
%%%%%%
For each $r$-tuple
$\textbf{\textit{a}}=(a_1,\cdots ,a_r)\not= {\bf 0}_r\in (\Z_{\geq 0})^r$,
%and
%$\sum_{k=1}^a_k\textbf{\textit{n}}_k={\bf 0}_r$,
define
the stabilization map
%%%%%%%%%
%%(5.6)%%
\begin{align}\label{def of stabilization}
s_{D,D+\textbf{\textit{a}}}:&
\Q^{D,\Sigma}_n(\C) \to \Q^{D+\textbf{\textit{a}},\Sigma}_n(\C)
\quad
\mbox{by}
\\
%\end{equation}
%%%%
%%(5.7)%%
%\begin{equation}
%%%%%%%%%%%
\nonumber
s_{D,D+\textbf{\textit{a}}}(f)&=
(\varphi_D(f_1)(z-x_1)^{a_1},\cdots ,\varphi_D(f_r)(z-x_r)^{a_r})
%%%%%%%%%%%
\end{align}
%%%%%%%%%%%
for $f=(f_1(z),\cdots ,f_r(z))\in \Q^{D,\Sigma}_n(\C)$.
}
%%%%%%
\end{dfn}
%%%(End of Definition 5.1)%%
%%
%%
%%
%%
%%(Remark 5.2)%%
\begin{rmk}
%%%%%%%%
{\rm
%%(i)%%
(i)
%%%%%%
Note that the definition of the map $s_{D,D+\textbf{\textit{a}}}$ depends
on the choice of the homeomorphism
$\varphi_D$ and the $r$-tuple $(x_1,\cdots ,x_r)\in (\C\setminus U_D)^r$ 
of points, but one can show that
the homotopy type of it does not depend on these choices.
%%%%%
%%%%%
%%%%%
\par
%%(ii)%%
(ii)
%%%%%%%
Let $\textbf{\textit{a}},\textbf{\textit{b}}\in (\Z_{\geq 0})^r$
be any two $r$-tuples such that
$\textbf{\textit{a}},\textbf{\textit{b}}\not={\bf 0}_r$.
Then it is easy to see that the equality
%%(5.7)%%
\begin{equation}\label{eq: stab-compo}
%%%%%%%%%%%%%%%%%%%
(s_{D+\textbf{\textit{a}},D+\textbf{\textit{a}}+\textbf{\textit{b}}})
\circ (s_{D,D+\textbf{\textit{a}}})
=s_{D,D+\textbf{\textit{a}}+\textbf{\textit{b}}}
\quad
(\mbox{up to homotopy)}
%%%%%%%
\end{equation}
%%%%%%%
holds.
Thus we mostly only consider the stabilization map $s_{D,D+\textbf{\textit{e}}_i}$
for each $1\leq i\leq r$, where
$\textbf{\textit{e}}_1=(1,0,\cdots ,0),
\textbf{\textit{e}}_2=(0,1,0,\cdots ,0),\cdots ,
\textbf{\textit{e}}_r=(0,0,\cdots ,0,1)\in \R^r$
denote the standard orthogonal basis of $\R^r$.
\par
(iii)
From (\ref{eq: real condition}) it  easily follows  that 
%%(5.8)%%
\begin{equation}\label{eq: varphi} 
\varphi_D(f)\in \R [z]
\quad
\mbox{ if }f=f(z)\in \R [z].
\end{equation}
%%%
Thus, 
for each $r$-tuple
$\textbf{\textit{a}}=(a_1,\cdots ,a_r)\not= {\bf 0}_r\in (\Z_{\geq 0})^r$,
one can easily show that the following holds:
%%(5.9)%%
\begin{equation}\label{eq: sd-R}
%%%%%%%
s_{D,D+\textbf{\textit{a}}}(\Q^{D,\Sigma}_n(\R))\subset
\Q^{D+\textbf{\textit{a}},\Sigma}_n(\R).
\end{equation}
}
\end{rmk}
%%(End of Remark 5.2)%%
%%
%%
%%
%%%(Definition 5.3)%%%
\begin{dfn}
{\rm
By (\ref{eq: sd-R}), one can define the stabilization map}
%%(5.11)%%
\begin{align}
s_{D,D+\textbf{\textit{a}}}^{\R}&:
\Q^{D,\Sigma}_n(\R) \to \Q^{D+\textbf{\textit{a}},\Sigma}_n(\R)
\quad
\mbox{\rm by the restriction}
%\end{equation}
%by the restriction}
%%%%
%%()%%
%\begin{equation}
%%%%%%%%%%%
\\
\nonumber
&\ \ 
s_{D,D+\textbf{\textit{a}}}^{\R}=s_{D,D+\textbf{\textit{a}}}\vert \Q^{D,\Sigma}_n(\R).
%%%%%%%%%%%
\end{align}
%%%%%%%%%%%
\end{dfn}
%%%(End of Definition 5.3)%%%
%%
%%
%%(Remark 5.4)%%
\begin{rmk}\label{rmk: equiv-remark}
{\rm
%Let $\Z_2=\{\pm 1\}$ denote the multiplicative cyclic group of order $2$.
%Note that the complex conjugation on $\C$ naturally induces a $\Z_2$-action
%on $\Q^{D,\Sigma}_n(\C)$, 
By using the definition of (\ref{def of stabilization}) and (\ref{eq: varphi})
we easily see that the following equality holds:
%$(\Q^{D,\Sigma}_n(\C))^{\Z_2}=\Q^{D,\Sigma}_n(\R),$
%and that
%t is also easy to see that
%%%(5.13)%%
%\begin{aequation}\label{eq: polZ2}
%(\Q^{D,\Sigma}_n(\C))^{\Z_2}=\Q^{D,\Sigma}_n(\R).
%\end{equation}
%%%
%It is also easy to see that 
%%(5.11)%%
\begin{equation}\label{eq: stabZ2}
s_{D,D+\textbf{\textit{a}}}^{\R}=(s_{D,D+\textbf{\textit{a}}})^{\Z_2}
\quad
\mbox{for each %$r$-tuple
$\textbf{\textit{a}}\not= {\bf 0}_r\in (\Z_{\geq 0})^r$.}
\ 
\qed
\end{equation}
%%%%%%
%for each $r$-tuple
%$\textbf{\textit{a}}=(a_1,\cdots ,a_r)\not= {\bf 0}_r\in (\Z_{\geq 0})^r$.
%This stabilization map $s_{D,D+\textbf{\textit{a}}}^{\R}$ will be used in \cite{KY16}.
%\qed
%%%%%
}
%%%%
%%%%%%%%
\end{rmk}
%%(End of Remark 5.4)%%%
%%
%%(End of SECTION 5)%%
%%(End of SECTION 5)%%
%%(End of SECTION 5)%%
%%
%%
%%
%%
%%(SECTION 6: Homology stability)%%
%%(SECTION 6: Homology stability)%%
%%(SECTION 6: Homology stability)%%
%%%%%%%%%%%%%%%%%%%%%%%%%%%%%%%%%%
\section{Homology stability}\label{section: homology stability}
%%%%%%%
%In this section, we always assume that $\dmin \geq n$, and 
We shall consider the homology stability of the space $\Q^{D,\Sigma}_n(\K)$.
%%
%%%%%%%%%%%%%%%%
\paragraph{6.1 The case $\K=\C$.}
%%%%%%%%%%%%%%%%
First, consider the case $\K=\C$.
Let $1\leq i\leq r$ and consider
the stabilization map
%%%(6.1)%%
\begin{equation}
%%%
s_{D,D+\textbf{\textit{e}}_i}:\Q^{D,\Sigma}_n(\C)
\to 
\Q^{D+\textbf{\textit{e}}_i,\Sigma}_n(\C).
\end{equation}
%%%
%%%
%%
It is easy to see that it extends to an open embedding
%%
%%%(6.2)%%
\begin{equation}\label{equ: sssd}
%%%%%%%%%
s_{D,i}:\C \times 
\Q^{D,\Sigma}_n(\C)
\to
\Q^{D+\textbf{\textit{e}}_i,\Sigma}_n(\C)
\end{equation}
%%%%%%%% 
%%%
It also naturally extends to an open embedding
%%()%%
%%%
$
\tilde{s}_{D,i}:\P_D^{\C}\to \P_{D+\textit{\textbf{e}}_i}^{\C}
$
%\end{equation}
and  by  restriction  we obtain an open embedding
%%%
%%%(6.3)%%
\begin{equation}\label{equ: open embedding}
%%%%%%%
\tilde{s}_{D,i}:\C\times \Sigma_D\to 
\Sigma_{D+\textit{\textbf{e}}_i}.
\end{equation}
%%%
Since one-point compactification is contravariant for open embeddings,
this map induces a map in the opposite direction
%%%(6.4)%%
\begin{equation}
\tilde{s}_{D,i +}:(\Sigma_{D+\textit{\textbf{e}}_i})_+
\to
(\C \times \Sigma_D)_+=S^{2}\wedge \Sigma_{D+}.
\end{equation}
%%%

We obtain the following commutative diagram
%%%
%%%%%(6.5)%%
\begin{equation}\label{eq: diagram}
%%%%%%%%%%%%%
\begin{CD}
\tilde{H}_k(\Q^{D,\Sigma}_n(\C);\Z) @>(s_{D,D+\textbf{\textit{e}}_i})_*>>
\tilde{H}_k(\Q^{D+\textit{\textbf{e}}_i,\Sigma}_n(\C);\Z)
\\
@V{AD_1}V{\cong}V @V{AD_2}V{\cong}V
\\
H^{2N(D)-k-1}_c(\Sigma_D;\Z)
@>(\tilde{s}_{D,i+})^{*}>>
H^{2N(D)-k+1}_c(\Sigma_{D+\textit{\textbf{e}}_i};\Z).
\end{CD}
\end{equation}
%%%%%%%%%%%%% 
Here,  $AD_k$ $(k=1,2)$ denote  the corresponding Alexander duality isomorphisms and
 ${\tilde{s}_{D,i+}}^{\ *}$ denotes the composite of 
%homomorphisms 
the suspension isomorphism with the homomorphism
${(\tilde{s}_{D+})^*}$ given by
%%(6.6)%%
\begin{equation}
%\begin{CD}
H^{M}_c(\Sigma_D;\Z)
\stackrel{\cong}{\rightarrow}
H^{M+2}_c
(\C\times \Sigma_D;\Z)
%\stackrel{(\tilde{s}_{d+})^*}{\longrightarrow}
\stackrel{(\tilde{s}_{D,i+})^*}{\longrightarrow}
H^{M+2}_c(\Sigma_{D+\textit{\textbf{e}}_i};\Z),
%\end{CD}
\end{equation}
where $M=2N(D)-k-1$.
%%%%%
\par
By the universality of the non-degenerate simplicial resolution
\cite{Mo2}, 
the map $\tilde{s}_{D,i}$ also naturally extends to a filtration preserving open embedding
%%%(6.7)%%%%%
\begin{equation}\label{equ: flitr-preserve map}
\tilde{s}_{D,i}:\C \times \SZ \to \mathcal{X}^{D+\textbf{\textit{e}}_i}
\end{equation}
%%%%%%%%%%%%%%
between non-degenerate simplicial resolutions.
This  induces a filtration preserving map
%%(6.8)%%
\begin{equation}
(\tilde{s}_{D,i})_+:
\mathcal{X}^{D+\textbf{\textit{e}}_i}_+\to 
(\C \times \SZ)_+
=S^{2}\wedge \SZ_+,
\end{equation}
%%%%%
and we finally obtain the homomorphism of spectral sequences
%%%(6.9)%%%%
\begin{align}\label{equ: theta1}
%%%%%%%%%%%%
&
\{ \tilde{\theta}_{k,s}^t:E^{t;D}_{k,s}\to 
E^{t;D+\textit{\textbf{a}}}_{k,s}\},
\quad
\mbox{where}
%%%%%%%
\\
\nonumber
&
\begin{cases}
\big\{E^{t;D}_{k,s}, \tilde{d}^{t}:
E^{t;D}_{k,s}\to E^{t;D}_{k+t,s+t-1}
\big\}
&\Rightarrow 
 H_{s-k}(\Q^{D,\Sigma}_n(\C);\Z),
\\
%%
%\nonumber
%&
\big\{E^{t;D+\textit{\textbf{e}}_i}_{k,s}, \tilde{d}^{t}:
E^{t;D+\textit{\textbf{e}}_i}_{k,s}\to 
E^{t;D+\textit{\textbf{e}}_i}_{k+t,s+t-1}
\big\}
&\Rightarrow 
 H_{s-k}(\Q^{D+\textit{\textbf{e}}_i}_n(\C);\Z),
 \end{cases}
%%
%\\
%\nonumber
%\mbox{and }&
%\ \mbox{ $E^1$-terms are given by}
\\
\nonumber
&
\begin{cases}
E^{1;D}_{k,s} &=
H_c^{2N(D)+k-1-s}(\SZ_k\setminus \SZ_{k-1};\Z),
\\
E^{1;D+\textit{\textbf{e}}_i}_{k,s}
&=
H_c^{2N(D)+k+1-s}
(\mathcal{X}_k^{D+\textit{\textbf{e}}_i}\setminus 
\mathcal{X}_{k-1}^{D+\textit{\textbf{e}}_i};\Z).
\end{cases}
\end{align}
%%%%%%%%%%%%%%%
%%
%%
%%
%%%(Lemma 6.1)%%%
\begin{lmm}\label{lmm: E1}
%%%%%%%%%%%%%%%%%
If $1\leq i\leq r$ and
$0\leq k\leq \lfloor \frac{d_{\rm min}}{n}\rfloor$, 
$\tilde{\theta}^1_{k,s}:E^{1;D}_{k,s}\to 
E^{1;D+\textit{\textbf{e}}_i}_{k,s}$ is
an isomorphism for any $s$.
\end{lmm}
%%%%%%%%%%%%%%%%
\begin{proof}
%%%%%%%%%%%%%%%%
Since the proof is completely analogous to that of \cite[Lemma 4.13]{KY12},
we omit the detail.
\end{proof}

Now we consider the spectral sequences induced by 
truncated simplicial resolutions.

%%%(Definition 6.2)%%%
\begin{dfn}
%%%%%%%%%%%%%%%%%%%%%%
{\rm
Let $X^{\Delta}$ denote the truncated 
(after the $\lfloor \frac{d_{\rm min}}{n}\rfloor$-th term) simplicial resolution of $\Sigma_D$
with the natural filtration
$$
\emptyset =X^{\Delta}_0\subset
X^{\Delta}_1\subset \cdots\subset
X^{\Delta}_{\lfloor \dmin/n\rfloor}\subset 
X^{\Delta}_{\lfloor \dmin/n\rfloor+1}=
X^{\Delta}_{\lfloor \dmin/n\rfloor +2}=\cdots =X^{\Delta},
$$
where $X^{\Delta}_k=\SZ_k$ if $k\leq \lfloor \frac{d_{\rm min}}{n}\rfloor$ and 
$X^{\Delta}_k=X^{\Delta}$ if $k\geq \lfloor \frac{d_{\rm min}}{n}\rfloor +1$.
%%%%%%
\par\vspace{1mm}\par
%%%%
Similarly,
let $Y^{\Delta}$ denote the truncated (after the $\lfloor \frac{d_{\rm min}}{n}\rfloor$-th term) simplicial resolution of 
$\Sigma_{D+\textit{\textbf{e}}_i}$
with the natural filtration
$$
\emptyset =Y^{\Delta}_0\subset
Y^{\Delta}_1\subset \cdots\subset
Y^{\Delta}_{\lfloor d_{\rm min}/n\rfloor}\subset 
Y^{\Delta}_{\lfloor d_{\rm min}/n\rfloor+1}=
Y^{\Delta}_{\lfloor d_{\rm min}/n\rfloor +2}=\cdots =Y^{\Delta},
$$
%%%%%%%
where $Y^{\Delta}_k=\mathcal{X}^{D+\textbf{\textit{e}}_i}_k$ if 
$k\leq \lfloor \frac{d_{\rm min}}{n}\rfloor$ and 
$Y^{\Delta}_k=Y^{\Delta}$ if $k\geq \lfloor \frac{d_{\rm min}}{n}\rfloor +1$.
%%%%%%%
}
%%%%%%%
\end{dfn}
%%%%%%(End of Definition 4.16)%%%%
%\par\vspace{3mm}\par
%%%%
By %\cite[\S 2]{Mo3}
%using Lemma \ref{Lemma: truncated} 
%and the same method
%as in 
\cite[\S 2 and \S 3]{Mo3}, %(cf. \cite[Lemma 2.2]{KY4}), 
we obtain the following {\it  truncated spectral sequences}
%%%%(6.10)%%%%%%%%
\begin{align}\label{equ: spectral sequ2}
%%%%%%%%%%%%%%%%
&
\begin{cases}
%%%
\dis
\big\{E^{t;\C}_{k,s}, d^{t}:E^{t;\C}_{k,s}\to 
E^{t;\C}_{k+t,s+t-1}
\big\}
&\ \Rightarrow H_{s-k}(\Q^{D,\Sigma}_n(\C);\Z),
%%%
\\
%%%
\dis
\big\{\ 
^{\p}E^{t;\C}_{k,s},\  d^{t}:\ ^{\p}E^{t;\C}_{k,s}\to 
\  ^{\p}E^{t}_{k+t,s+t-1}
\big\}
&\ \Rightarrow H_{s-k}(\Q^{D+\textit{\textbf{e}}_i}_n(\C);\Z),
\end{cases}
\\
\nonumber
& \quad \mbox{ where}
%%%
\\
%%(6.11)%%
\ \ & \label{eq:11}
%%%
\begin{cases}
E^{1;\C}_{k,s}&=\  H_c^{2N(D)+k-1-s}(X^{\Delta}_k\setminus X^{\Delta}_{k-1};\Z),
\\
^{\p}E^{1;\C}_{k,s}&=\  H_c^{2N(D)+k+1-s}(Y^{\Delta}_k\setminus Y^{\Delta}_{k-1};\Z).
\end{cases}
%%%%
\end{align}
%%%%
%where
%%%(6.11)%%
%\begin{equation}
%%%%
%\begin{cases}
%E^{1;\C}_{k,s}&=\  H_c^{2N(D)+k-1-s}(X^{\Delta}_k\setminus X^{\Delta}_{k-1};\Z),
%\\
%^{\p}E^{1;\C}_{k,s}&=\  H_c^{2N(D)+k+1-s}(Y^{\Delta}_k\setminus Y^{\Delta}_{k-1};\Z).
%\end{cases}
%\end{equation}
%%%%%
%%%
By the naturality of truncated simplicial resolutions,
the filtration preserving map
$\tilde{s}_{D,i}:\C \times \SZ \to \mathcal{X}^{D+\textbf{\textit{e}}_i}$
  gives rise to a natural filtration preserving map
%%()%%
%\begin{equation}
%%%%%%%%
$\tilde{s}_{D,i}^{\p}:\C \times X^{\Delta} \to Y^{\Delta}$
%\end{equation}
%%%%%%%% 
which, in a way analogous to  (\ref{equ: theta1}), induces
a homomorphism of spectral sequences 
%%%%%%%%%%%
%%%(6.12)%%%%
\begin{equation}\label{equ: theta2}
%%%%%%%%%%%%
\{ \theta_{k,s}^t:E^{t;\C}_{k,s}\to \ ^{\p}E^{t;\C}_{k,s}\}.
%%%%%%%%%%%
\end{equation}
%%%%%%%%%%
%%
%%
%%
%%%%(Lemma 6.3)%%
\begin{lmm}\label{lmm: Ed}
%%%%%%%%%%%%%%%%%
%%%%%%%%
\begin{enumerate}
%%(i)%%%
\item[$\I$]
%%%%%%%
If $k<0$ or $k\geq \lfloor \frac{\dmin}{n}\rfloor +2$,
$E^{1;\C}_{k,s}=\ ^{\p}E^{1;\C}_{k,s}=0$ for any $s$.
%%(ii)%%
\item[$\II$]
%%%%%%%%
$E^{1;\C}_{0,0}=\ ^{\p}E^{1;\C}_{0,0}=\Z$ 
and $E^{1;\C}_{0,s}=\ ^{\p}E^{1;\C}_{0,s}=0$ if $s\not= 0$.
%%(iii)%%%
\item[$\III$]
%%%%%%%%%
If $1\leq k\leq \lfloor \frac{d_{\rm min}}{n}\rfloor$, there are  
isomorphisms
$$
E^{1;\C}_{k,s}\cong \ ^{\p}E^{1;\C}_{k,s}\cong H^{2nrk-s}_c(C_{k;\Sigma};\pm \Z).
$$
%%(iv)%%
\item[$\IV$]
%%%%%%%%
If $1\leq k\leq \lfloor \frac{d_{\rm min}}{n}\rfloor$, 
$E^{1;\C}_{k,s}=\ ^{\p}E^{1;\C}_{k,s}=0$ for any 
$s\leq (2n\rmin (\Sigma)-1)k-1.$
%%(v)%%%
\item[$\V$]	
%%%%%%%%
If $k=\lfloor \frac{d_{\rm min}}{n}\rfloor +1$,
$E^{1;\C}_{k,s}=
\ ^{\p}E^{1;\C}_{k,s} =0
$ 
%$E^{1;\C}_{\lfloor d_{\rm min}/n\rfloor +1,s}=
%\ ^{\p}E^{1;\C}_{\lfloor d_{\rm min}/n\rfloor +1,s} =0
%$ 
for any 
$s\leq (2n\rmin (\Sigma)-1)\lfloor \frac{d_{\rm min}}{n}\rfloor-1$.
%%%%%%%%%%%
\end{enumerate}
%%%%%%
%%%%%
\end{lmm}
%%%%%%
\begin{proof}
%%%%%(Proof of Lemma 4.17)%%%%
Let us write $\rmin =\rmin (\Sigma)$
and $d_{\rm min}^{\p}=\lfloor \frac{\dmin}{n}\rfloor$.
Since the proofs of both cases are identical,  it suffices to prove the assertions for $E^{1;\C}_{k,s}$.
\par
(i), (ii), (iii):
Since $X^{\Delta}_k=X^{\Delta}$ for any $k\geq d_{\rm min}^{\p}+2$,
the assertions (i) and (ii) are clearly true.
Since $X^{\Delta}_k=\SZ_k$ for any $k\leq d_{\rm min}^{\p}$,
the assertion
(iii) easily follows from Lemma \ref{lemma: E11}.
\par
%%%(iv)%%%%
(iv)
Suppose that $1\leq k\leq d_{\rm min}^{\p}$. 
%%
%Since
%$\dim C_{k;\Sigma}=k+2nk(r-\rmin)$
By using the equality (\ref{eq: dim CSigma}),
%%%
$$
2nrk-s > \dim C_{k;\Sigma}
\
\Leftrightarrow \
s\leq (2n\rmin -1)k-1.
$$
Thus, the assertion (iv) follows from the isomorphism given by (iii).
%%%(v)%%%%%%
\par
(v)
%%%%
By Lemma \cite[Lemma 2.1]{Mo3}, we see that
%%%%%%%
\begin{align*}
%%%%%%
\dim (X^{\Delta}_{d_{\rm min}^{\p}+1}\setminus 
X^{\Delta}_{d_{\rm min}^{\p}})
&=
\dim (\SZ_{d_{\rm min}^{\p}}\setminus \SZ_{d_{\rm min}^{\p}-1})+1
=l_{D,d_{\rm min}^{\p},n}+\dim C_{d_{\rm min}^{\p};\Sigma}+1
\\
&=2N(D)+2d_{\rm min}^{\p}-2n\rmin d_{\rm min}^{\p}.
%\\
%&=2N(D)-2d_{\rm min}^{\p}(\rmin-1)
\end{align*}
%%%
Since
$E^{1;\C}_{d_{\rm min}^{\p}+1,s}=
H_c^{2N(D)+d_{\rm min}^{\p}-s}
(X^{\Delta}_{d_{\rm min}^{\p}+1}\setminus X^{\Delta}_{d_{\rm min}^{\p}};\Z)$
\ (by (\ref{eq:11}))
\ and
%%%
\begin{align*}
%%%%%
2N(D)+d_{\rm min}^{\p}-s
&>
\dim (X^{\Delta}_{d_{\rm min}^{\p}+1}\setminus X^{\Delta}_{d_{\rm min}^{\p}})
=2N(D)+2d_{\rm min}^{\p}-2n\rmin d_{\rm min}^{\p}
\\
%\Leftrightarrow
%2N(D)+d_{\rm min}^{\p}-s
%&>
%2N(D)+3d_{\rm min}^{\p}-2n\rmin d_{\rm min}^{\p}
\Leftrightarrow
s&<(2nr_{\rm min}-1)d_{\rm min}^{\p}
\Leftrightarrow
s\leqq (2nr_{\rm min}-1)d_{\rm min}^{\p}-1,
%%%%%%%
\end{align*}
%%%%%%
%$2N(D)+d_{\rm min}-s>\dim (X^{\Delta}_{d+1}\setminus X^{\Delta}_d)$
%$\Leftrightarrow$
%$s\leq (2\rmin -2)d_{\rm min}-1$,
we see that
$E^{1;\C}_{d_{\rm min}^{\p}+1,s}=0$ for any $s\leq (2n\rmin -1)d_{\rm min}^{\p}-1$.
%and the assertion (iv) follows.
%%%%
\end{proof}
%%%(End of proof of Lemma 4.17)%%%%%%%

%%%(Lemma 6.4)%%
\begin{lmm}\label{lmm: E2}
%%%%%%%%%%%%%%%%%
If $0\leq k\leq \lfloor \frac{d_{\rm min}}{n}\rfloor$, 
$\theta^1_{k,s}:E^{1;\C}_{k,s}\stackrel{\cong}{\longrightarrow} 
\ ^{\p}E^{1;\C}_{k,s}$ is
an isomorphism for any $s$.
%%%%%%%%%%%%%
\end{lmm}
%%%%%%%%%%%%%%%%
\begin{proof}
Since $(X^{\Delta}_k,Y^{\Delta}_k)=(\SZ_k,\mathcal{X}^{D+\textbf{\textit{e}}_i}_k)$ 
for $k\leq \lfloor \frac{\dmin}{n}\rfloor$,
the assertion follows from Lemma \ref{lmm: E1}.
\end{proof}
%%%%%%%%%%%%%%
%%%
%%%
%%%
%%%(Stability Theorem of stabilization maps)%%%%
%%%(Stability Theorem of stabilization maps)%%%%
%%%(Stability Theorem of stabilization maps)%%%%
%%%
%%%%%(Theorem 6.5: Theorem III)%%
\begin{thm}\label{thm: III}
%%%%%%%%%%%%%%%%%%%
%Let $n\geq 2$. Then
For each $1\leq i\leq r$,
the stabilization map
$$
s_{D,D+\textbf{\textit{e}}_i}:
\Q^{D,\Sigma}_n(\C)
\to 
\Q^{D+\textbf{\textit{e}}_i,\Sigma}_n(\C)
$$
is a homology equivalence through dimension
$d(D;\Sigma ,n,\C)$. 
%where
%$d(D;\Sigma ,n,\C)$ denotes the integer given by 
%%(\ref{eq: dDSigma}).
%$
%d(D;\Sigma ,n,\C)=(2n\rmin (\Sigma)-3)\lfloor \frac{d_{\rm min}}{n}\rfloor -2
%$
%as in (\ref{eq: dDSigma}).
%%%%%%%%%%%%%%%%%
\end{thm}
%%%%%%%%%%%%%%%%%%
\begin{proof}
%%%%%%%%%%%%%%%%%
%%%(Proof of Theorem 4.18)%%%%%
%%
%%
We write $\rmin =\rmin (\Sigma)$
and
$d_{\rm min}^{\p}=\lfloor \frac{d_{\rm min}}{n}\rfloor$
as in the proof of Lemma \ref{lmm: Ed}.
%If $d_{\rmin}<n$, $d(D,\Sigma,n,\C)=-2<0$ and the assertion is vacuous. 
Without loss of generality, we may assume that
$d_{\rm min}\geq n\geq 1$.
%Then note that
%$\Q^{D,\Sigma}_n$ and 
%$\Q^{D+\textbf{\textit{e}}_i,\Sigma}_n$ are simply connected
%by Lemma \ref{lmm: 1-connected}.
%Thus it suffices to prove that the map
%$s_{D,D+\textbf{\textit{e}}_i}$ is a homology equivalence
%through dimension $d(D;\Sigma,n)$.
\par
Let us consider the homomorphism
$\theta_{k,s}^t:E^{t;\C}_{k,s}\to \ ^{\p}E^{t;\C}_{k,s}$
of truncated spectral sequences given in (\ref{equ: theta2}).
%%%%
By using the commutative diagram (\ref{eq: diagram}) and the comparison theorem for spectral sequences, we see that 
it suffices to prove that the positive integer $d(D;\Sigma ,n,\C)$ 
has the 
following property:
%%%%%%%%%%%%%%%%%%
\begin{enumerate}
%%%%%%%%%%%%%%%%%%
\item[$(\dagger)$]
%%%%%%%%%%%%%%%%%
$\theta^{\infty}_{k,s}$
is  an isomorphism for all $(k,s)$ such that $s-k\leq d(D;\Sigma ,n,\C)$.
%%%%%%%%%%%%%%%%
\end{enumerate}
%%
%%
%%(The case r<0 or r \leq d+2)%%%%%
By Lemma \ref{lmm: Ed}, 
%$E^{1;\C}_{k,s}=\ ^{\p}E^{1;\C}_{k,s}=0$ if
%$k<0$, or if $k\geq \dmin^{\p} +2$, or if $k=\dmin^{\p} +1$ with 
%$s\leq (2n\rmin -1)\dmin^{\p} -1$.
%Since $\{(2n\rmin -1)\dmin^{\p} -1\}-(\dmin^{\p}+1)=(2n\rmin -2)\dmin^{\p} -2
%=d(D;\Sigma ,n,\C)$,
we  can easily see that:
%()%%
%%%%%%%
%%%%%%%
\begin{enumerate}
%%%%%%%
\item[$(\dagger)_1$]
if $k< 0$ or $k\geq \dmin^{\p} +1$,
$\theta^{\infty}_{k,s}$ is an isomorphism for all $(k,s)$ such that
$s-k\leq d(D;\Sigma ,n,\C)$.
\end{enumerate}
\par
Next,  assume that $0\leq k\leq \dmin^{\p}$, and investigate the condition that
$\theta^{\infty}_{k,s}$  is an isomorphism.
%%%
%Then
Note that the groups $E^{1;\C}_{k_1,s_1}$ and $^{\p}E^{1;\C}_{k_1,s_1}$ are not known for
%%(S_1)%%%
$(u,v)\in\mathcal{S}_1=
\{(\dmin^{\p}+1,s)\in\Z^2:s\geq (2n\rmin -1)\dmin^{\p} \}$.
%%%%
%%%%
By considering the differentials $d^1$'s of
$E^{1;\C}_{k,s}$ and $^{\p}E^{1;\C}_{k,s}$,
%$d^1:\E^1_{k,s}\to \E^{1}_{k+1,s}$
%and
%$d^1:\Ed^1_{k,s}\to \Ed^{1}_{k+1,s}$
and applying Lemma \ref{lmm: E2}, we see that
$\theta^2_{k,s}$ is an isomorphism if
$(k,s)\notin \mathcal{S}_1 \cup \mathcal{S}_2$, where
%%%
%%(S_2)%%%
$$
\mathcal{S}_2=
\{(u,v)\in\Z^2:(u+1,v)\in \mathcal{S}_1\}
=\{(\dmin^{\p} ,v)\in \Z^2:v\geq (2n\rmin -1)\dmin^{\p}\}.
$$
%%%%
%%%
%%%%
A similar argument  shows that
$\theta^3_{k,s}$ is an isomorphism if
%%(S_3)%%
$(k,s)\notin \bigcup_{t=1}^3\mathcal{S}_t$, where
$\mathcal{S}_3=\{(u,v)\in\Z^2:(u+2,v+1)\in \mathcal{S}_1\cup
\mathcal{S}_2\}.$
%%%%%
%\par
%%%%
Continuing in the same fashion,
considering the differentials
$d^t$'s on $E^{t;\C}_{k,s}$ and $^{\p}E^{t;\C}_{k,s}$
and applying the inductive hypothesis,
%Lemma \ref{lmm: E2}, 
we  see that $\theta^{\infty}_{k,s}$ is an isomorphism
if $\dis (k,s)\notin \mathcal{S}:=\bigcup_{t\geq 1}\mathcal{S}_t
=\bigcup_{t\geq 1}A_t$,
where  $A_t$ denotes the set
%%%%(Def. of A_t)%%%%%%%
$$
A_t=
\left\{
\begin{array}{c|l}
 &\mbox{ There are positive integers }l_1,\cdots ,l_t
\mbox{ such that},
\\
(u,v)\in \Z^2 &\  1\leq l_1<l_2<\cdots <l_t,\ 
u+\sum_{j=1}^tl_j=\dmin^{\p} +1,
\\
& \ v+\sum_{j=1}^t(l_j-1)\geq (2n\rmin -1)\dmin^{\p}
\end{array}
\right\}.
$$
%%%%%%%% 
Note that 
if this set was empty for every $t$, then, of course, the conclusion of 
Theorem \ref{thm: III} would hold in all dimensions (this is known to be false in general). 
%\par
If $\dis A_t\not= \emptyset$, it is easy to see that
%%%
\begin{align*}
a(t)&=
\min \{s-k:(k,s)\in A_t\}=
(2n\rmin -1)\dmin^{\p} -(\dmin^{\p} +1)+t
\\
&=
(2n\rmin -2)\dmin^{\p}+t-1
=d(D;\Sigma ,n,\C)+t+1.
\end{align*}
%%%
Hence, we obtain that
$\min \{a(t):t\geq 1,A_t\not=\emptyset\}=d(D;\Sigma ,n,\C)+2.$
%%%
 Since $\theta^{\infty}_{k,s}$ is an isomorphism
for any $(k,s)\notin \bigcup_{t\geq 1}A_t$ for each $0\leq k\leq \dmin^{\p}$,
we have the following:
%%(*2)%%%
\begin{enumerate}
%%%%%%%%
\item[$(\dagger)_2$]
If $0\leq k\leq \dmin^{\p}$,
$\theta^{\infty}_{k,s}$ is  an isomorphism for any $(k,s)$ such that
$s-k\leq  d(D;\Sigma ,n,\C)+1.$
\end{enumerate}
%%%%%%
Then, by $(\dagger)_1$ and $(\dagger)_2$, we know that
$\theta^{\infty;\C}_{k,s}:E^{\infty;\C}_{k,s}\stackrel{\cong}{\rightarrow}
\ ^{\p}E^{\infty;\C}_{k,s}$ 
is an isomorphism for any $(k,s)$
if $s-k\leq d(D;\Sigma ,n,\C)$. 
Hence, by $(\dagger)$  we have the desired
assertion and this completes the proof of Theorem \ref{thm: III}.
%%%%
%%%%
\end{proof}
%%(End of proof of Theorem 5.13)%%
%%
%%
%%
%%(Coroolary of stability Theorem)%%
%%(Coroolary of stability Theorem)%%
%%(Coroolary of stability Theorem)%%
%%%%%%%%%%%%%%%%%%%%%%%%%%%%%%%%%%%%
%%
%%
%%
%%%%%(Corollary 6.6: Corollary III)%%
%%%%%%%%%%%%%%%%%%%%%%%%%%%%%%%%%%%%
\begin{crl}\label{crl: III*}
%%%%%%%%%%%%%%%%%%%
For each $\textbf{\textit{a}}\not= {\bf 0}_r\in (\Z_{\geq 0})^r$, 
the stabilization map
$$
s_{D,D+\textbf{\textit{a}}}:
\Q^{D,\Sigma}_n(\C)\to 
\Q^{D+\textbf{\textit{a}},\Sigma}_n(\C)
$$
is a homology equivalence through dimension
$d(D;\Sigma ,n,\C)$. 
%%%%%%%%%%%%%%%%%
\end{crl}
%%%%%%%%%%%%%%%%%%
\begin{proof}
%%%%%%%
The assertion easily follows from
(\ref{eq: stab-compo}) and Theorem \ref{thm: III}.
%%%
\end{proof}
%%(End of proof of Corollary 6.6)%%%
%%
%%
%%
%%%%%%%%%%%%%
\paragraph{6.2 The case $\K=\R$.}
Next,  we shall consider the case $\K=\R$.
By using exactly the same approach as in 
Lemmas \ref{lemma: vector bundle*},
\ref{lemma: E11}, \ref{lmm: E1}, \ref{lmm: Ed}, \ref{lmm: E2},
Theorem \ref{thm: III}, and  Corollary \ref{crl: III*},
%%%
we can obtain the following result.
%%
%%
%%%(Lemma 6.7)%%%
\begin{lmm}\label{lmm: the case K=R}
%%%
There is the following
 truncated spectral sequence
%%(6.13)%%
\begin{equation}\label{eq: spectral sequence K=R}
\big\{E^{t;\R}_{k,s}, d^{t}:E^{t;\R}_{k,s}\to 
E^{t;\R}_{k+t,s+t-1}
\big\}
\ \Rightarrow H_{s-k}(\Q^{D,\Sigma}_n(\R);\Z)
\end{equation}
satisfying the following conditions:
\begin{enumerate}
%%(i)%%%
\item[$\I$]
%%%%%%%
If $k<0$ or $k\geq \lfloor \frac{\dmin}{n}\rfloor +2$,
$E^{1;\R}_{k,s}=0$ for any $s$.
%%(ii)%%
\item[$\II$]
%%%%%%%%
$E^{1;\R}_{0,0}=\Z$ and $E^{1;\R}_{0,s}=0$ if $s\not= 0$.
%%(iii)%%%
\item[$\III$]
%%%%%%%%%
If $1\leq k\leq \lfloor \frac{d_{\rm min}}{n}\rfloor$, there is a natural  
isomorphism
$$
E^{1;\R}_{k,s}\cong H^{nrk-s}_c(C_{k;\Sigma,\R};\pm \Z).
$$
%%(iv)%%
\item[$\IV$]
%%%%%%%%
If $1\leq k\leq \lfloor \frac{d_{\rm min}}{n}\rfloor$, 
$E^{1;\R}_{k,s}=0$ for any 
$s\leq (n\rmin (\Sigma)-1)k-1.$
%%(v)%%%
\item[$\V$]	
%%%%%%%%
If $k=\lfloor \frac{d_{\rm min}}{n}\rfloor +1$,
$E^{1;\R}_{k,s} =0$ 
for any 
$s\leq (n\rmin (\Sigma)-1)\lfloor \frac{d_{\rm min}}{n}\rfloor-1$.
\qed
%%%%%%%%%%%
\end{enumerate}
\end{lmm}
%%(End of Lemma 6.7)%%%
%%
%%
%%
%%(Theorem 6.8)%%
%%%%%(Theorem 6.8: Theorem III K=R)%%
\begin{thm}\label{thm: III (KY16)}
%%%%%%%%%%%%%%%%%%%
%Let $n\geq 2$. Then
For each $\textbf{\textit{a}}\not= {\bf 0}_r\in (\Z_{\geq 0})^r$,
the stabilization map
$$
s_{D,D+\textbf{\textit{a}}}^{\R}:
\Q^{D,\Sigma}_n(\R)\to 
\Q^{D+\textbf{\textit{a}},\Sigma}_n(\R)
$$
is a homology equivalence through dimension
$d(D;\Sigma ,n,\R)$, where
$d(D;\Sigma ,n,\R)$ denotes the integer given by (\ref{eq: dDSigma}). 
%%%(4.40)%%
%\begin{equation}\label{eq: sdRmin}
%%%%%%%%
%d(D;\Sigma ,n,\R)=(n\rmin (\Sigma)-2)\lfloor d_{\rm min}/n\rfloor -2.
%\end{equation}
%%%%%%%%%%%%%%%%%
\end{thm}
%%%%%%%%%%%%%%%%%%%
\begin{proof}
This assertion can be proved by using  the  spectral sequence
(\ref{eq: spectral sequence K=R})
in exactly the same way
as in the case of $\Q^{D,\Sigma}_n(\C)$, so we omit the details.
\end{proof}
%%%%%(End of proof of Theorem 6.8)%%
%%
%%
%%%%%(Corollary 6.9: Corollary III K=R)%%
\begin{crl}\label{crl: III* (KY16)}
%%%%%%%%%%%%%%%%%%%
%Let $n\geq 2$. Then
For each $\textbf{\textit{a}}\not= {\bf 0}_r\in (\Z_{\geq 0})^r$,
the stabilization map
$$
s_{D,D+\textbf{\textit{a}}}:
\Q^{D,\Sigma}_n(\C)\to 
\Q^{D+\textbf{\textit{a}},\Sigma}_n(\C)
$$
is a $\Z_2$-equivariant homology equivalence through dimension
$d(D;\Sigma ,n,\R).$
\end{crl}
%%%%%%%%%%%%%%%%%%
\begin{proof}
%%%
Since $d(D;\Sigma ,n,\R)<d(D;\Sigma ,n,\C)$,
the assertion follows from
(\ref{eq: stabZ2}),
Corollary \ref{crl: III*}
and Theorem \ref{thm: III (KY16)}.
%%%
\end{proof}
%%(End of proof of Corollary 6.9)%%
%%
%%
%%
%%
%%(End of SECTION 6)%%%%
%%
%%
%%(SECTION 7)%%%
\section{Connectivity}\label{section: connectivity}
%%%
%%%
%In this section we always assume that $\dmin \geq n$. 
%and
%study connectivity of the space $\Q^{D,\Sigma}_n(\K)$.

%%%
%\paragraph{7.1. The case $\K=\C$.}
%First, we shall consider the case $\K=\C$.
%$(n,\dmin,\rmin (\Sigma))\not=(1,1,2)$ or
%$\dmin \geq n\geq 2$.
%%
%%
%%(Lemma 7.1)%%
\begin{lmm}\label{lmm: pi1=abelian}
%%%%%%%%%%%%
$\I$ 
The space
$\Q^{D,\Sigma}_n(\C)$ is simply connected. 
\par
$\II$
%If $n\geq 2$ or $n=1$ and $\rmin(\Sigma)\geq 3$,
The space $\Q^{D,\Sigma}_n(\R)$  is simply connected
if $(n,\rmin(\Sigma))\not=(1,2)$.
%
%\par
%$\III$
%The space
%$\Q^{D,\Sigma}_n(\R)$ 
%is not path-connected if $(n,\rmin(\Sigma))=(1,2)$.
%But its each path-component is simply connected.
%The space
%$\Q^{D,\Sigma}_n(\R)$  is simply connected if $n\geq 2$ and if  $n= 1$ and $ \rmin (\Sigma)\geq 3$.
%
%$(n,\rmin (\Sigma))\not= (1,2).$
%%%%%%
\end{lmm}
%%%
\begin{proof}
%%%% %%
Note that an element of $\pi_1(\Q^{D,\Sigma}_n(\K))$ can be
represented by an $r$-tuple
$(\eta_1,\cdots ,\eta_r)$ of strings of $r$-different colors
where each $\eta_k$ $(1\leq k\leq r)$
has total multiplicity $d_k$, as in the case of strings representing elements of  the classical braid
group $\mbox{Br}_d=\pi_1(C_d(\C))$ \cite{Hansen}.
However, in our case an $r$-tuple $(\eta_1,\cdots ,\eta_r)$ of
strings
 of $r$-different colors can move continuously representing the same element of the fundamental group,\footnote{%
 %%%(Footnote 9)%%%
Let $f(z)\in \R[z]$ be a real coefficient polynomial and $\alpha\in \C\setminus \R$ be a
complex root of $f(z)$ of multiplicity $n_{\alpha}.$
Then $f(z)$ has the root $\overline{\alpha}$ of the same multiplicity $n_{\alpha}$. 
Thus, for the case $\K=\R$,
each string $\eta_k$ moves symmetrically  along the real axis $\R$.
%%%
}
%%%(end of Footnote 9)%% 
 as long as the following situation 
$(*)_{\sigma}$ does not occur for each $\sigma =\{i_1,\cdots ,i_s\}\in I(\KS)$:
%%%
\begin{enumerate}
\item[$(*)_{\sigma}$]
The strings $\{\eta_{i}\}_{i\in \sigma}$  of $s$-different colors 
with multiplicity $\geq n$
pass through a single point of the real line $\R$.
\end{enumerate}
%%%%%%
%\par
%Let $a,b\in \pi_1(\Q^{D,\Sigma}_n(\K))$ be any two elements, and suppose that
%elements $a$ and $b$ are represented by
%$r$-strings $(\eta_1,\cdots ,\eta_r)$ and $(\mu_1,\cdots ,\mu_r)$, respectively.
\par
%%%%
(i) In  the case $\K=\C$, 
%for  any $\sigma\in I(\KS)$ and $\{i,j\}\subset \sigma$,
we can continuously deform the strings $(\eta_1,\cdots ,\eta_r)$ and, if necessary, make them pass through one another 
  in  $\C\setminus \R$, 
  so that any collection of strings can be continuously deformed to a trivial one. 
Thus $\pi_1(\Q^{D,\Sigma}_n(\C))$  is trivial. 
\par
(ii)
Let $\K=\R$.
If $n\geq 2$, 
a similar argument as above shows that the fundamental group must be trivial, since any string of multiplicity $\geq n$ can be split into stings of multiplicity  less than $n$
(by the continuous deformation).
Thus, the space $\Q^{D,\Sigma}_n(\R)$ is path-connected and simply connected if $n\geq 2$.
\par
Next, consider the case $n=1$ with $\rmin(\Sigma)\geq 3$. 
%Clearly if $n\geq 2$, 
%a similar argument as above shows that the fundamental group must be trivial, since any string of multiplicity $n$ can be split into stings of multiplicity  less than $n$
%%and a string of multiplicity $n-1$ 
%and the split strings are now free to pass through any other strings.
%\par
%Next, consider the case $n=1$.  
%In the case  $\rmin (\Sigma)=2$,  note that the space $\Q^{D,\Sigma}_1(\R)$ 
%is not path connected.
%Indeed, if $\rmin (\Sigma)=2$,
%there has to exist  $\sigma\in I(\KS)$  such that   $\sigma =\{i,j\}$
%(by (\ref{eq: rmin min=mini sigma})).
% This means that particles on the real line corresponding to the $i$-th and the $j$-th polynomial cannot cross one another on the real line (the $i$-th and the $j$-th  polynomials cannot have common real roots).
% So $\Q^{D,\Sigma}_1(\R)$ is not path connected.\footnote{%
%%%%(Footnote 10)%%%%
%However, if $(n,\rmin(\Sigma))=(1,2)$, one can show that each component of $\Q^{D,\Sigma}_1(\R)$ is simply connected.
%Since we do not use this fact, we omit the detail of its proof.
%%since any $k$-fold real roots, for even $k$ can be split and the resulting roots moved off the real line and  they are free to pass through any other roots, and 
%%%Clearly, each connected component is  simply connected, as 
%%there are no restrictions on the movement of roots (particles) within a connected component. }
%}
%%%%(End of Footnote 10)%%%%%
%\par
%However,
%if $ \rmin (\Sigma)\geq 3$,  
Then the space $\Q^{D,\Sigma}_1(\R)$ is path-connected and  it is  simply connected. To see this, 
let $\sigma\in I(\KS)$  and  $\{i,j\}\subset \sigma$. Since $\mbox{card}(\sigma)\geq \rmin (\Sigma)\geq 3$
(by (\ref{eq: rmin min=mini sigma})), there is some number $k\in \sigma$ such that
$k\notin \{i,j\}$.  But this means that the $i$-th braid and the $j$-th braid can pass through one another on the real line, as long as they both don't pass through the $k$-th braid at the same time. 
By using this fact, we see that
any collection of strings can be continuously deformed to a trivial one.
Thus, $\Q^{D,\Sigma}_1(\R)$ is path connected and that 
$\pi_1(\Q^{D,\Sigma}_1(\R))$ is trivial. 
Since $n\geq 2$ or  $n=1$ with $\rmin(\Sigma)\geq 3$
$\Leftrightarrow$ $(n,\rmin(\Sigma))\not=(1,2)$,
the assertion (ii) follows.
%%%
\end{proof}
%%(End of proof of Lemma 7.1)%%
%%
%%%%(Remark 7.2)%%
\begin{rmk}\label{rmk: path-connected}
%%%
{\rm
The space
$\Q^{D,\Sigma}_n(\R)$ 
is not path-connected if $(n,\rmin(\Sigma))=(1,2)$.
But its each path-component is simply connected.
%}
%\end{rmk}
%%%%(End of Remark 7.2)%%
%\begin{proof}
\par
To see this,
suppose that 
$(n,\rmin(\Sigma))=(1,2)$.
Since $\rmin (\Sigma)=2$,
there has to exist  $\sigma\in I(\KS)$  such that   $\sigma =\{i,j\}$
(by (\ref{eq: rmin min=mini sigma})).
Since $n=1$,
 this means that particles on the real line corresponding to the $i$-th and the $j$-th polynomial cannot cross one another on the real line (i.e. the $i$-th and the $j$-th  polynomials cannot have common real roots).
 Thus, $\Q^{D,\Sigma}_1(\R)$ 
is not path-connected.
 However, since
%However, each path-connected component of  
%the space $\Q^{D,\Sigma}_1(\R)$ is  simply connected, as 
there are no restrictions on the movement of roots (particles) within a connected component, each path-component is simply connected.
\qed
}
\end{rmk}
%%%(End of Remark 7.2)%%%%
%%
%%
 %Since we do not need this fact, we omit the detail of the proof.
%\end{proof}
%%End of proof of Lemma 7.1)%%%
 % So $\Q^{D,\Sigma}_1(\R)$ is not path connected.\footnote{%
%%%%(Footnote 10)%%%%
%However, if $(n,\rmin(\Sigma))=(1,2)$, one can show that each component of $\Q^{D,\Sigma}_1(\R)$ is simply connected.
%Since we do not use this fact, we omit the detail of its proof.
%since any $k$-fold real roots, for even $k$ can be split and the resulting roots moved off the real line and  they are free to pass through any other roots, and 
%%Clearly, each connected component is  simply connected, as 
%there are no restrictions on the movement of roots (particles) within a connected component. }
%}

%%(Corollary 7.3)%%
\begin{crl}\label{crl: 1-connected}
%%%%%%%%%%%%
$\I$
If the condition $($\ref{1.4}$)^*$ holds,
%the space
$\Q^{D,\Sigma}_n(\C)$ is simply connected.
\par
$\II$
If the condition $($\ref{1.4}$)^{\dagger}$ holds,
%the space
$\Q^{D,\Sigma}_n(\R)$ is simply connected.
\end{crl}
%%%
\begin{proof}
The assertions follow from  Lemma \ref{lmm: pi1=abelian}.
%Corollary
%\ref{crl: connectedness K=R}. 
\end{proof}
%%(End of Proof of Corollary 7.3)%%%

%%
%%%%%(Lemma 7.4)%%
\begin{lmm}\label{lmm: vanishing line of E1}
%%%%%%
%%%%%
\begin{enumerate}
%%%%%%%%%%%%%%%%%
%%(a)%%
\item[$\I$]
%%%%%%%
If $k<0$, or $k\geq \lfloor \frac{d_{\rm min}}{n}\rfloor +2$, or
$k=0$ and $s\not= 0$,  $E^{1;\K}_{k,s}=0$.
%%%(b)%%
\item[$\II$]
%%%%%%%%
If $1\leq k\leq \lfloor \frac{d_{\rm min}}{n}\rfloor$ and
$s-k\leq (d(\K)n\rmin (\Sigma)-2)k -1,$
$E^{1;\K}_{k,s}=0$.
%%%(a)%%
%\begin{enumerate}
%\item[$\mbox{\rm (a)}$]
%If $s-k\leq (2n\rmin (\Sigma)-2)k -1,$  $E^1_{k,s}=0$.
%%%
%\item[$\mbox{\rm (b)}$]
%If $s-k\leq (n\rmin (\Sigma)-2)k -1,$  $E^{1,\R}_{k,s}=0$.
%%%%
%\end{enumerate}
%%%%%%
%%%%%%%%%%%%
%%%(b)%%
\item[$\III$]
%%%%%%%%
If $k=\lfloor \frac{d_{\rm min}}{n}\rfloor+1$ and
$s-k\leq (d(\K)n\rmin (\Sigma)-2)\lfloor \dmin/n\rfloor -2$,
 $E^{1;\K}_{k,s}=0$.
%%%%
%%
\end{enumerate}
%%%%
\end{lmm}
\begin{proof}
The assertions  follow from Lemmas \ref{lmm: Ed} and \ref{eq: spectral sequence K=R}.
\end{proof}
%%%(End of Lemma 7.4)%%

%%%%%
%%%%(Lemma 7.5)%%%%%
\begin{lmm}\label{lmm: vanishing line of homology}
%%%%%%%%%%%%%%%%%
$\I$
If  $\lfloor \frac{\dmin}{n}\rfloor \geq 2$,
%$n\geq 2$ and $\lfloor d_{\rm min}/n\rfloor \geq 2$, 
%or $n=1$ and $\dmin \geq 2$, 
$$
\tilde{H}_i(\Q^{D,\Sigma}_n(\K);\Z)=0
\quad
\mbox{ for any } i\leq d(\K)n\rmin (\Sigma) -3.
$$
\par
$\II$
If $\lfloor \frac{\dmin}{n}\rfloor =1$,
%If $n\geq 2$ and $\lfloor d_{\rm min}/n\rfloor=1$, 
%or $n=\dmin=1$,
% and $\rmin(\Sigma) \geq 3$, 
$$
\tilde{H}_i(\Q^{D,\Sigma}_n(\K);\Z)=0
\quad
\mbox{ for any }i\leq d(\K)n\rmin (\Sigma) -4.
$$
\end{lmm}
%%%%(Proof of Lemma 7.5)%%
\begin{proof}
%%%%%
Let us write $\dmin^{\p}=\lfloor \frac{\dmin}{n}\rfloor$.
Consider the spectral sequences 
(\ref{equ: spectral sequ2}) and
(\ref{eq: spectral sequence K=R}).
%%%%%%%%%%%%%%%%%%
Define the integer $a(k)$ by
%%()%%
\begin{equation*}
a(k)=(d(\K)n\rmin (\Sigma)-2)n_0(k)-\epsilon (k)
\quad
\mbox{for each $1\leq k\leq \dmin^{\p} +1$,}
\end{equation*}
%%%%
%for each integer $1\leq k\leq \dmin^{\p} +1$,
where $n_0(k)$ and $\epsilon (k)$ denote the integers given by
%%()%%
\begin{equation*}\label{eq: epk}
(n_0(k),\epsilon (k))=
\begin{cases}
(k,1) & \mbox{ if }1\leq k\leq  \dmin^{\p},
\\
(\dmin^{\p},2) & \mbox{ if }k= \dmin^{\p} +1.
\end{cases}
%%%
\end{equation*}
%%%
Then, by
Lemma \ref{lmm: vanishing line of E1}, 
we see that
$E^{1;\K}_{k,s}=0$ for any $(k,s)\not=(0,0)$ if
$s-k\leq m_0=\min\{a(k):1\leq k\leq  \dmin^{\p} +1\}$ is satisfied.
%%%
Hence, $\tilde{H}_k(\Q^{D,\Sigma}_n(\K);\Z)=0$
for any $k\leq m_0$.
%%%
However, we see that
%%()%%
\begin{align*}
m_0&=
\min\{a(k):1\leq k\leq  \dmin^{\p} +1\}
%\\
%&
=\min\{a(1),a(\dmin^{\p} +1) \}
\\
&=
\begin{cases}
d(\K)n\rmin (\Sigma)-3 & \mbox{ if }\dmin^{\p} \geq 2,
\\
d(\K)n\rmin (\Sigma)-4 & \mbox{ if } \dmin^{\p} =1.
\end{cases}
\end{align*}
%%%%%%%%%
Hence, we  obtain the assertions (i) and (ii).
%%%
\end{proof}
%%%%(End of proof of Lemma 7.5)%%%)
%%%%
%%
%
%%%(Corollary 7.6)%%%
\begin{crl}\label{crl: connectedness}
%%%%%%%%%%%%%%%%%%%%
%\begin{enumerate}
%%(1)%%
$\I$
%\item[$\I$]
If $n\geq 2$ and $\lfloor \frac{\dmin}{n}\rfloor \geq 2$,
%or $n=1$ and $\dmin \geq 2$.
%the space
$\Q^{D,\Sigma}_n(\C)$ is
$(2n\rmin (\Sigma)-3)$-connected.
%%(2)%%
\par
$\II$
%\item[$\II$] 
If $n\geq 2$ and $\lfloor \frac{\dmin}{n}\rfloor=1$, 
%the space
$\Q^{D,\Sigma}_n(\C)$ is
$(2n\rmin (\Sigma)-4)$-connected.
%%(3)%%
\par
$\III$
%\item[$\III$]
If $n=1$ and $\dmin \geq 2$, 
%the space
$\Q^{D,\Sigma}_n(\C)$ is
$(2\rmin (\Sigma)-3)$-connected.
\par
$\IV$
%%(4)%%
%\item[$\IV$]
Let $n=\dmin=1$.
Then 
%the space
$\Q^{D,\Sigma}_n(\C)$ is
$(2\rmin (\Sigma)-4)$-connected
if $\rmin (\Sigma)\geq 3$, and it is simply connected if
$\rmin (\Sigma)=2$.
\end{crl}
%%(Proof of Corollary 7.6)%%
\begin{proof}
%%%%%%%%
Since $\Q^{D,\Sigma}_n(\C)$ is simply connected
(by Lemma  \ref{lmm: pi1=abelian}),
the assertions  follow from the Hurewicz Theorem and Lemma 
\ref{lmm: vanishing line of homology}.
%%%
\end{proof}
%%(End of Proof of Corollary 7.6)%%

%%%(Corollary 7.7)%%
\begin{crl}\label{crl: connectedness K=R}
%%%%
%\begin{enumerate}
%%(1)%%
$\I$
%\item[$\I$]
If $n\geq 2$ and $\lfloor \frac{\dmin}{n}\rfloor \geq 2,$
%the space
$\Q^{D,\Sigma}_n(\R)$ is
$(n\rmin (\Sigma)-3)$-connected. 
%%(2)%%%
\par
$\II$
%\item[$\II$]
Let $n\geq 2$ and $\lfloor \frac{\dmin}{n}\rfloor =1$.
Then $\Q^{D,\Sigma}_n(\R)$ is
$(n\rmin (\Sigma)-4)$-connected if  $n\rmin(\Sigma)\geq 5,$
and it is simply connected if $n=\rmin(\Sigma)= 2$.
%%(3)%%
\par
$\III$
%\item[$\IV$] 
Let $n=1$, $\dmin \geq 2$.
Then
$\Q^{D,\Sigma}_n(\R)$ is
$(\rmin (\Sigma)-3)$-connected if $\rmin(\Sigma)\geq 4$,
and it is simply connected if $\rmin(\Sigma)=3$.
%
%If
%$n=1$, $\dmin \geq 2$ 
%and $\rmin (\Sigma)\geq 4,$
%%the space
%$\Q^{D,\Sigma}_n(\R)$ is
%$(\rmin (\Sigma)-3)$-connected.
%%(5)%%
\par
$\IV$
%\item[$\V$] 
Let $n=\dmin =1$.
Then
$\Q^{D,\Sigma}_n(\R)$ is
$(\rmin (\Sigma)-4)$-connected if $\rmin(\Sigma)\geq 5$,
and it is simply connected if $\rmin(\Sigma)=3$ or $4$.
%
%If
%$n=\dmin =1$ 
%and $\rmin (\Sigma)\geq 5,$
%%the space
%$\Q^{D,\Sigma}_n(\R)$ is
%$(\rmin (\Sigma)-4)$-connected.
%\end{enumerate}
%
\end{crl}
%%%(End of Corollary 7.7)%%
%%
%%(Proof of Corollary 7.7)%%
\begin{proof}
%%%%%%%%%%%%%%%%%
If $(n,\rmin(\Sigma))\not=(1,2)$,
%$n\geq 2$ or $\rmin(\Sigma)\geq 3$, 
the space $\Q^{D,\Sigma}_n(\R)$ is simply connected
(by  Lemma \ref{lmm: pi1=abelian}).
Thus
the assertions easily  follow from 
he Hurewicz Theorem and Lemma 
\ref{lmm: vanishing line of homology}.
%%It remains to prove the case $n=1$.
%If $n=1$,
%by  Lemma \ref{lmm: pi1=abelian} and the Hurewicz Theorem, there is an isomorphism
%%%(7.1)%%
%\begin{equation}
%\pi_1(\Q^{D,\Sigma}_1(\R))\cong H_1(\Q^{D,\Sigma}_1(\R);\Z)
%\quad
%\mbox{ if }\rmin (\Sigma)\geq 3.
%\end{equation}
%Thus the assertions (iv) and (v) easily  follow from
%%the Hurewicz Theorem and
%Lemma  \ref{lmm: vanishing line of homology}.
\end{proof}
%%%(End of Proof of Corollary 7.7)%%
%%%
%%%
%%%
%%%
%%%
%%
%%%(Corollary 7.8)%%%%%%%%%%%%%%%
\begin{crl}\label{crl: III}\label{crl: III+1connected}
%%%%%%%%%%%%%%%%%%%
Let $\textbf{\textit{a}}\not= {\bf 0}_r\in (\Z_{\geq 0})^r$.
\par
$\I$
If the condition $($\ref{1.4}$)^*$ holds, 
the stabilization map
$$
s_{D,D+\textbf{\textit{a}}}:
\Q^{D,\Sigma}_n(\C)\to 
\Q^{D+\textbf{\textit{a}},\Sigma}_n(\C)
$$
is a homotopy equivalence through dimension
$d(D;\Sigma ,n,\C)$.
\par
$\II$
If the condition $($\ref{1.4}$)^{\dagger}$ holds, 
the stabilization map
$$
s_{D,D+\textbf{\textit{a}}}^{\R}:
\Q^{D,\Sigma}_n(\R)\to 
\Q^{D+\textbf{\textit{a}},\Sigma}_n(\R)
$$
is a homotopy equivalence through dimension
$d(D;\Sigma ,n,\R)$.
\par
$\III$
%%%%%
If the condition $($\ref{1.4}$)^{\dagger}$ holds, 
the stabilization map
$$
s_{D,D+\textbf{\textit{a}}}^{}:
\Q^{D,\Sigma}_n(\C)\to 
\Q^{D+\textbf{\textit{a}},\Sigma}_n(\C)
$$
is a $\Z_2$-equivariant homotopy equivalence through dimension
$d(D;\Sigma ,n,\R)$.
%%%%%%%%%%%%%%%%%
\end{crl}
%%%%%%%%%%%%%%%%%%
\begin{proof}
%%%%%%%%
The assertions  (i) and (ii) follow from
Theorem \ref{thm: III (KY16)}, 
Corollaries \ref{crl: III*} and \ref{crl: 1-connected}.
Since $d(D;\Sigma ,n,\R)<d(D;\Sigma ,n,\C)$ and
$(s_{D,D+\textbf{\textit{a}}})^{\Z_2}=s_{D,D+\textbf{\textit{a}}}^{\R}$,
the assertion (iii) follows from (i) and (ii).
\end{proof}
%%(End of proof of Corollary 7.8)%%
%%
%%%
%%%(End of SECTION 7)%%%
%%%(End of SECTION 7)%%%
%%%(End of SECTION 7)%%%
%%%
%%%
%%%
%%%
%%%
%%%
%%%
%%%(SECTION 8: Scanning maps)%%%
%%%(SECTION 8: Scanning maps)%%%%%%%
%%%(SECTION 8: Scanning maps)%%%%%%%
%%%%%%%%%%%%%%%%%%%%%%%%%
\section{Scanning maps}\label{section: scanning maps}
%%%%%%%%%%%%%%%%%%%%%%%%%

In this section we study about the configuration space model of
$\Q^{D,\Sigma}_n(\K)$ and the corresponding scanning map.
%%%
\par
\paragraph{8.1 Configuration space models.}
%%%
First, consider about the configuration space model of $\Q^{D,\Sigma}_n(\K)$.
%%%%
%%(Definition 8.1)%%%
\begin{dfn}\label{dfn: SP}
%%%%%%%%%%%%%%%%%%%%%
{\rm
 For a positive integer $d\geq 1$ and a based space $X,$ let
 $\SP^d(X)$ denote {\it the $d$-th symmetric product} of $X$ defined as
the orbit space 
%%%%%%%%
%%(8.1)%%
\begin{equation}
%%%%%%%%%%
\SP^d(X)=X^d/S_d,
\end{equation}
where the symmetric group $S_d$ of $d$ letters acts on the $d$-fold
product $X^d$ in the natural manner.
}
\end{dfn}
%%%%%(End of Definition 8.1)%%
%%%
%%
%%

%%%%%%%%(Remark 8.2)%%%%%
\begin{rmk}
%%%%%%%%%%%%%%%%%%%%%%%%%
{\rm
%%(i)%%
(i)
%%%%%%
Note that
an element $\eta\in \SP^d(X)$ may be identified with a formal linear
combination
%%(8.2)%%%
\begin{equation}
%%%
\eta =\sum_{k=1}^sn_kx_k,
\end{equation}
%%%%%%%%%%
where
$\{x_k\}_{k=1}^s\in C_s(X)$ 
and $\{n_k\}_{k=1}^s\subset \N$ with
$\sum_{k=1}^sn_k=d$.
In this situation we shall refer to $\eta$ as configuration (or divisor) of points, the point
$x_k\in X$ having a multiplicity $n_k$.
%%%
\par
(ii)
For example, when $X=\C$, we have the natural homeomorphism
%%%(8.3)%%
\begin{equation}\label{eq: identification}
%%%
\psi_d:
\P_d^{\C}\stackrel{\cong}{\longrightarrow}\SP^d(\C)
\end{equation}
%%%
given by using the above identification
%%%(8.4)%%%%
\begin{equation}\label{eq:homeo}
%%%%
\psi_d(f(z))=
%\psi_d(\prod_{k=1}^s(z-\alpha_k)^{d_k})= 
\sum_{k=1}^sd_k\alpha_k
\qquad
\mbox{for $f(z)=\prod_{k=1}^s(z-\alpha_k)^{d_k}\in \P_d^{\C}$.}
\quad
\qed
\end{equation}
%%%%%%%
%for $f(z)=\prod_{k=1}^s(z-\alpha_k)^{d_k}\in \P_d^{\C}$.
%\qed
}
%%%%%%%%%
\end{rmk}
%%%%%(End of Remark 8.2)%%%%%%%%%

%%%%%(Definition 8.3)%%%%%%%%%
\begin{dfn}
%%%%%%%%%%%%%%%%%%%%%%%%%%%%%
{\rm
(i) For a subspace $A\subset X$, let $\SP^d(X,A)$ denote the quotient space
%%(8.4)%%
\begin{equation}
%%%%%
\SP^d(X,A)=\SP^d(X)/\sim
%%%
\end{equation}
%%%%%
where the equivalence relation $\sim$ is defined by
%%(8.5)%%%
\begin{equation}
\xi\sim \eta \Leftrightarrow
\xi \cap (X\setminus A)=\eta \cap (X\setminus A)
\quad
\mbox{ for }\xi,\eta\in \SP^d(X).
\end{equation}
%%%
Thus, the points of $A$ are ignored.
When $A\not=\emptyset$, by adding a point in $A$ we have the natural
inclusion
$
\SP^d(X,A)\subset \SP^{d+1}(X,A).
$
Thus, when $A\not=\emptyset$, one can define the space
$\SP^{\infty}(X,A)$ by the union
%%(8.6)%%
\begin{equation}
%%%%%%
\SP^{\infty}(X,A)=\bigcup_{d\geq 0}\SP^d(X,A),
%%%
\end{equation}
%%%%%
where we set $\SP^0(X,A)=\{\emptyset\}$ and
$\emptyset$ denotes the empty configuration.
\par
(ii)
From now on, we always assume that $X\subset \C$.
For each $r$-tuple $D=(d_1,\cdots ,d_r)\in \N^r$,
let $\SP^D(X)=\prod_{i=1}^r\SP^{d_i}(X)$, and
define the space
$\QQ_{D,n}(X)$ 
%and $E^{\Sigma}_{D,n}(X)$  
by
%%(8.7)%%%%%%
\begin{align}
%%%%%%%%% 
\QQ_{D,n}(X)
&=
\{(\xi_1,\cdots ,\xi_r)
\in \SP^D(X):\mbox{the condition $(*)_{n}^{\Sigma}$ holds}\},
%\prod_{i=1}^r\SP^{d_i}(X)
%: (\bigcap_{k\in \sigma}\xi_k)\cap \R=\emptyset 
%\ \mbox{for any }\sigma\in I(\KS)\},
%(*)_{\Sigma ,\R}\},
%\\
%E^{\Sigma}_{D,n}(X)
%&=
%\{(\xi_1,\cdots ,\xi_r)
%\in \SP^D(X):
%\mbox{the condition $(*)_{\Sigma,n}$ holds}\},
%\prod_{i=1}^r\SP^{d_i}(X)
%: \ \bigcap_{k\in \sigma}\xi_k=\emptyset
%\ \mbox{for any }\sigma\in I(\KS)\},
%(*)_{\Sigma}\},
%\end{cases}
\end{align}
%%%%%
where the condition $(*)_{n}^{\Sigma}$ 
%and $(*)_{\Sigma,n}$
is given by
\begin{enumerate}
\item[$(*)_{n}^{\Sigma}:$]
The configuration $(\bigcap_{k\in \sigma}\xi_k)\cap \R$
 contains no point $x\in X$ of multiplicity $\geq n$
 for any $\sigma \in I(\KS)$.
%%%
%\item[$(*)_{\Sigma,n}:$]
%The configuration $\bigcap_{k\in \sigma}\xi_k$
% contains no point $x\in X$ of multiplicity $\geq n$
% for any $\sigma \in I(\KS)$.
%%%
\end{enumerate}
%%%%%%
\par
(iii)
When $A\subset X$ is a subspace, define
an equivalence relation \lq\lq$\sim$\rq\rq \ on
the space $\QQ_{D,n}(X)$ 
%(resp. the space $E^{\Sigma}_{D,n}(X)$) 
by
$$
(\xi_1,\cdots ,\xi_r)\sim
(\eta_1,\cdots ,\eta_r)
\quad
\mbox{if }\quad
\xi_i \cap (X\setminus A)=\eta_i \cap (X\setminus A)
%\quad
%\mbox{for each }1\leq j\leq r.
$$
for each $1\leq j\leq r$.
Let 
$\QQ_{D,n}(X,A)$ 
%and $E^{\Sigma}_{D,n}(X,A)$ 
be the quotient space
defined by
%%(8.8)%%
\begin{equation}
\QQ_{D,n}(X,A)=\QQ_{D,n}(X)/\sim .
%\mbox{ and }\ 
%E^{\Sigma}_{D,n}(X,A)=E^{\Sigma}_{D,n}(X)/\sim .
\end{equation}
%%%%
When $A\not=\emptyset$,  by
adding points in $A$ we have  natural inclusion
%%(8.9)%%
\begin{equation}\label{eq: D+k}
%%%
\QQ_{D,n}(X,A)\subset \ \QQ_{D+\e_i,n}(X,A)
\quad
\mbox{ for each }1\leq i\leq r,
%\quad
%E^{\Sigma}_{D,n}(X,A)\subset \ E^{\Sigma}_{D+\e_i,n}(X,A)
%%
\end{equation}
%%
%for each $1\leq i\leq r,$
where
$D+\e_i=(d_1,\cdots,d_{i-1},d_i+1,d_{i+1},\cdots ,d_r)$.
\par
Thus,  when $A\not=\emptyset$, one can define the space
$\QQ_n (X,A)$ 
%and
%$E^{\Sigma}_{n}(X,A)$ 
by the union
%%(8.10)%%
\begin{equation}
%%%
%\begin{cases}
\QQ_{n}(X,A) =\bigcup_{D\in \N^r}
\QQ_{D,n}(X,A),
%\quad
%E^{\Sigma}_{n}(X,A) =\bigcup_{D\in \N^r}
%E^{\Sigma}_{D,n}(X,A),
%\end{cases}
\end{equation}
%%%%%%%%%%%
where the empty configuration
$(\emptyset ,\cdots ,\emptyset)$ is the base-point of
$\QQ_n (X,A)$.
% or $E_n^{\Sigma}(X,A)$.
%\qed
}
%%%%
\end{dfn}
%%(End of Definition 8.3)%%%
%%
%%
%%
%%
%%%
%%%
%%%
%%%
%%(Remark 8.4)%%
\begin{rmk}\label{rmk: ESigma}
%%%%
{\rm
%(i)
%If $X\subset \R$,  the following equalities hold:
%%%()%%%%%
%\begin{equation*}
%%%%%%%%%%%
%\QQ_{D.n} (X)=E^{\Sigma}_{D,n}(X)
%\quad
%\mbox{and}
%\quad
%\QQ_n (X,A)=E^{\Sigma}_n(X,A).
%\end{equation*}
%%%%%%
%\par
(i)
Let $D=(d_1,\cdots ,d_r)\in \N^r$.
Then 
by using the identification
(\ref{eq: identification})
we easily obtain the homeomorphism 
%%(8.11)%%
\begin{equation}\label{eq: Pol=E}
%%%
\begin{CD}
\Q^{D,\Sigma}_n(\C) @>\Psi_D>\cong> \QQ_{D,n}(\C)
\\
(f_1(z),\cdots ,f_r(z))
@>>>
(\psi_{d_1}(f_1(z)),\cdots ,\psi_{d_r}(f_r(z)))
\end{CD}
\end{equation}
%%%%%%
\par
(ii)
Now 
let $\varphi_D:\C \stackrel{\cong}{\longrightarrow}U_D$
and
$\textbf{\textit{x}}_D=(x_{D,1},\cdots ,x_{D,r})\in F(\C\setminus \overline{U_D},r)$ 
be the homeomorphism and the point for
defining the stabilization map $s_{D}$
given
in Definition \ref{dfn: stab}.
Then define the map
%%(8.12)%%
\begin{align}
%%%%%%
s_{D}^{\Sigma}&:\QQ_{D,n}(\C) \to \QQ_{D+\e,n}(\C)
\qquad
\mbox{ by}
%%%
%\end{equation}
%%%
%%()%%
%\begin{equation}\label{eq: stab-sigma}
\\
\nonumber
s_{D}^{\Sigma}(\xi_1,\cdots ,\xi_r)
&=
(\varphi_D(\xi_1)+x_{D,1},\cdots
,\varphi_D(\xi_r)+x_{D,r})
%%%%%%%%
\end{align}
%%%%%%%%
for
$(\xi_1,\cdots ,\xi_r)\in \QQ_{D,n}$, where we write
$\varphi_D(\xi)=\sum_{k=1}^sn_k\varphi_D(x_k)$
if $\xi=\sum_{k=1}^sn_kx_k\in \SP^d(\C)$
and $(n_k,x_k)\in \N\times \C$ with $\sum_{k=1}^sn_k=d$.
%%%%
\par
%%%%
Then by using the above homeomorphism (\ref{eq: Pol=E}),
we have the following commutative diagram
%%%
%%%(8.13)%%
\begin{equation}\label{CD: stabilization}
%%%
\begin{CD}
\Q^{D,\Sigma}_n(\C) @>s_{D,D+\e}>> \Q^{D+\e,\Sigma}_n(\C)
\\
@V{\Psi_D}V{\cong}V @V{\Psi_{D+\e}}V{\cong}V
\\
\QQ_{D,n}(\C) @>s_{D}^{\Sigma}>> \QQ_{D+\e,n}(\C)
\end{CD}
\end{equation}
%%%%%%%
%%%%%%%
\par
(iv)
Note that $\QQ_{D,n}(\C)$ is path-connected.
Indeed, for any two points 
$
\xi_0,\xi_1\in \QQ_{D,n}(\C),
$
one can construct
a path $\omega :[0,1]\to \QQ_{D,n}(\C)$ such that
$\omega (i)=\xi_i$ for $i\in\{0,1\}$
by using the string representation 
used in  \cite[\S Appendix]{GKY1}. 
Thus the space $\Q^{D,\Sigma}_n(\C)$ is also path-connected. 
By choosing the path $\omega$ $\Z_2$-equivariant way, one can show that
$\Q^{D,\Sigma}_n(\R)$ is also path-connected if $n\geq 2$ or if $n=1$ and 
$\rmin (\Sigma)\geq 3$ $\Leftrightarrow$ $(n,\rmin(\Sigma))\not=(1,2)$
(see  also the proof of Lemma \ref{lmm: pi1=abelian}
and Remark \ref{rmk: path-connected}).
\qed
}
\end{rmk}
%%%(End of Remark 8.4)%%
%%%
%%%
%%%
%%%
%%%(Definition 8.5)%%%
\begin{dfn}\label{dfn: stability K=C}
%%%%%%%%%%%%%%%%%%%%%
{\rm
%%%
Define the stabilized space $\Q_{D+\infty}^{\Sigma}(\C)$ by
the colimit
%%(8.15)%%
\begin{equation}\label{eq: stabilized space}
%%%%%%%
\Q_{n}^{D+\infty ,\Sigma}(\C)=
\lim_{k\to\infty}
\Q^{D+k\e,\Sigma}_n(\C),
\end{equation}
%%%%
where %we set $s_{D,k}=s_{D+k\e,D+(k+1)\e}$ and 
 the colimit is taken from the family of stabilization maps}
%%
%%(8.16)%%
\begin{equation}
\{s_{D+k\e,D+(k+1)\e}:\Q^{D+k\e,\Sigma}_n(\C)\to 
\Q^{\Sigma,D+(k+1)\e}_n(\C)\}_{k\geq 0}
\qquad\quad
\qed
\end{equation}
\end{dfn}
%%(End of Definition 8.5)%%%%%%
%%%
%%%
%%%
%%%
%%%%%%%%%%%%%%%%%%%%%
\paragraph{8.2 Scanning maps.}
Now we are ready to define the scanning map.
From now on, we identify $\C=\R^2$ in a usual way.
%%
%%%
%%%
%%%
%%%
%%(Definition 8.6)%%
\begin{dfn}
%%%%%%%%%%%%%%%
%(i)
{\rm
For a rectangle $X$ in $\C =\R^2$, let $\sigma X$ denote
the union of the sides of $X$ which are parallel to the $y$-axis, and
for a subspace $Z\subset \C=\R^2$, let $\overline{Z}$ be the closure of $Z$.
From now on, let $I$ denote the interval
$I=[-1,1]$
and
let $0<\epsilon <\frac{1}{1000000}$ be any fixed real number.
\par
For each $x\in\R$, let $V(x)$ be the set defined by
%%(8.17)%%
\begin{eqnarray}
%%%%
V(x)
&=&
\{w\in\C: \vert \mbox{Re}(w)-x\vert  <\epsilon , \vert\mbox{Im}(w)\vert<1 \}
\\
\nonumber
&=&
(x-\epsilon ,x+\epsilon )\times (-1,1),
\end{eqnarray}
%%%%%%%%
%%
and let's identify $I\times I=I^2$ with the closed unit rectangle
$\{t+s\sqrt{-1}\in \C: -1\leq t,s\leq 1\}$ 
in $\C$.
%%%
\par
For each $D=(d_1,\cdots ,d_r)\in \N^r$,
define {\it the horizontal scanning map}
%%%(8.18)%%
\begin{equation}
sc^{}_{D}:\QQ_{D,n}(\C )\to \Omega \QQ_n (I^2,\partial I\times I)
=\Omega \QQ_n (I^2,\sigma I^2)
\end{equation}
%%%
as follows.
%%
%%%%%%%%%%%%%%%%%%%%
For each $r$-tuple 
$\alpha =(\xi_1,\cdots ,\xi_r)\in \QQ_{D,n}(\C)$
of configurations,
let
$
sc^{}_D(\alpha):\R \to \QQ_n (I^2,\partial I\times I)
=\QQ_n (I^2,\sigma I^2)
$
denote the map given by
\begin{align*}
%%%%%%
\R\ni x
&\mapsto
(\xi_1\cap\overline{V}(x),\cdots ,\xi_r\cap\overline{V}(x))
\in
\QQ_n (\overline{V}(x),\sigma \overline{V}(x))
\cong \QQ_{n}(I^2,\sigma I^2),
\quad
%%%%%
\end{align*}
%for $x\in \R$,
where 
we use the canonical identification
%\begin{equation}
$(\overline{V}(x),\sigma \overline{V}(x))
\cong (I^2,\sigma I^2).$
%\end{equation}
%%%
\par\vspace{2mm}\par
%%%%%
Since $\dis\lim_{x\to\pm \infty}sc^{}_D(\alpha)(x)
=(\emptyset ,\cdots ,\emptyset)$,  
by setting $sc^{}_D(\alpha)(\infty)=(\emptyset ,\cdots ,\emptyset)$
we obtain a based map
$sc^{}_{D}(\alpha)\in \Omega \QQ_n(I^2,\sigma I^2),$
where we identify $S^1=\R \cup \infty$ and
we choose the empty configuration
$(\emptyset ,\cdots ,\emptyset)$ as the base-point of
$\QQ_n (I^2,\sigma I^2)$.
%%%
%%
One can show that the following  diagram is homotopy commutative:
%%%%%(8.19)%%%%%
\begin{equation}\label{CD: scanning}
\begin{CD}
\QQ_{D+k\e,n} (\C)@>sc_{D+k\e}>> \Omega \QQ_n (I^2,\sigma I^2)
\\
@V{s_{D+k\e}^{\Sigma}}VV \Vert @.
\\
\QQ_{D+(k+1)\e,n}(\C) @>sc_{D+(k+1)\e}>> \Omega \QQ_n (I^2,\sigma I^2) 
\end{CD}
\end{equation}
%%%%%%%%%%
By using the above diagram
and by identifying
$\Q_n^{D+k\e,\Sigma}(\C)$ with 
$\QQ_{D+k\e,n}(\C)$,
%%%%
we finally obtain {\it the stable horizontal scanning map}
%%(8.20)%%%%%
\begin{equation}\label{eq: stable horizontal}
%%%%%%%%%%%%
S^H
=\lim_{k\to \infty}sc_{D+k\e}
:\Q^{D+\infty,\Sigma}_{n}(\C)
%=\lim_{k\to\infty}\Q_{D+k\e}^*(S^1,\XS)
\to
\Omega
\QQ_n (I^2,\sigma I^2),
%%%%
\end{equation}
%%%%
where
$\Q^{D+\infty,\Sigma}_{n}(\C)$ is defined in
(\ref{eq: stabilized space}).
%%%
}
%%%%%
\end{dfn}
%%(End of Definition 8.6)%%%
%%%
%%%
%%%
%%%
%%%
%%%
%%%
%%%(Scanning map Theorem)%%%%%
%%%(Scanning map Theorem)%%%%%
%%%(Scanning map Theorem)%%%%%
%%(Theorem 8.7)%%%%%%%%%%%
%%(Theorem 8.7)%%%%%%%%%%%
\begin{thm}[\cite{Se}, (cf. \cite{Gu1}, \cite{KY10})]\label{thm: scanning map}
%%%%%%%%%%%%%%%%%%%
The  stable horizontal scanning map
$$
S^H
:\Q^{D+\infty,\Sigma}_{n}(\C)
\stackrel{\simeq}{\longrightarrow}
\Omega
\QQ_n (I^2,\sigma I^2)
$$
is a homotopy equivalence.
\end{thm}
%%%%%%%%%%%%
\begin{proof}
%%%(Proof of Theorem 6.8)%%%
%Since $R_{\Sigma}$ is a homotopy equivalence and $S^H=(\Omega R_{\Sigma})\circ S$,
%it suffices to show that $S$ is a homotopy equivalence.
%\par
The proof is  analogous to the one given in
\cite[Prop. 3.2, Lemma 3.4]{Se} and \cite[Prop. 2]{Gu1}.
However, as it appears to be probably most difficult and least familiar part of the article \cite{Se}, we gave its precise proof in \cite[Theorem 5.6]{KY10}
(see also \cite[Remark 5.8]{KY10}).
\end{proof}
%%%(End of Detailed Proof of Theorem 8.7)%%%
%%%(End of Detailed Proof of Theorem 8.7)%%%
%%%(End of Detailed Proof of Theorem 8.7)%%%
%%%%%%%%%%%%%%%%%%%%%%%
%%%%
%Next, consider the horizontal scanning map for the space
%$\Q^{D,\Sigma}_n(\R)$.
%%%
%%(Definition 8.8)%%
%%%%%%
\begin{dfn}\label{dfn: stability K=R}
%%%%%%%%%%%%%%%%%%%%%
{\rm
%%%
(i)
Define the stabilized space $\Q_{D+\infty}^{\Sigma}(\R)$ by
the colimit
%%(8.21)%%
\begin{equation}\label{eq: stabilized space K=R}
%%%%%%%
\Q_{n}^{D+\infty ,\Sigma}(\R)=
\lim_{k\to\infty}
\Q^{D+k\e,\Sigma}_n(\R),
\end{equation}
%%%%
where  
%we set
%%%%(8.22)%%
%\begin{equation} 
%s^{\R}_{D,k}=s^{\R}_{D+k\e,D+(k+1)\e}
%\end{equation}
%and
the colimit is taken from the family of stabilization maps
%%(8.22)%%%
\begin{equation}
\{s^{\R}_{D+k\e,D+(k+1)\e}:\Q^{D+k\e,\Sigma}_n(\R)\to 
\Q^{\Sigma,D+(k+1)\e}_n(\R)\}_{k\geq 0}.
\end{equation}
%%%
%%
\par
(ii)
Recall the $\Z_2$-action on the space
$\Q^{D,\Sigma}_n(\C)$ induced from the complex conjugation on $\C$.
Then
by using  (\ref{eq: stabZ2}), one can easily see the following:
%%(8.23)%%
\begin{equation}
\Q^{D+\infty}_n(\R)=(\Q^{D+\infty}_n(\C))^{\Z_2}.
\end{equation}
%%%%%%%
Moreover, since
$
s^{\R}_{D,k}=(s_{D,k})^{\Z_2}$
as in Remark \ref{rmk: equiv-remark},
one can define the horizontal scanning map
%%(8.24)%%
\begin{equation}
S^{\Z_2}
=\lim_{k\to \infty}(sc_{D+k\e})^{\Z_2}
:\Q^{D+\infty,\Sigma}_{n}(\R)
%=\lim_{k\to\infty}\Q_{D+k\e}^*(S^1,\XS)
\to
\Omega
\QQ_n (I^2,\sigma I^2)^{\Z_2}
\end{equation}
in a complete similar way as (\ref{eq: stable horizontal}).
\par
Since $\Q^{D,\Sigma}_n(\R)=\Q^{D,\Sigma}_n(\C)^{\Z_2}\subset \Q^{D,\Sigma}_n(\C)$, one can identify
the  space $\Q^{D+\infty,\Sigma}_{n}(\R)$ with the subspace
of
$\Q^{D+\infty,\Sigma}_{n}(\C)$.
By this identification, we can identify
%%(8.25)%%
\begin{equation}
\Q^{D+\infty,\Sigma}_{n}(\R)=\Q^{D+\infty,\Sigma}_{n}(\C)^{\Z_2}
\ \mbox{ and }\ 
S^{\Z_2}=S^H\vert \Q^{D+\infty,\Sigma}_{n}(\R)=(S^H)^{\Z_2}.
\end{equation}
}
%%%%
\end{dfn}
%%(End of Definition 8.8)%%%%%%
%%%
%%%
%%%
%%(Theorem 8.9)%%%%%%%%%%%
\begin{thm}[\cite{Se}, (cf. \cite{Gu1}, \cite{KY10})]\label{thm: scanning map K=R}
%%%%%%%%%%%%%%%%%%%
The  stable horizontal scanning map
$$
(S^H)^{\Z_2}
:\Q^{D+\infty,\Sigma}_{n}(\R)
\stackrel{\simeq}{\longrightarrow}
\Omega
\QQ_n (I^2,\sigma I^2)^{\Z_2}
$$
is a homotopy equivalence if $(n,\rmin(\Sigma))\not=(1,2)$.
\end{thm}
%%%%%%%%%%%%
\begin{proof}
Note that $\Q^{D,\Sigma}_n(\R)$ is simply connected if $(n,\rmin(\Sigma))\not=(1,2)$. 
%(by Lemma \ref{lmm: pi1=abelian}).
%%
%Note that $\pi_1(\Q^{D,\Sigma}_n(\R))$ is an abelian group if
%$(n,\rmin(\Sigma))\not=(1,2)$
%(by Lemma \ref{lmm: pi1=abelian}).
The proof is completely analogous to that of Theorem \ref{thm: scanning map}.
\end{proof}
%%%(End of Theorem 8.9)%%
%%
%%(Corollary 8.10)%%
\begin{crl}\label{crl: scanning map K=R}
%%%
The  stable horizontal scanning map
$$
S^H
:\Q^{D+\infty,\Sigma}_{n}(\C)
\stackrel{\simeq}{\longrightarrow}
\Omega
\QQ_n (I^2,\sigma I^2)
$$
is a $\Z_2$-equivariant homotopy equivalence
if $(n,\rmin(\Sigma))\not=(1,2)$.
\end{crl}
\begin{proof}
%Since $\Q^{D+\infty,\Sigma}_{n}(\R)=\Q^{D+\infty,\Sigma}_{n}(\C)^{\Z_2}$,
The assertion follows from Theorems \ref{thm: scanning map}
and \ref{thm: scanning map K=R}.
\end{proof}
%%(End of proof of Corollary 8.10)%%%
%%%
%%%
%%%
%%%
%%%
%%(End of SECTION 8)%%%%%%
%%(End of SECTION 8)%%%%%%
%%%(End of SECTION 8)%%%%%%
%%%%%%%%%%%%%%%%%%%%%%
%%
%%
%%
%%%
%%%%(SECTION 9)%%%%%%
%%%%(SECTION 9)%%%%
\section{The  stable result}\label{section: stability}
%%%%%%%%%%%%%%%%%%%
In this section we prove the stable theorem
(Theorem \ref{thm: V}).

%and give the proofs of the main results
%(Theorems \ref{thm: I}, \ref{thm: I R}, and Corollary \ref{crl: I})
%by using the stable theorem (Theorem \ref{thm: V}).

%\paragraph{9.1 The stability result.}
%%%%%%%%%%%%%%%%%
%First, we shall prove the stability result.
%%(Definition 9.1)%%%
\begin{dfn}
%%%%
{\rm
From now on, let $\e =(a_1,\cdots,a_r)\in \N^r$ be any fixed an $r$-tuple of positive integers
such that $\sum_{k=1}^ra_k\nn_k={\bf 0}_m$.
\par
Let $D=(d_1,\cdots ,d_r)\in \N^r$ be an $r$-tuple of positive integers
such that
$\sum_{k=1}^rd_k\nn_k={\bf 0}_m.$
Then
it is easy to see that the following two diagram is homotopy commutative
for $\K=\R$ or $\C$:
$$
\begin{CD}
%%(the case K=C)%%%%%%%%%
\Q^{D,\Sigma}_n(\K) @>j_{D,n,\K}>> 
\Omega \mathcal{Z}_{\KS}(\K^n,(\K^n)^*)\simeq
\Omega \mathcal{Z}_{\KS}(D^{d(\K)n},S^{d(\K)n-1})  
\\
@V{s^{\K}_{D,D+\e}}VV  \qquad \Vert @. 
\\
\Q^{D+\textit{\textbf{e}},\Sigma}_n(\K) 
@>j_{D+\textit{\textbf{e}},n,\K}>>
\Omega \mathcal{Z}_{\KS}(\K^n,(\K^n)^*)\simeq
\Omega \mathcal{Z}_{\KS}(D^{d(\K)n},S^{d(\K)n-1})
\end{CD}
$$
where
we set
%%(9.1)%%
\begin{equation} 
s^{\K}_{D,D+\e}=s^{}_{D,D+\e}
\quad
\mbox{ if }\K=\C.
\end{equation}
%%%
%%%
Hence, for $\K=\R$ or $\C$, we obtain the stabilized  map
%%
%%(9.2)%%%%%%%%%
\begin{equation}\label{eq: inclusion stab}
%%%%%%%%%%%%%%%%
j_{D+\infty,n,\K}:
\Q^{D+\infty,\Sigma}_n(\K)
%=\lim_{t\to\infty}\Q^{D+t\e,\Sigma}_n
\to
%\Omega \mathcal{Z}_{\KS}(\C^n,(\C^n)^*)\simeq
\Omega \mathcal{Z}_{\KS}(D^{d(\K)n},S^{d(\K)n-1}),
\end{equation}
%%%
where we set
%%%(9.3)%%
\begin{equation}
%%%%
\dis j_{D+\infty,n,\K}=\lim_{t\to\infty}j_{D+t\e,D+(t+1)\e,\K}.
\end{equation}
%%%%%
}
\end{dfn}
%%(End of definition 9.1)%%%

The main purpose of this section is to prove the following result.
%%%
%%
%%
%%
%%
%%
%%
%%%(Stable Theorem)%%%%
%%%(Theorem 9.2)%%
\begin{thm}\label{thm: V}
%%%
Let $\K=\R$ or $\C$, and let
$D=(d_1,\cdots ,d_r)\in \N^r$ be an $r$-tuple of positive integers
such that
$\sum_{k=1}^rd_k\nn_k={\bf 0}_m.$
Then the stabilized map
$$
j_{D+\infty,n,\K}:
\Q^{D+\infty,\Sigma}_n(\K)
\stackrel{\simeq}{\longrightarrow}
\Omega \mathcal{Z}_{\KS}(D^{d(\K)n},S^{d(\K)n-1})
$$
is a homotopy equivalence.
\end{thm}
%%%%%%(End of Theorem 7.2)%%
%%
%%
%%
%%%%%%%%%%%%%%%%%%%%
%Since the assertion of Theorem \ref{thm: V} is true when $n=1$
%(by \cite[Theorem 8.2]{KY11}), from now on we always assume that $n\geq 2$.
%\par
%Since the proof of the case $\K=\R$ is completely analogous to the case of $\K=\C$,
%from now on we only consider the case $\K=\C$.
Before proving Theorem \ref{thm: V} we need the following definition and lemma.
%%
%%%(Definition 8.3)%%%
\begin{dfn}
%%%
{\rm
Let $\K=\R$ or $\C$ as before.
Now we identify $\C=\R^2$ in a usual way and let us write
$U=\{w\in \C:\vert \mbox{Re}(w) \vert <1,\vert \mbox{Im}(w)\vert <1\}
=(-1,1)\times (-1,1)$ and $I=[-1,1]$.
\par\vspace{1mm}\par
%%%%%%
(i)
For an open set $X\subset \C$, let $F^{\K}_n(X)$ 
denote the space
of $r$-tuples $(f_1(z),\cdots ,f_r(z))\in \K [z]^r$
of (not necessarily monic) polynomials satisfying the following condition
$(*)_{n,\R}$:
\begin{enumerate}
%%(i-1)%%
\item[$(*)_{n,\R}$]
%%%%%%
 For any $\sigma =\{i_1,\cdots ,i_s\}\in I(\KS)$,
the polynomials
$f_{i_1}(z),\cdots ,f_{i_s}(z)$ have no common 
{\it real} roots of multiplicity $\geq n$ in $X$
(i.e. no common 
roots of multiplicity $\geq n$ in $X$).
 \end{enumerate}
%%%%%
%Similarly, let $F^{\K}_n(X)\subset F^{\Sigma,\K}_n(X)$ denote the subspace
%of all $(f_1(z),\cdots ,f_r(z)) \in F^{\Sigma,\K}_n(X)$
%satisfying the following condition $(*)_{n,\C}$:
%\begin{enumerate}
%%%(i-1)%%
%\item[$(*)_{n,\C}$]
%%%%%%%
% For any $\sigma =\{i_1,\cdots ,i_s\}\in I(\KS)$,
%the polynomials
%$f_{i_1}(z),\cdots ,f_{i_s}(z)$ have no common 
%roots of multiplicity $\geq n$ in $X$.
% \end{enumerate}
%%%
%%
\par
%%%%%%%%%%
\par
(ii) Let $ev_{0,\K}:F^{\K}_n(U)\to \mathcal{Z}_{\KS}(\K^n,(\K^n)^*)$ 
denote the map given by evaluation at $0$, i.e.
%%(9.4)%%
\begin{equation}
ev_{0,\K}(f_1(z),\cdots ,f_r(z))=(F_n(f_1)(0),\cdots ,F_n(f_r)(0))
\end{equation}
%%%%
for $(f_1(z),\cdots,f_r(z))\in F^{\K}_n(U)$,
where $F_n(f_i)(z)$ denotes the $n$-tuple of monic polynomials of the same degree $d_i$
given by (\ref{eq: Fnfi}).
\par\vspace{1mm}\par
(iii)
Let $\tilde{F}^{\K}_n(U)\subset F^{\K}_n(U)$ 
%(resp.  $\tilde{F}^{\K}_n(U)\subset F^{\K}_n(U)$)
denote the subspace
%consisting 
of all
$(f_1(z),\cdots ,f_r(z))\in F^{\Sigma,\K}_n(U)$ 
%(resp. $(f_1(z),\cdots ,f_r(z))\in F^{\K}_n(U)$)
such that no $f_i(z)$ is identically
zero.
\par
Let
$ev_{\K}:\tilde{F}^{\K}_n(U)\to \mathcal{Z}_{\KS}(\K^n,(\K^n)^*)$
denote
the map given by the restriction
%%%(9.5)%%
\begin{equation}
%%%
ev_{\K}=ev_{0,\K}\vert \tilde{F}^{\K}_n(U).
\end{equation}
\par
%Similarly, let
%$\widehat{ev}_{\K}:\tilde{F}^{\K}_n(U)\to \mathcal{Z}_{\KS}(\K^n,(\K^n)^*)$
%denote the map given by
%$\widehat{ev}_{\K}=ev_0^{\K}\vert \tilde{F}^{\K}_n(U)$.
%%
It is easy to see that the following two equality holds:
%%(9.6)%%
\begin{equation}\label{eq: evj}
%%%%%%%
%\widehat{ev}_{\K}=ev_{\K}\circ j_n^{\Sigma,\K},
%\quad
ev_{\R}=(ev_{\C})^{\Z_2}.
\end{equation}
%%%%%%%
\par
(iv)
Note that the group $\T^r_{\K}=(\K^*)^r$
acts freely on the space
%$\tilde{F}^{\K}_n(U)$ 
%and 
$\tilde{F}^{\K}_n(U)$
in a natural way, and let
%%(9.7)%%
\begin{equation}\label{eq: qR}
%%%%%%%
%q_{\K}:\tilde{F}^{\K}_n(U)\to \tilde{F}^{\K}_n(U)/\T^r_{\K}
%\ \mbox{ and }\ \ 
p_{\K}:\tilde{F}^{\K}_n(U)\to \tilde{F}^{\K}_n(U)/\T^r_{\K}
\end{equation}
%%%
denote the natural projection, where 
 %$\tilde{F}^{\K}_n(U)/\T^r_{\K}$
%and 
$\tilde{F}^{\K}_n(U)/\T^r_{\K}$
denotes
the  corresponding orbit space.
%%
%\par
%(vi)
%For $\K=\C$, let
%%%(9.9)%%
%\begin{equation}
%u_n^{\C}:\tilde{F}^{\C}_n(U)/\T^r_{\C}\to 
%E^{\Sigma}_n (\overline{U},\partial \overline{U})
%\end{equation}
%%%%
%denote the natural map which
%assigns to an $r$-tuple  $(f_1(z),\cdots ,f_r(z))\in \tilde{F}^{\Sigma}_n(U)$ 
%%of polynomials
%the $r$-tuple of their configurations represented by
%their  roots which lie 
%in $\overline{U}$.
%\par
%(vii)
%Similarly, for $\K=\R$, let us consider the map
%%%(9.10)%%
%\begin{equation}
%u_n^{\R}:\tilde{F}^{\R}_n(U)/\T^r_{\R}\to 
%E^{\Sigma}_n (\overline{U},\partial \overline{U})
%\end{equation}
%%%%
%given by the natural map 
%assigns to an $r$-tuple  $(f_1(z),\cdots ,f_r(z))\in \tilde{F}^{\R}_n(U)$ 
%%of polynomials
%the $r$-tuple of their configurations represented by
%their  roots which lie 
%in $\overline{U}$.
%%%%%%
%Then it is easy to see that its image is contained in
%$E^{\Sigma}_n (\overline{U},\partial \overline{U})^{\Z_2}$, and that
%$\tilde{F}^{\R}_n(U)/\T^r_{\R}=(\tilde{F}^{\C}_n(U)/\T^r_{\C})^{\Z_2}$.
%Thus, we finally obtain the map
%%%(9.11)%%
%\begin{equation}
%u_n^{\R}=(u_n^{\C})^{\Z_2}:\tilde{F}^{\R}_n(U)_{\R}/\T^r_{\R}
%=(\tilde{F}^{\C}_n(U)/\T^r_{\C})^{\Z_2}
%\to
%E^{\Sigma}_n (\overline{U},\partial \overline{U})^{\Z_2}.
%\end{equation}
%%%
%%%%%
}
\end{dfn}
\begin{lmm}\label{lmm: ev-rev}
%%%%%%%
Let $\XS$ be a simply connected non-singular toric variety 
such that 
the condition $($\ref{equ: condition dk}$)^*$ is satisfied.
\par
%%(i)%%
$\I$
If the condition $($\ref{1.4}$)^*$ is satisfied,
the space $\Q^{D+\infty,\Sigma}_n(\C)$ is simply connected.
Similarly,
if the condition $($\ref{1.4}$)^{\dagger}$ is satisfied,
the space $\Q^{D+\infty,\Sigma}_n(\R)$ is simply connected.
\par
%%(iii)%%%
$\II$
The map
$ev_{\K}:\tilde{F}^{\K}_n(U)\stackrel{\simeq}{\longrightarrow}
\mathcal{Z}_{\KS}(\K^n,(\K^n)^*)$
is a 
%$\T^r_{\C}$-equivariant 
homotopy equivalence.
%%%%%
\end{lmm}
%%%
%%%
\begin{proof}
%%%%%(Proof of Lemma 9.5)%%%%
(i)
The assertion easily follows from Corollary \ref{crl: 1-connected}.
\par
(ii)
%The assertion (ii) follows from \cite[Lemma 6.4]{KY12}.
For each $\textbf{\textit{b}}=(b_0,b_1,\cdots ,b_{n-1})\in \K^n$, let
$f_{\textbf{\textit{b}}}(z)\in \K [z]$
denote the polynomial of degree $\leq n$ defined by
%%(9.12)%%
\begin{equation}
f_{\textbf{\textit{b}}}(z)=b_0+\sum_{k=1}^{n-1}\frac{b_k-b_0}{k!}z^k.
\end{equation}
%%%%%%%
Let $i_0:\mathcal{Z}_{\KS}(\K^n,(\K^n)^*)\to F^{\K}_n(U)$ 
be the inclusion map given by
%%(9.13)%%
\begin{equation}
i_0(\textbf{\textit{b}}_1,\cdots ,\textbf{\textit{b}}_r)
=\big(f_{\textbf{\textit{b}}_1}(z),\cdots ,f_{\textbf{\textit{b}}_r}(z)\big)
\end{equation}
%%%
for
$(\textbf{\textit{b}}_1,\cdots ,\textbf{\textit{b}}_r)
\in \mathcal{Z}_{\KS}(\K^n,(\K^n)^*).$ 
Since the degree of each polynomial $f_{\textbf{\textit{b}}_1}(z)$ has at most
$n-1$, it has no root of multiplicity $\geq n$.
Thus, the map $i_0$ is well-defined, and  clearly the equality
$ev_0\circ i_0=\mbox{id}$ holds.
\par
Let $f: F^{\K}_n(U)\times [0,1]\to
F^{\K}_n(U)$ be
the homotopy given by
$$
f((f_1,\cdots ,f_t),t)=(f_{1,t}(z),\cdots ,f_{r,t}(z)),
$$
where
$f_{i,t}(z)=f_i(tz)$.
This gives a homotopy between the map
$i_0\circ ev_{0,\K}$ and the identity map, and this proves that 
the map
$$
ev_{0,\K}:F^{\K}_n(U)
\stackrel{\simeq}{\longrightarrow}
\mathcal{Z}_{\KS}(\K^n,(\K^n)^*)
$$ is a deformation retraction.
Since $F^{\K}_n(U)$ is an infinite dimensional manifold and
$\tilde{F}^{\K}_n(U)$ is a closed submanifold of $F^{\K}_n(U)$ of infinite codimension,
it follows from \cite[Theorem 2]{EK} that
the inclusion
%%%(9.10)%%
\begin{equation}
i^{\Sigma,\K}_n:\tilde{F}^{\K}_n(U)
\stackrel{\simeq}{\longrightarrow} F^{\K}_n(U)
\end{equation}
%%%%%%
 is a homotopy equivalence.
Hence the restriction 
$ev_{\K}=ev_{0,\K}\circ i^{\Sigma,\K}_n$ is also  
a homotopy equivalence.
\end{proof}
\begin{dfn}
{\rm
Note that $(\overline{U},\sigma\overline{U})=(I^2,\sigma I^2)=(I\times I,\partial I\times I)$.
Let 
%%(9.11)%%
\begin{equation}
%%%%%%%%
\begin{cases}
w_n^{\C}:\tilde{F}^{\C}_n(U) \to 
\QQ_n(\overline{U},\sigma \overline{U})
=\QQ_n(I^2,\sigma I^2)
%=\QQ_n(I^2,\partial I\times I)
\\
w_n^{\R}:\tilde{F}^{\R}_n(U) \to 
\QQ_n(\overline{U},\sigma \overline{U})^{\Z_2}
=\QQ_n(I^2,\sigma I^2)^{\Z_2}
%=\QQ_n(I^2,\partial I\times I)
\end{cases}
\end{equation}
%%%
denote the natural maps which
assigns to an $r$-tuple  $(f_1(z),\cdots ,f_r(z))\in \tilde{F}^{\K}_n(U)$
($\K=\C$ or $\R$) 
%of polynomials
the $r$-tuple of their configurations represented by
their real roots which lie 
in 
$\overline{U}=I^2$.
These maps clearly induce the maps
%%(9.12)%%
\begin{equation}
\begin{cases}
v_n^{\C}:\tilde{F}^{\C}_n(U)/\T^r_{\C} \to 
\QQ_n(\overline{U},\sigma \overline{U})
=\QQ_n(I^2,\sigma I^2)
\\
v_n^{\R}:\tilde{F}^{\R}_n(U)/\T^r_{\R} \to 
\QQ_n(\overline{U},\sigma \overline{U})^{Z_2}
=\QQ_n(I^2,\sigma I^2)^{\Z_2}
\end{cases}
\end{equation}
%%%
such that the following diagram is commutative:
%%()%%
\begin{equation*}
\begin{CD}
\tilde{F}^{\C}_n(U) @>w_n^{\C}>> \QQ_n(\overline{U},\sigma \overline{U})
@<<\supset< \QQ_n(\overline{U},\sigma \overline{U})^{\Z_2} 
@<w_n^{\R}<< \tilde{F}^{\R}_n(U)
\\
@V{p_{\R}^{\C}}VV \Vert @. \Vert @. @V{p_{\R}^{\R}}VV
\\
\tilde{F}^{\C}_n(U)/\T^r_{\C} @>v_n^{\C}>> \QQ_n(\overline{U},\sigma \overline{U})
@<<\supset< \QQ_n(\overline{U},\sigma \overline{U})^{\Z_2} 
@<v_n^{\R}<< \tilde{F}^{\R}_n(U)/\T^r_{\R}
\end{CD}
\end{equation*}
%%%
%%
}
%%%
\end{dfn}
%%(End of Definition 9.5)%%
%$$
%============
%$$
%%%(Definition 7.8)%%
%\begin{dfn}
%%%%%
%{\rm
%%%%
%\par
%$${}$$
%(iii)
%Note that $\overline{U}\cap \R=I\times \{0\}$, and
%let
%%%(7.5)%%
%\begin{equation}
%%%%%%
%v_n:\tilde{F}^{\Sigma,\R}_n(U)/\T^r_{\R} \to 
%E^{\Sigma,\R}_n(I,\partial I)=E^{\Sigma}_n (I, \partial I)
%\end{equation}
%%%%%%
%denote the natural map which
%assigns to an $r$-tuple  $(f_1(z),\cdots ,f_r(z))\in \tilde{F}^{\Sigma,\R}_n(U)$ 
%%of polynomials
%the $r$-tuple of their configurations represented by
%their real roots %which lie 
%in 
%$I=[-1,1]$.
%%%%
%%
%\par
%(iv)
%By using this identification
%$(I,\partial I)=(I\times \{0\},\partial I\times \{0\})\subset
%(\overline{U},\partial \overline{U})$,
%we obtain the inclusion map
%%%(7.7)%%
%\begin{equation}\label{eq: map iSigma}
%%%%%%%%
%%%%%
%i_{\Sigma}:
%\ E^{\Sigma,\R}_n(I,\partial I)=E^{\Sigma}_n(I,\partial I)
%\stackrel{\subset}{\longrightarrow} 
%E^{\Sigma}_n(\overline{U},\partial \overline{U}).
%%%%
%\end{equation}
%%%
%}
%%%%%%
%\end{dfn}
%%%
%%
%%%(Lemma 9.6)%%
\begin{lmm}\label{lemma: fiber}
%%%%%%
Any fiber of the map $w_n^{\K}$ is homotopy equivalent to the space
$\T^r_{\K}$.
\end{lmm}
%%%%%
%(Proof of Lemma 9.6)%%
\begin{proof}
%%%%%%%%%%%%%%%%%%%%%
Any fiber of the map $w_n^{\K}$ is homeomorphic to the space
$fib(r)$ consisting of all $r$-tuples
$(f_1(z),\cdots ,f_r(z))\in \K[z]^r$ of $\K$-coefficients polynomials such that
each polynomial $f_i(z)$ has no root in $U$.
It suffices to show that there is a homotopy equivalence
%%(9.13)%%
\begin{equation}\label{eq: fib(r)}
fib(r)\simeq \T^r_{\K}.
\end{equation}
%%%
%Since each element $\textit{\textbf{x}}=(x_1,\cdots ,x_m)\in \T^m_{\K}$ can be
%considered as the $m$-tuple of $\K$-coefficients polynomials of degree $0$,
First define the inclusion map $j_0:\T^r_{\K}\to fib(r)$ by
$j_0(\textit{\textbf{x}})=(x_1,\cdots ,x_r)$
for $\textit{\textbf{x}}=(x_1,\cdots ,x_r)\in \T^r_{\K}$.
%%%
Next,
let $f=(f_1(z),\cdots ,f_r(z))\in fib(r)$ be any element.
Since $0\in U$,
$(f_1(0),\cdots ,f_r(0))\in \T^r_{\K}$.
Hence, one can define the evaluation map
$\epsilon_0:fib(r)\to \T^r_{\K}$
by
$\epsilon_0({\rm f})=(f_1(0),\cdots ,f_r(0))$ for
${\rm f}=(f_1(z),\cdots ,f_r(z))\in fib(r)$.
It is easy to see that $\epsilon_0\circ j_0=\mbox{id}_{\T^r_{\K}}.$
\par
Now consider the map $j_0\circ \epsilon_0$.
Note that if a polynomial $g(z)\in \K [z]$ has a root $\alpha \in \C\setminus U$ and
$0<t\leq 1$, the polynomial $g(tz)$ has a root $\alpha/t\in \C\setminus U$.
Thus, one can define the homotopy
$
F:fib(r)\times [0,1]\to fib(r)
$ by
$F({\rm f},t)=(f_1(tz),\cdots ,f_r(tz))$
for $({\rm f},t)=((f_1(z),\cdots ,f_r(z)),t)\in fib(r)\times [0,1].$
It is easy to see that the map $F$ gives a homotopy between
the maps
$j_0\circ \epsilon_0$ and $\mbox{id}_{fib(r)}$.
Hence, we see that the map 
$\epsilon_0:fib(r)\stackrel{\simeq}{\longrightarrow}\T^r_{\K}$
is a desired homotopy equivalence.
%%%%%
\end{proof}
%%%%(End of Proof of Lemma 9.6)%%%

%%
%%(Lemma 9.7)%%%
\begin{lmm}\label{lmm: the map wn}
%%%
The map
$ 
w_n^{\C}:\tilde{F}^{\C}_n(U) \to 
\QQ_n(\overline{U},\sigma \overline{U})
$ 
is a quasifibration with fiber
$\T^r_{\C}$.
Similarly,
the map
$ 
w_n^{\R}:\tilde{F}^{\R}_n(U) \to 
\QQ_n(\overline{U},\sigma \overline{U})^{\Z_2}
$ 
is a quasifibration with fiber
$\T^r_{\R}$.
\end{lmm}
\begin{proof}
Since the proof is completely analogous, we give the proof only for the map
$w_n^{\C}$.
The assertion may be proved by using the well-known Dold-Thom criterion.
%Let $B=\QQ_n(\overline{U},\sigma \overline{U})$, and
Recall that
 the base space $B=\QQ_n(\overline{U},\sigma \overline{U})$ consists of $r$-tuple of divisors (or configurations)  
 $(\xi_1,\cdots ,\xi_r)$ 
 satisfying the condition
 %%%(\dagger_Sigma)%%%
\begin{enumerate}
\item[$(\dagger)_{\Sigma}$]
The configuration
$(\cap_{k\in \sigma}\xi_k)\cap \R \cap (\overline{U}\setminus \sigma \overline{U})$ does not contains no points of multiplicity $\geq n$
for any $\sigma\in I(\KS)$.
\end{enumerate}
%%%
For each $r$-tuple $(d_1,\cdots ,d_r)\in (\Z_{\geq 0})^r$ of non-negative integers,
we denote by
$B_{\leq d_1,\cdots, \leq d_r}$ the subspace of $B$ consisting of all
$r$-tuples $(\xi_1,\cdots ,\xi_r)\in B$ satisfying the condition
%%(9.14)%%
\begin{equation}
\deg (\xi_k\cap \R \cap (\overline{U}\setminus \sigma \overline{U}))\leq d_k
\quad
\mbox{ for each }1\leq k\leq r. 
\end{equation}
%%%
We filter the base space $B$ by an increasing family of subspaces
$\{B_{\leq d_1,\cdots, \leq d_r}\}$.
It suffices to prove that  each restriction
%%(9.15)%%
\begin{equation}\label{eq: quasi-fib over subspace}
%%%%%%%
w_n\vert w_n^{-1}(B_{\leq d_1,\cdots, \leq d_r}):w_n^{-1}(B_{\leq d_1,\cdots, \leq d_r})\to 
B_{\leq d_1,\cdots, \leq d_r}
\end{equation}
is a quasifibration.
Its proof is essentially  completely analogous to that of \cite[Lemma 5.13]{KY10}
(cf. \cite[Lemmas 3.3, 3.4]{Se}).
The difference  is only the condition which we treated.
In the case of \cite[Lemma 5.13]{KY10}, we consider the $m$-tuple
$(\xi_1,\cdots ,\xi_m)$ of configurations which satisfies 
the condition $(\dagger)_1$,
where
%%%(\dagger_1)%%%
\begin{enumerate}
\item[$(\dagger)_1$]
The configuration
$(\cap_{k=1}^m\xi_k)\cap \R \cap (\overline{U}\setminus \sigma \overline{U})$ does not contains no points of multiplicity $\geq n$.
\end{enumerate}
%%%
On the other hand, in our case, we need to consider $r$-tuple
$(\xi_1,\cdots ,\xi_r)$ of configurations satisfying the condition 
$(\dagger)_{\Sigma}$.
%%%
If we replace by the condition $(\dagger)_{\Sigma}$ in the proof of \cite[Lemma 5.13]{KY10},
we can prove that each restriction (\ref{eq: quasi-fib over subspace}) is a quasifibration  
 by the completely identical similar way.
 So we omit the detail.
\end{proof}
%%(End of proof of Lemma 9.7)%%%
%%
%%
%%%%(Corollary 9.8)%%%%%
\begin{crl}\label{crl: the map vn}
%%%%%%%%%%%%%%%%%
The map
$v_n^{\C}:\tilde{F}^{\C}_n(U)/\T^r_{\C} 
\stackrel{\simeq}{\longrightarrow} 
\QQ_n(\overline{U},\sigma \overline{U})$
is a homotopy equivalence.
Similarly,
the map
$v_n^{\R}:\tilde{F}^{\R}_n(U)/\T^r_{\R} 
\stackrel{\simeq}{\longrightarrow} 
\QQ_n(\overline{U},\sigma \overline{U})^{\Z_2}$
is also a homotopy equivalence.
\end{crl}
%Let $\XS$ be a simply connected non-singular toric variety 
%such that 
%the condition $($\ref{equ: condition dk}.1$)$ is satisfied.
%\par
%%$\I$ 
%%The induced homomorphism
%%$$
%%(i_{\Sigma})_*:\pi_k(E^{\Sigma}_n(I,\partial I))
%%\stackrel{\cong}{\longrightarrow}
%%\pi_k(E^{\Sigma}_n(\overline{U},\partial \overline{U}))
%%$$
%%is an isomorphism for any $k\geq 3$.
%\par
%%%(ii)%%
%$\II$
%The  loop map
%%%
%$
%\Omega v_n:\Omega (\tilde{F}^{\Sigma,\R}_n(U)/\T^r_{\R})  
%\stackrel{\simeq}{\longrightarrow}
%\Omega E^{\Sigma}_n(I,\partial I)
%$
%is a homotopy equivalence. 
%%%%%
%%%
\begin{proof}
Since the proof is completely analogous, we give the proof only for the map
$v_n^{\C}$.
Let $F_{n}$ denote the homotopy fiber of the map $w_n^{\C}$.
It follows from \cite[Lemma 2.1]{CMN} that
there is the following homotopy commutative diagram
%%%(9.27)%%%
\begin{equation}
\begin{CD}
\T^r_{\C} @>>=> \T^r_{\C} @>>> *
\\
\Vert @. @VVV @VVV
\\
\T^r_{\C} @>>> \tilde{F}^{\C}_n(U) @>w_n^{\C}>>
\QQ_n(\overline{U},\sigma\overline{U})
\\
@VVV @V{p_{\C}}VV \Vert @.
\\
F_{n} @>>>
\tilde{F}^{\C}_n(U)/\T^r_{\C}
@>v_n^{\C}>>
\QQ_n(\overline{U},\sigma\overline{U})
\end{CD}
\end{equation}
%%%
where all above vertical and horizontal sequences are fibration sequences.
%%%
By this diagram, we easily see that $F_{n}$ is contractible.
Thus, 
$v_n^{\C}$ is a homotopy equivalence.
%%%%%%%%%%%%%%%%%
\end{proof}
\par\vspace{1mm}\par
%%%%%%%%%%%%%%%%%%%%%%
Now we can give the proof of Theorem \ref{thm: V}.
%%%%%%%%%%%%%%%%%%%%%%%%%
%%(Proof of the stability Theorem)%%%%%%
%%(Proof of Theorem 7.2)%%%%%%%%%%%
%%(Proof of Theorem 7.2)%%%%%%%%%%%
%%(Proof of Theorem 7.2)%%%%%%%%%%%
%%%%%%%%%%%%%%%%%%%%%%%%%
\begin{proof}[Proof of Theorem \ref{thm: V}]
%%%%%%%%%%%%%%%%%%%%%%%%%%
%%%%%
%We would like to show that the map
%%%()%%
%\begin{equation*}
%%%%%
%j_{D+\infty,n}:
%\Q^{D+\infty,\Sigma}_n
%\stackrel{\simeq}{\longrightarrow}
%\Omega \mathcal{Z}_{\KS}(D^{2n},S^{2n-1})
%\leqno{(\dagger)}
%%%
%\end{equation*}
%%%%
%is a homotopy equivalence.
First, we shall prove the assertion for case $\K=\C$.
\par\noindent
It follows from Lemma \ref{lmm: ev-rev} and Lemma \ref{lmm: simplicial complex of KS(n)}
 that
two spaces $\Q^{D+\infty,\Sigma}_n(\C)$ and 
$\Omega \mathcal{Z}_{\KS}(D^{2n},S^{2n-1})$ are simply connected.
Thus, it suffices to prove that the map $j_{D+\infty,n,\C}$ induces an isomorphism
$$
(j_{D+\infty,n,\C})_*:\pi_k(\Q^{D+\infty,\Sigma}_n(\C))
\stackrel{\cong}{\longrightarrow}
\pi_k(\Omega \mathcal{Z}_{\KS}(D^{2n},S^{2n-1}))
\quad
\mbox{for any $k\geq 2$.}
$$
Let us identify $\C =\R^2$ and
let $U=(-1,1)\times (-1,1)$ as before.
Define the scanning map
$scan: \tilde{F}^{\C}_n(\C)\to \Map (\R,\tilde{F}^{\C}_n(U))$
by
%%(9.17)%%
\begin{equation} 
scan (f_1(z),\cdots ,f_r(z))(w)=(f_1(z+w),\cdots ,f_r(z+w))
%%%
\end{equation}
%%%
for $(f_1(z),\cdots ,f_r(z)),w)\in \tilde{F}^{\C}_n(\C)\times \R$, and
%Now
consider the diagram
%%%
$$
\begin{CD}
\tilde{F}^{\C}_n(U) @>ev_{\C}>\simeq>\mathcal{Z}_{\KS}(D^{2n},S^{2n-1})
\\
@V{p_{\C}}VV @. 
\\
\tilde{F}^{\C}_n(U)/\T^r_{\C} 
@>v_n^{\C}>\simeq> 
\QQ_n (\overline{U},\sigma \overline{U})
%%%%%
\end{CD}
$$
%where $p_{\R}$ is a natural projection as in (\ref{eq: qR}).
%where
%$p:\tilde{F}(U)\to \tilde{F}(U)/\GS$
%denotes the natural projection map.
%%%%%%
This induces the commutative diagram below
%%%%%
$$
\begin{CD}
%%%%%
\tilde{F}^{\C}_n(\C) @>scan>> \Map (\R, \tilde{F}^{\C}_n(U)) 
@>(ev_{\C})_{\#}>\simeq> 
\Map (\R,\mathcal{Z}_{\KS}(D^{2n},S^{2n-1}))
\\
@V{p_{\C}}VV @V{(p_{\C})_{\#}}VV @.
\\
\tilde{F}^{\C}_n(\C)/\T^r_{\C} @>scan>>
\Map (\R, \tilde{F}^{\C}_n(U)/\T^r_{\C}) 
@>(v_n^{\C})_{\#}>{\simeq}>
\Map (\R, \QQ_n (\overline{U},\sigma \overline{U}))
%%%%%
\end{CD}
$$
%%%%%
Observe that $\Map (\R,\cdot )$ can be replaced by $\Map^* (S^1,\cdot)$
by extending  from $\R$ to $S^1=\R\cup \infty$
(as base-point preserving maps).
%%%%%%
Thus 
by setting
$$
\begin{cases}
\widehat{j_{D,n,\C}}:\Q^{D,\Sigma}_n(\C)
 \stackrel{\subset}{\longrightarrow} 
 \tilde{F}^{\C}_n(\C) 
\stackrel{scan}{\longrightarrow}
\Map^*(S^1,\tilde{F}^{\C}_n(U))
=\Omega \tilde{F}^{\C}_n(U)
\\
%%%%%
\widehat{j_{D,n,\C}^{\p}}:E^{\Sigma,\R}_{D,n}(\C)
\stackrel{\subset}{\longrightarrow}
\tilde{F}^{\C}_n(\C) \stackrel{scan}{\longrightarrow}
\Map^*(S^1,\tilde{F}^{\C}_n(U)/\T^r_{\C})
=\Omega (\tilde{F}^{\C}_n(U)/\T^r_{\C})
\end{cases}
$$
we obtain the following commutative diagram
%%%%(9.18: The main diagram)%%
\begin{equation}\label{CD: main1}
\begin{CD}
%%%%%%
\Q^{D,\Sigma}_n(\C)
@>\widehat{j_{D,n,\C}}>> \Omega \tilde{F}^{\C}_n(U) 
@>\Omega ev_{\C}>\simeq>
\Omega \mathcal{Z}_{\KS}(D^{2n},S^{2n-1})
\\
@V{\cong}VV @V{\Omega p_{\C}}V{}V @. 
%@.
%%
\\
\QQ_{D,n}(\C)
@>\widehat{j_{D,n,\C}^{\p}}>>
\Omega (\tilde{F}^{\C}_n(U)/\T^r_{\C})
@>\Omega  v_n^{\C}>{\simeq}>
\Omega \QQ_n (\overline{U},\sigma \overline{U})
%%%%%%%
\end{CD}
\end{equation}
%%%%%%%%%%%%%%%%%%%%%%%%%%%%
If we identify $\Q^{D+\infty,\Sigma}_n(\C)$ with the colimit
$\dis \lim_{t\to\infty}\QQ_{D+t\e,n}(\C)$, by replacing
 $D$ by $D+t\textit{\textbf{e}}$
$(t\in \N)$
and letting $t\to\infty$, 
we obtain the following homotopy commutative diagram:
%%
%%
%%%(9.19)%%%
\begin{equation}\label{CD: main diagram}
\begin{CD}
%%%%%%%%%%%%
\Q^{D+\infty,\Sigma}_n (\C)
@>\widehat{j_{D+\infty,n,\C}}>> 
\Omega\tilde{F}^{\C}_n(U) 
@>\Omega ev_{\C}>\simeq>
\Omega \mathcal{Z}_{\KS}(D^{2n},S^{2n-1})
%%%
\\
%%%
 \Vert @.
@V{\Omega p_{\C}}V{}V
@.
%%%%
\\
%%%%
\Q^{D+\infty,\Sigma}_n (\C)
@>\widehat{j_{D+\infty,n,\C}^{\p}}>{}>
\Omega (\tilde{F}^{\C}_n(U)/\T^r_{\C})
@>{\Omega v_n^{\C}}>\simeq> 
\Omega \QQ_n (\overline{U},\sigma \overline{U})
%%%
\end{CD}
\end{equation}
%%%%%%%%%%%%%%
where we set
$\dis \widehat{j_{D+\infty,n,\C}}=\lim_{t\to\infty}\widehat{j_{D+t\e,n,\C}}$
and
$\dis \widehat{j_{D+\infty,n,\C}^{\p}}=
\lim_{t\to\infty}\widehat{j_{D+t\e,n,\C}^{\p}}.$
\newline
Since
$(\Omega ev_{\C})\circ \widehat{j_{D+t\textit{\textbf{e}},n,\C}}
=j_{D+t\textit{\textbf{e}},n,\C}$ and
$(\Omega v_n^{\C})\circ \widehat{j_{D+t\textit{\textbf{e}},n,\C}^{\p}}
=sc_{D+t\textit{\textbf{e}}}$
(by identifying 
$\Q^{D+t\e,\Sigma}_n(\C)$ with the space
$E^{\Sigma}_{D+t\e,n}(\C)$),
%%%%%%%%%%%%%
we also obtain the following two equalities:
%%%%%(9.20)%%
\begin{equation}\label{eq: j=S}
%%%%%%%%%%%
j_{D+\infty,n,\C}=(\Omega ev_{\C})\circ \widehat{j_{D+\infty,n,\C}},
\quad 
S^H=(\Omega v_n^{\C})\circ \widehat{j_{D+\infty,n,\C}^{\p}}. 
%%%%%%%%%%
\end{equation}
%%%%%%%%%%%%%%%%%%
Since the map $ev_{\C}$ is a homotopy equivalence, 
it suffices to prove that the map
%%%()%%
\begin{equation*}\label{eq: dagger}
%%%
\widehat{j_{D+\infty,n,\C}}:
\Q^{D+\infty,\Sigma}_n(\C)
\stackrel{}{\longrightarrow}
\Omega\tilde{F}^{\C}_n(U) 
\leqno{(\dagger\dagger)_{\C}}
%%%%
\end{equation*}
induces an isomorphism on the homotopy group
$\pi_k(\ )$ for any $k\geq 2$.
%%%
\par
%%%
Since
$S^H=(\Omega v_n^{\C})\circ \widehat{j_{D+\infty,n,\C}^{\p}}$ 
and $\Omega v_n^{\C}$ are
homotopy equivalences
(by Theorem \ref{thm: scanning map} and  Corollary \ref{crl: the map vn}),
the map
$\widehat{j_{D+\infty,n,\C}^{\p}}$ is  a
homotopy equivalence.
%%%
Since $p_{\C}$ is a fibration with fiber $\T^r_{\C}$, 
the map $\Omega p_{\C}$ 
induces an isomorphism on the homotopy group $\pi_k(\ )$ for any $k\geq 2$.
Hence, by using the equality
$(\Omega p_{\C})\circ \widehat{j_{D+\infty,n,\C}}=\widehat{j_{D+\infty,n,\C}^{\p}}$
(up to homotopy equivalence), we see that
%(by the diagram (\ref{CD: main diagram})), 
the map
$\widehat{j_{D+\infty,n,\C}}$ is 
induces an isomorphism on the homotopy group $\pi_k(\ )$ for any $k\geq 2$.
This completes the proof for the case $\K=\C$.
%%%
%%%
%%%
%%%
%%%
%%%(The case K=R)%%
\par\vspace{2mm}\par
%%%%%%
Next, consider the case $\K=\R$.
%%%%%%%%%%
Note that the proof is almost identical to the case $\K=\C$.
However, since $\Omega p_{\R}$ is a homotopy equivalence, the proof is easier than that of the case $\K=\C$.
\par
Define the scanning map
$sca: \tilde{F}^{\R}_n(\C)\to \Map (\R,\tilde{F}^{\R}_n(U))$
by
%%(9.21)%%
\begin{equation} 
sca (f_1(z),\cdots ,f_r(z))(w)=(f_1(z+w),\cdots ,f_r(z+w))
%%%
\end{equation}
%%%
for $(f_1(z),\cdots ,f_r(z)),w)\in \tilde{F}^{\R}_n(\C)\times \R$.
%So that $scan_{\R}=scan\vert \tilde{F}^{\Sigma,\R}_n(\C)$.
Now
consider the diagram
%%%
$$
\begin{CD}
\tilde{F}^{\R}_n(U) @>ev_{\R}>\simeq>\mathcal{Z}_{\KS}(D^{n},S^{n-1})
\\
@V{p_{\R}}VV @. 
\\
\tilde{F}^{\R}_n(U)/\T^r_{\R} 
@>v_n^{\R}>\simeq> 
\QQ_n (\overline{U},\sigma \overline{U})^{\Z_2}
%%%%%
\end{CD}
$$
%%%%%%
This induces the commutative diagram below
%%%%%
$$
\begin{CD}
%%%%%
\tilde{F}^{\R}_n(\C) @>sca>> \Map (\R, \tilde{F}^{\R}_n(U)) 
@>(ev_{\R})_{\#}>\simeq> 
\Map (\R,\mathcal{Z}_{\KS}(D^{n},S^{n-1}))
\\
@V{p_{\R}}VV @V{(p_{\R})_{\#}}VV @.
\\
\tilde{F}^{\R}_n(\C)/\T^r_{\R} @>sca>>
\Map (\R, \tilde{F}^{\R}_n(U)/\T^r_{\R}) 
@>(v_n^{\R})_{\#}>{\simeq}>
\Map (\R, \QQ_n (\overline{U},\sigma \overline{U})^{\Z_2})
%%%%%
\end{CD}
$$
%%%%%
Observe that $\Map (\R,\cdot )$ can be replaced by $\Map^* (S^1,\cdot)$
by extending  from $\R$ to $S^1=\R\cup \infty$
(as base-point preserving maps).
%%%%%%
Thus 
by setting
$$
\begin{cases}
\widehat{j_{D,n,\R}}:\Q^{D,\Sigma}_n(\R)
 \stackrel{\subset}{\longrightarrow} 
 \tilde{F}^{\R}_n(\C) 
\stackrel{sca}{\longrightarrow}
%\Map^*(S^1,\tilde{F}^{\Sigma,\R}_n(U))
\Omega \tilde{F}^{\R}_n(U)
\\
%%%%%
\widehat{j_{D,n,\R}^{\p}}:
\QQ_{D,n}(\C)^{\Z_2}
\stackrel{\subset}{\longrightarrow}
\tilde{F}^{\R}_n(\C) \stackrel{sca}{\longrightarrow}
%\Map^*(S^1,\tilde{F}^{\Sigma,\R}_n(U)/\T^r_{\R})
\Omega (\tilde{F}^{\R}_n(U)/\T^r_{\R})
\end{cases}
$$
we obtain the following commutative diagram
%%%%(9.22: The main diagram)%%
\begin{equation}\label{CD: main2}
\begin{CD}
%%%%%%
\Q^{D,\Sigma}_n(\R)
@>\widehat{j_{D,n,\R}}>> \Omega \tilde{F}^{\R}_n(U) 
@>\Omega ev_{\R}>\simeq>
\Omega \mathcal{Z}_{\KS}(D^{n},S^{n-1})
\\
@V{\cong}VV @V{\Omega p_{\R}}V{\simeq}V @. 
%@.
%%
\\
\QQ_{D,n}(\C)^{\Z_2}
@>\widehat{j_{D,n,\R}^{\p}}>>
\Omega (\tilde{F}^{\R}_n(U)/\T^r_{\R})
@>\Omega  v_n^{\R}>{\simeq}>
\Omega \QQ_n (\overline{U},\sigma \overline{U})^{\Z_2}
%%%%%%%
\end{CD}
\end{equation}
%%%%%%%%%%%%%%%%%%%%%%%%%%%%
If we identify $\Q^{D+\infty,\Sigma}_n(\R)$ with the colimit
$\dis \lim_{t\to\infty}\QQ_{D+t\e,n}(\C)^{\Z_2}$, by replacing
 $D$ by $D+t\textit{\textbf{e}}$
$(t\in \N)$
and letting $t\to\infty$, 
we obtain the following homotopy commutative diagram:
%%
%%
%%%(9.23)%%%
\begin{equation}\label{CD: main diagram K=R}
\begin{CD}
%%%%%%%%%%%%
\Q^{D+\infty,\Sigma}_n (\R)
@>\widehat{j_{D+\infty,n,\R}}>> 
\Omega\tilde{F}^{\R}_n(U) 
@>\Omega ev_{\R}>\simeq>
\Omega \mathcal{Z}_{\KS}(D^{n},S^{n-1})
%%%
\\
%%%
 \Vert @.
@V{\Omega p_{\R}}V{\simeq}V
@.
%%%%
\\
%%%%
\Q^{D+\infty,\Sigma}_n (\R)
@>\widehat{j_{D+\infty,n,\R}^{\p}}>{}>
\Omega (\tilde{F}^{\R}_n(U)/\T^r_{\R})
@>{\Omega v_n^{\R}}>\simeq> 
\Omega \QQ_n (\overline{U},\sigma \overline{U})^{\Z_2}
%%%
\end{CD}
\end{equation}
%%%%%%%%%%%%%%
where we set
$\dis \widehat{j_{D+\infty,n,\R}}=\lim_{t\to\infty}\widehat{j_{D+t\e,n,\R}}$
and
$\dis \widehat{j_{D+\infty,n,\R}^{\p}}=
\lim_{t\to\infty}\widehat{j_{D+t\e,n,\R}^{\p}}.$
\newline
Since
$(\Omega ev_{\R})\circ \widehat{j_{D+t\textit{\textbf{e}},n,\R}}
=j_{D+t\textit{\textbf{e}},n,\R}$ and
$(\Omega v_n^{\R})\circ \widehat{j_{D+t\textit{\textbf{e}},n,\R}^{\p}}
=(sc_{D+t\textit{\textbf{e}}})^{\Z_2}$,
%(up to homotopy equivalence),
%%%%%%%%%%%%%
we also obtain the following two equalities:
%%%%%(9.24)%%
\begin{equation}\label{eq: j=S K=R}
%%%%%%%%%%%
j_{D+\infty,n,\R}=(\Omega ev_{\R})\circ \widehat{j_{D+\infty,n,\R}},
\quad 
(S^H)^{\Z_2}=(\Omega v_n^{\R})\circ \widehat{j_{D+\infty,n,\R}^{\p}}. 
%%%%%%%%%%
\end{equation}
%%%%%%%%%%%%%%%%%%
Since the map $ev_{\R}$ is a homotopy equivalence, 
it suffices to prove that the map
%%%()%%
\begin{equation*}\label{eq: dagger}
%%%
\widehat{j_{D+\infty,n,\R}}:
\Q^{D+\infty,\Sigma}_n(\C)
\stackrel{}{\longrightarrow}
\Omega\tilde{F}^{\R}_n(U) 
\leqno{(\dagger\dagger)_{\R}}
%%%%
\end{equation*}
is a homotopy equivalence.
%induces an isomorphism on the homotopy group
%$\pi_k(\ )$ for any $k\geq 2$.
%%%
\par
%%%
Since
$(S^H)^{\Z_2}=(\Omega v_n^{\R})\circ \widehat{j_{D+\infty,n,\R}^{\p}}$ 
and $\Omega v_n^{\R}$ are
homotopy equivalences
(by Theorem \ref{thm: scanning map K=R} and  Corollary \ref{crl: the map vn}),
the map
$\widehat{j_{D+\infty,n,\C}^{\p}}$ is  a
homotopy equivalence.
On the other hand,
since $p_{\R}$ is a covering projection with fiber $(\Z_2)^r$,
the map $\Omega p_{\R}$ is a homotopy equivalence.
Hence,
by using the diagram (\ref{CD: main diagram K=R}), we see that
the map $\widehat{j_{D+\infty,n,\R}}$ is a homotopy equivalence.
This completes the proof of Theorem \ref{thm: V}.
%%%%%%%%%%%%%%%%%%%%%%%%
\end{proof}
%%(End of Proof of Theorem 9.2)%%%
%%%%%%%
%%
%%
%%(End of SECTION 9)%%%%%%%
%%(End of SECTION 9)%%%%%%%
%%(End of SECTION 9)%%%%%%%
%%(End of SECTION 9)%%%%%%%
%%(End of SECTION 9)%%%%%%%
%%
%%
%%
%%
%%(SECTION 10)%%%%
%%(SECTION 10)%%%%
%%(SECTION 10)%%%%
%%(SECTION 10)%%%%
%%
%%
%%
\section{Proofs of the main results}\label{section: proofs of main results}
%%
%\paragraph{10.1 Proofs of the main results.}
Now  we give the proofs of the main results
(Theorems \ref{thm: I}, \ref{thm: I R}, and Corollary \ref{crl: I}).
%%
%%%
%%%%%%
%First, we give the proofs of the main results.
%%%%%%%%%%%%%%%%%%%%%%%%%%%%%%%%
%%%(Proof of Theorem 2.14)%%
%%(Proof of Theorem 2.14)%%%
%%(Proof of Theorem 2.14)%%%
%%%%%%%%%%%%%%%%%%%%%%%%%%%%%%%%
\begin{proof}[Proofs of Theorem \ref{thm: I}]
%%%%%%%%%%%%%%%%%%%%%%%%%%%%%%%%
%%(i)%%%
(i)
Suppose that $\sum_{k=1}^rd_k\textbf{\textit{n}}_k={\bf 0}_n$.
Then the assertion (i) easily follows from Corollary \ref{crl: III+1connected} and
Theorem \ref{thm: V}.
\par
%%(ii)%%%
(ii)
Next assume that
$\sum_{k=1}^rd_k\textbf{\textit{n}}_k\not= {\bf 0}_n$.
Recall from  $($\ref{equ: condition dk}$)^*$ that
there is an $r$-tuple
$D_*=(d_1^*,\cdots ,d_r^*)\in \N^r$ such that
$\sum_{k=1}^rd_k^*\textbf{\textit{n}}_k= {\bf 0}_n$.
If we choose a sufficiently large integer $m_0\in \N$, then 
the condition $d_k<m_0d_k^*$ holds for each $1\leq k\leq r$.
%%%%
Then consider the map
$j_{D,n,\C}:\Q^{D,\Sigma}_n(\C)\to \Omega \mathcal{Z}_{\KS}(D^{2n},S^{2n-1})$
defined by 
%%(10.1)%%
\begin{equation}\label{eq: def of jDnC}
%%%%%%%
j_{D,n,\C}=j_{D_0,n,\C}\circ s_{D,D_0},
\end{equation}
%%% 
where
$D_0=m_0D_*=(m_0d^*_1,m_0d^*_2,\cdots ,m_0d^*_r)$
and $j_{D,n,\C}$ is given by the composite of the following maps
%%%
%%(10.2)%%
\begin{equation}
%%%
%j_{D,n,\C}=j_{D_0,n,\C}\circ s_{D,D_0}:
j_{D,n,\C}:
\Q^{D,\Sigma}_n(\C)
\stackrel{s_{D,D_0}}{\longrightarrow}
\Q^{D_0,\Sigma}_n(\C)
\stackrel{j_{D_0,n,\C}}{\longrightarrow}
\Omega \mathcal{Z}_{\KS}(D^{2n},S^{2n-1}).
\end{equation}
%%%%
%%%%%
Since  the maps $s_{D,D_0}$ and
$j_{D_0,\C}$ are  homotopy equivalences through dimensions
$d(D;\Sigma,n,\C)$ and
$d(D_0;\Sigma,n,\C)$, respectively
(by Corollary  \ref{crl: III+1connected} and Theorem \ref{thm: I}),
by using
$d(D;\Sigma,n,\C)\leq d(D_0;\Sigma,n,\C)$
the map $j_{D,n,\C}$
 is a homotopy equivalence through dimension
$d(D;\Sigma,n,\C)$.
%%%%%%%%%%%%%%%%
\end{proof}
%%%%(End of Theorem 2.14)%%%%%%
%%%
%%%(Proof of theorem 2.15)%%
\begin{proof}[Proof of Theorem \ref{thm: I R}]
%%%
(i)
Suppose that $\sum_{k=1}^rd_k\textbf{\textit{n}}_k={\bf 0}_n$.
Then the assertion (i) easily follows from Corollary \ref{crl: III+1connected} and
Theorem \ref{thm: V}.
\par
(ii) This is proved completely analogous way as that of (ii) of Theorem \ref{thm: I}.
Indeed, under the same assumption as (ii) of the proof of Theorem \ref{thm: I}, 
we define the map
$j_{D,n,\R}:\Q^{D,\Sigma}_n(\R)\to \Omega \mathcal{Z}_{\KS}(D^n,S^{n-1})$ by
%%(10.3)%%
\begin{equation}\label{eq: def of jDnR}
j_{D,n,\R}=j_{D_0,n,\R}\circ s_{D,D_0}^{\R}.
\end{equation}
%%%
Since
$d(D;\Sigma,n,\R)\leq d(D_0;\Sigma,n,\R)$,
 it is easy to see that this map is a homotopy equivalence through dimension
$d(D;\Sigma,n,\R)$.
%%%
\end{proof}
%%%(End of proof of Theorem 2.15)%%

%%%
%%%
%%(Proof of Corollary 2.16)%%%%%%%%
\begin{proof}[Proof of Corollary \ref{crl: I}.]
%%%%%%%%%%%%%%%%%%%%%%%
%The assertion (i) easily follows from Lemma \ref{lmm: toric}.
%%
Consider the map of composite
$$
\Omega \mathcal{Z}_{\KS}(D^{2n},S^{2n-1})
\stackrel{\simeq}{\longrightarrow}
\Omega \mathcal{Z}_{\KS}(\C^n,(\C^n)^*)
\stackrel{\Omega q_{n,\C}}{\longrightarrow}
 \Omega \XS (n).
 $$
 Since $\Omega q_{n,\C}$
is a universal covering
(by Corollary \ref{crl: universal covering}),
the assertions easily follow from  Theorem \ref{thm: I}.
%%%
\end{proof}

\par\vspace{1mm}\par
\noindent{\bf Acknowledgements. }
%%%%%%%%%
%The authors should like to take this opportunity to thank
%Professors Martin Guest 
%for his many valuable  insights and suggestions concerning 
%scanning maps.
%%%
%\par
The second author was supported by 
JSPS KAKENHI Grant Number JP22K03283.
This work was also supported by the Research Institute of Mathematical
Sciences, a Joint Usage/Research Center located in Kyoto University.

%%%%%%%
%\par\vspace{1mm}\par
%\par\noindent{\bf Acknowledgements. }
%%%%%%%%%
%The second author was supported by 
%JSPS KAKENHI Grant Number  18K03295. 
%This work was also supported by the Research Institute for Mathematical Sciences, 
%a Joint Usage/Research Center located in Kyoto University.

%% The Appendices part is started with the command \appendix;
%% appendix sections are then done as normal sections
%% \appendix

%% \section{}
%% \label{}

%% References
%%
%% Following citation commands can be used in the body text:
%% Usage of \cite is as follows:
%%   \cite{key}          ==>>  [#]
%%   \cite[chap. 2]{key} ==>>  [#, chap. 2]
%%   \citet{key}         ==>>  Author [#]

%% References with bibTeX database:

%\bibliographystyle{model1-num-names}
%\bibliography{<your-bib-database>}

%% Authors are advised to submit their bibtex database files. They are
%% requested to list a bibtex style file in the manuscript if they do
%% not want to use model1-num-names.bst.

%% References without bibTeX database:

%%%%%%%%%%%%%%%%%%%

\end{document}